\documentclass[a4paper, 12pt]{amsart}
\usepackage{fullpage}
\usepackage[utf8]{inputenc}
\usepackage{amsmath,amsthm,amssymb,amscd}
\usepackage{mathtools}
\usepackage{mathrsfs}
\usepackage{tikz-cd}
\usepackage{multicol}
\usepackage{stmaryrd}
\usepackage{dsfont}
\usepackage{cancel}
\usepackage{amsfonts}
\usepackage{yfonts}
\usepackage{esint}
\usepackage{graphicx}
\graphicspath{ {./images/} }
\newsavebox\MBox
\usepackage{hyperref}
\usepackage{nameref}
\usepackage{etoolbox}
\usepackage{comment}
\usepackage{ulem}

\usepackage{qtree}
\usepackage{forest}

\usepackage{enumitem}
\setlist[enumerate,1]{label={(\alph*)}}
\usepackage{chngcntr}
\counterwithin{equation}{section}
\usepackage{comment}

\newtheorem{thm}{Theorem}

\newtheorem{theorem}{Theorem}[section]

\newtheorem{prp}{Proposition}
\newtheorem{proposition}[theorem]{Proposition}
\newtheorem{lemma}[theorem]{Lemma}
\newtheorem{corollary}[theorem]{Corollary}

\theoremstyle{definition}
\newtheorem{definition}[theorem]{Definition}

\newtheorem{notation}[theorem]{Notation}
\newtheorem{terminology}[theorem]{Terminology}

\theoremstyle{remark}
\newtheorem{remark}[theorem]{Remark}
\newtheorem{example}[theorem]{Example}

\newcommand*{\Q}{\mathbb{Q}}
\newcommand*{\R}{\mathbb{R}}
\newcommand*{\Z}{\mathbb{Z}}
\newcommand*{\N}{\mathbb{N}}
\newcommand{\C}{\mathbb{C}}
\newcommand*{\E}{\mathbb{E}}
\newcommand*{\F}{\mathbb{F}}

\newcommand*{\A}{\mathbb{A}}

\newcommand*{\T}{\mathcal{T}}
\newcommand*{\fq}{\mathfrak q}
\newcommand*{\fm}{\mathfrak m}
\newcommand*{\fp}{\mathfrak p}
\newcommand*{\fu}{\mathfrak u}
\renewcommand*{\L}{\mathbb{L}}

\newcommand*{\mS}{\mathcal{S}}

\newcommand*{\mM}{\mathcal{M}}

\newcommand*{\mJ}{\mathcal{J}}
\newcommand*{\mI}{\mathcal{I}}
\newcommand*{\mO}{\mathcal{O}}

\newcommand*{\mW}{\mathcal{W}}

\renewcommand*{\a}{\alpha}
\renewcommand*{\b}{\beta}
\newcommand*{\g}{\gamma}
\renewcommand*{\d}{\delta}
\newcommand*{\D}{\Delta}

\newcommand*{\La}{\Lambda}
\newcommand*{\s}{\sigma}
\renewcommand*{\t}{\tau}

\newcommand{\w}{\omega}
\newcommand{\W}{\Omega}
\newcommand*{\z}{\zeta}
\newcommand*{\emp}{\varnothing}
\newcommand*{\G}{\Gamma}

\newcommand*{\pa}{\partial}
\renewcommand*{\D}{\Delta}
\renewcommand*{\S}{\Sigma}
\renewcommand*{\phi}{\varphi}

\newcommand*{\we}{\wedge}
\renewcommand{\tilde}{\widetilde}

\DeclareMathOperator{\coker}{\text{coker}}

\DeclarePairedDelimiterX\braket[2]{\langle}{\rangle}{#1 \delimsize\vert #2}
\newcommand{\pull}[1]{{#1}^\star}
\newcommand{\expinv}[1]{{#1}^\diamond}

\newcommand{\mdl}[1]{
    \ifnum #1=1
        {\mM_{k_1+1,I}({\b_1})}
        \else
    \ifnum #1=2
        {\mM_{k_2+1,J}({\b_2})}
        \else
    {\mM_{k+1,l}(\b)}\fi\fi
}
\newcommand{\qor}{Q}
\newcommand{\oqb}[1]{
    \ifnum #1=-3
        {\fq^{\b}_{-1,l}}
        \else
    \ifnum #1=-2
        {\fq^{\b_2}_{-1,|J|}}
        \else
    \ifnum #1=-1
        {\fq^{\b_1}_{-1,|I|}}
        \else
    \ifnum #1=1
        {\fq^{\b_1}_{k_1,|I|}}
        \else
    \ifnum #1=0
        {\fq^{\b_0}_{1,0}}
        \else
    \ifnum #1=2
        {\fq^{\b_2}_{k_2,|J|}}
    \else
    {\fq^{\b}_{k,l}}\fi\fi\fi\fi\fi\fi
}

\newcommand{\oq}[1]{
    \ifnum #1=-1
        {\fq_{-1,l}}
        \else
    \ifnum #1=1
        {\fq_{(1:3)+1+(3:3),|I|}}
        \else
    \ifnum #1=0
        {\fq^{\b_0}_{1,0}}
        \else
    \ifnum #1=2
        {\fq_{(2:3),|J|}}
    \else
    {\fq_{k,l}}\fi\fi\fi\fi
}

\newcommand{\om}[1]{
    \ifnum #1=-1
        {\fm^\g_{-1}}
        \else
    \ifnum #1=1
        {\fm^\g_{|(1:3)|+1+|(3:3)|}}
        \else
    \ifnum #1=0
        {\fm^{\g}_{1}}
        \else
    \ifnum #1=2
        {\fm^\g_{|(2:3)|}}
    \else
    {\fm^\g_{k}}\fi\fi\fi\fi
}
\newcommand{\omb}[1]{
    \ifnum #1=-1
        {\fm^{\b,\g}_{-1}}
        \else
    \ifnum #1=1
        {\fm^{\b,\g}_{(1:3)+1+(3:3)}}
        \else
    \ifnum #1=0
        {\fm^{\b,\g}_{1}}
        \else
    \ifnum #1=2
        {\fm^{\b,\g}_{(2:3)}}
    \else
    {\fm^{\b,\g}_{k}}\fi\fi\fi\fi
}
\newcommand{\orientor}[1]{\ensuremath{#1}-orientor}
\newcommand{\eorientor}[1]{\ensuremath{#1}-endo-orientor}
\newcommand{\Otm}{O}
\newcommand{\efield}{\mathbb E}

\newcommand{\target}{{\mathcal T}}

\DeclareMathOperator{\Ima}{Im}
\newcommand{\Id}{\text{Id}}
\newcommand{\bu}{\bullet}
\newcommand{\da}{^\dagger}

\newcommand{\cort}[1]{\mathcal K_{#1}}
\newcommand{\zcort}[1]{\mathbb K_{#1}}
\newcommand{\ort}[1]{\mathcal L_{#1}}
\newcommand{\zort}[1]{\mathbb L_{#1}}
\newcommand{\eort}{\mathcal E}

\newcommand{\rort}{\mathcal{R}}
\newcommand{\lort}{\mathcal{L}}

\tikzset{
  symbol/.style={
    draw=none,
    every to/.append style={
      edge node={node [sloped, allow upside down, auto=false]{$#1$}}}
  }
}

\title{$A_\infty$ relations in orientor calculus on moduli of stable disk maps}

\author[O. Kedar]{Or Kedar}
\address{Institute of Mathematics\\ Hebrew University, Givat Ram\\Jerusalem, 91904, Israel } \email{or.kedar@mail.huji.ac.il}

\author[J. Solomon]{Jake P. Solomon}
\address{Institute of Mathematics\\ Hebrew University, Givat Ram\\Jerusalem, 91904, Israel } \email{jake@math.huji.ac.il}

\begin{document}

\date{Nov. 2022}
\begin{abstract}
Let $L\subset X$ be a not necessarily orientable relatively $Pin$ Lagrangian submanifold in a symplectic manifold $X$. Evaluation maps of moduli spaces of $J$-holomorphic disks with boundary in $L$ may not be relatively orientable. To deal with this problem, we introduce the notion of an orientor. Interactions between the natural operations on orientors are governed by orientor calculus. Orientor calculus gives rise to a family of orientors on moduli spaces of $J$-holomorphic stable disk maps with boundary in $L$ that satisfy natural relations. In a sequel, we use these orientors and relations to construct the Fukaya $A_\infty$ algebra of $L.$ 
\end{abstract}
\maketitle
\pagestyle{plain}
\tableofcontents

\section{Introduction}
\label{introduction section}
\subsection{Overview}
\label{overview section}
Let $L\subset X$ be a not necessarily orientable relatively $Pin^\pm$ Lagrangian submanifold in a symplectic manifold $(X,\w)$. It is of fundamental importance in symplectic geometry to associate to $L$ a Fukaya $A_\infty$ algebra, which encodes the geometry of the moduli space of $J$-holomorphic stable disk maps with boundary in $L$. Previous work in the non-orientable case is limited, and in particular, it applies only over fields of characteristic $2$ or when $L$ is orientable relative to a local system on $X$. See Section~\ref{Previous developments section}. The present paper and its sequel~\cite{non-orientable-A-infty} construct a Fukaya $A_\infty$ algebra for $L$ over the field $\R$ in a way that does not require $L$ to be orientable. 

The main difficulty in our construction is the relative non-orientability of the evaluation maps of moduli spaces of $J$-holomorphic stable disk maps with boundary in $L$.
This leads us to introduce the notion of an orientor of a map between manifolds, to which the present paper is devoted. To use orientors effectively, we first study the properties of pull-back orientors and induced orientors on boundaries and quotients. These properties are then used to construct a family of orientors $\qor^\b_{k,l}$ of the evaluation maps of moduli spaces of $J$-holomorphic stable disk maps with $k+1$ boundary marked points and $l$ interior marked points. The orientors $\qor^\b_{k,l}$ satisfy a series of equations resembling the $A_\infty$ relations and the cyclic property of the structure of maps of an $A_\infty$ algebra. See Theorems~\ref{boundary of q theorem introduction},~\ref{cyclic theorem introduction} and~\ref{boundary of q k=-1 introduction theorem}. The orientors $\qor^\b_{k,l}$ for $\b$ fixed and $k,l,$ varying are related by forgetful maps, as formulated in Theorem~\ref{factorization q through forgetful theorem introduction}. In~\cite{non-orientable-A-infty} we establish relevant properties of pushforward of differential forms along an orientor and use the orientors $\qor^\b_{k,l}$ to construct a cyclic unital $A_\infty$ algebra structure on differential forms on $L$ with values in a local system of rings. We expect that the orientors $\qor^\b_{k,l}$ can also be used to construct $A_\infty$ algebra structures on the Morse or singular co-chains of $L.$ Furthermore, we expect the orientors $\qor^\b_{k,l}$ can be generalized to construct from $L$ an object of the Fukaya category of $X.$ 

Building on the work of~\cite{Elad1,Sara1,Sara2,Sara3}, we plan to use the $A_\infty$ algebra of $L$ to define open Gromov-Witten invariants for $L$ and to study the structure of these invariants. When $L$ is fixed by an anti-symplectic involution, and $\dim L = 2,$ we expect the open Gromov-Witten invariants of $L$ to recover Welschinger's real enumerative invariants~\cite{Welschinger-invariants-lower-bounds}. When $\dim L > 2$ or when $L$ is not fixed by anti-symplectic involution, it appears that $A_\infty$ algebra of $L$ plays an essential role in the definition of invariants.

Lagrangian submanifolds arise naturally as the real points of smooth complex projective varieties that are invariant under complex conjugation. Natural constructions in algebraic geometry, such as blowups and quotients, give rise to non-orientable Lagrangians.  Examples of computations of open Gromov-Witten-Welschinger invariants for non-orientable Lagrangian submanifolds of dimension $2$ appear in~\cite{Horev,Kharlamov-new-logarithmic-equivalence,Kharlamov-logarithmic-equivalence,Kharlamov-enumeration-rational-curves,Kharlamov-logarithmic-asymptotics,Kharlamov-toric-surfaces,Kharlamov-degree-3,Kharlamov-small-non-toric,Kharlamov-degree-2,Kharlamov-degree-2-erratum}.

\subsection{Previous developments}\label{Previous developments section}
In \cite{Fukaya}, a construction of the Fukaya $A_\infty$ algebra structure on a version of singular chains of $L$ with $\Q$ coefficients is provided when $L$ is orientable. 
In~\cite{Fukaya-Spherically-positive}, the construction of the Fukaya $A_\infty$ algebra structure on singular chains is extended to the non-orientable case when $X$ is spherically positive, using coefficients in $\Z/2$. The spherically positive assumption is used to force stable maps with automorphisms into sufficiently high codimension that they do not lead to denominators when pushing-forward chains by the evaluation maps of moduli spaces. Since the order of automorphism groups can be even, such denominators are not allowable when working with $\Z/2$ coefficients.
However, for planned applications to open Gromov-Witten theory, it is essential to work over a field of characteristic zero.

Given a local system of $1$-dimensional vector spaces $\T$ on $X$, it should be possible to construct a version of the Fukaya category in which objects arise from Lagrangian submanifolds $L \subset X$ that are relatively oriented with respect to $\T.$ By relative orientation, we mean an isomorphism from $\T|_L$ to the orientation local system of $L.$ In~\cite{Seidel-Lefschetz}, such a construction is carried out when the first Chern class $c_1(X)$ is $2$-torsion, for a local system $\T$ that arises naturally in the context of gradings. The relative orientation of $L$ with respect to $\T$ forces the Maslov index $\mu : H_2(X,L) \to \Z$ to take on only even values. It follows that the evaluation maps of moduli spaces of $J$-holomorphic disks are relatively orientable~\cite{JakePhD}, so the main difficulties in the construction of the present work do not arise.

In the construction of Floer homology for two orientable Lagrangians $L_1,L_2,$ that intersect cleanly given in~\cite[Section 3.7.5]{Fukaya}, a local system arises when the intersection $L_1 \cap L_2$ is not orientable. The symplectic topology of non-orientable Lagrangians has been studied extensively in~\cite{Audin, Dai-Lag-surfaces,Evans-Klein-nonsqueeze,Givental-Lag-embedding, Nemirovski-Klein-in-space,Nemirovski-Klein-homology-class,Polterovich-surgery, Rezchikov,Shevchishin-Klein,Shevchishin-Smirnov-projective-plane}.

\subsection{Orientors}\label{Orientors section introduction}
Fix a commutative ring $\A$. Throughout the paper, a local system of graded $\A$-modules will be called simply a local system. For a graded $\A$-module $V$ and an orbifold with corners $M$, denote by $\underline V$ the globally constant sheaf over $M$ with fiber $V$. Moreover, write $\tilde M$ for the orientation double cover of $M$, considered as a $\Z/2$ principle bundle. For a map $f:M\to P,$ define
\[
\cort{f}:=\hom_{\Z/2}(\tilde M,f^*\tilde P)\times_{\Z/2}\A.
\] 
In other words, $\cort{f}$ is the relative orientation local system of $f$, concentrated in degree $-\dim M+\dim P$.
Let $Q,K$ be local systems over $M,P,$ respectively. An \orientor{f} $F$ of $Q$ to $K$ is a morphism
\[
F:Q\to \cort{f}\otimes f^*K.
\]
A few examples of orientors are in place. A relative orientation $\mO^f$ of $f$ may be regarded as an \orientor{f} $\phi^{\mO^f}$ of $\underline\A$ to $\underline\A$. See Definition~\ref{orientation as orientor}. Any morphism of local systems over $M$ is an \orientor{\Id_M}.

We list a few important operations on orientors. The precise statements appear in Section~\ref{orientors section}. If $g:P\to N$ is another smooth map, $R$ is a local system over $N$ and $G$ is a \orientor{g} of $K$ to $R$, then $F,G$ may be composed into a \orientor{g\circ f} $G\bu F$ of $Q$ to $R$ as in Definition~\ref{orientor composition}. The composition is associative, as shown in Lemma~\ref{composition of orientors is associative}. If $A,B$ are local systems over $P$, the orientor $F$ extends to an \orientor{f} ${}^AF^B$ of $f^*A\otimes Q\otimes f^*B$ to  $A\otimes Q\otimes B$. 
If $f$ is a relatively oriented map, with relative orientation $\mO^f$, then $G$ may be pulled back to a \orientor{g\circ f} denoted $\expinv{\left(f,\mO^f\right)}G$, given by
\[
\expinv{\left(f,\mO^f\right)}G=G\bu\phi^{\mO^f}.
\]
See Definition \ref{pullback of orientor}.
Moreover,
consider a pullback diagram of orbifolds as follows.
\begin{equation*}\label{generic pullback diagram introduction}
\begin{tikzcd}
M\times_NP\ar[r,"r"]\ar[d,"q"]&P\ar[d,"g"]\\M\ar[r,"f"]&N
\end{tikzcd}    
\end{equation*}
A \orientor{g} $G$ may be pulled back to a \orientor{q} denoted $\expinv{(r/f)}G$ as in Definition~\ref{pullback of orientor by pullback-diagram}. Finally, if $g:P\to M$ and $B\subset \pa_g P\xrightarrow{\iota}P$ is a closed and open subset of the vertical boundary of $P$ relative to $g$, there exists an \orientor{\iota|_B} $\pa_g$,
which induces a notion of boundary of a \orientor{g} $G$, which is the \orientor{g\circ\iota|_B} $\pa G$, given by
\[(-1)^{|G|}G\bu \pa_g.\] See Definition \ref{Boundary of orientor}. 

\subsection{Orientors for moduli spaces of stable maps}\label{orientors for moduli introduction section}
Let $(X,\w)$ be a symplectic manifold and $L\subset X$ be a connected relatively $Pin^\pm$ Lagrangian. Let $\fp$ be a relative $Pin^\pm$ structure on $L$. Denote by $\mu:H_2(X,L;\Z)\to \Z$ the Maslov index. See~\cite{Maslov} for a review on the Maslov index. Denote by $\w:H_2(X,L;\Z)\to \R$ the symplectic volume. Let $J$ be an $\w$-tame complex structure on $X$. Let $\Pi$ be a quotient of $H_2(X,L;\Z)$ by a possibly trivial subgroup contained in the kernel of the map $\mu\oplus \w:H_2(X,L;\Z)\to \Z\oplus \R$. Thus, the homomorphisms $\mu,\w$ descend to $\Pi$. Denote by $\b_0\in \Pi$ the zero element.

In fact, throughout the paper we work in a generalized setting of a symplectic fibration $\pi^X:X\to \W$ over a manifold with corners $\W$ and a relatively $Pin^\pm$ exact Lagrangian subfibration $L \subset X$. The group $\Pi$ is generalized accordingly. See Section~\ref{families target definition section} for details. The freedom of having arbitrary $\W$ allows, among others, a natural construction of pseudoisotopies of $A_\infty$ algebras as well as higher pseudoisotopies. Allowing general $\W$ should also prove useful in the context Lefschetz fibrations as developed, for example, in~\cite{Seidel-Lefschetz}.

In what follows, one may assume that $\W$ is a point. Occasionally, we clarify in parentheses what is meant in the case of general $\W.$
Let $\pi^L:=\pi^X|_L:L\to \W$, and denote by $\lort_{L}$ the relative orientation local system of $\pi^L$ concentrated in degree $-1$. Set
\[
\rort_L:=\bigoplus_{j\in \Z}\lort_L^{\otimes j}.
\]Here, negative powers correspond to the dual local system. Denote by \[m:\rort_L\otimes \rort_L\to~\rort_L,\quad 1_L:\underline \A\to \rort_L\] the tensor product and the inclusion in degree $0$, respectively, which provide $\rort_L$ the structure of a local system of unital graded non-commutative $\A$-algebras. 

Let $(k,l,\b)\in \Z_{\geq-1}\times \Z_{\geq0}\times \Pi$. Let $J$ be an $\w$-tame (vertical) almost complex structure on $X$.
Denote by $\mM_{k+1,l}(\b):=\mM_{k+1,l}(X,L,J,\b)$ the moduli space of (vertical) $J$-holomorphic genus-zero stable maps with boundary, of degree $\b$, with $k+1$ boundary marked points and $l$ interior marked points. For $0\leq i\leq k$, denote by \[evb_i^{(k,l,\b)}:\mM_{k+1,l}(\b)\to L\] the evaluation at the $i$th boundary point. We often omit part or all of the superscripts $(k,l,\b)$ when they are clear from the context. 
By vertical stable map, we mean a map that maps entirely to one fiber of $\pi^X : X \to \W,$ that is, whose composition with $\pi^X : X \to \W$ is constant. So, we obtain a map 
\[\pi^{\mM_{k+1,l}(\b)}:\mM_{k+1,l}(\b)\to \W\] 
that forgets everything except the fiber of $\pi^X : X \to \Omega$ to which a vertical $J$-holomorphic stable map maps. To streamline the exposition, we will assume that $\mdl{3}$ is a smooth orbifold with corners and $evb_0^\b$ is a submersion. These assumptions hold in a range of important examples~\cite[Example 1.5]{Sara1}. Our construction of \orientor{evb_0}s applies to arbitrary symplectic manifolds and Lagrangian submanifolds by the theory of the virtual fundamental class being developed by several authors~\cite{Fu09a,FO19,FO20,HW10,HWZ21} as explained in Section~\ref{families target definition section}. The orientor analogs of the unit and divisor axioms of Gromov-Witten theory given in Theorem~\ref{factorization q through forgetful theorem introduction} require compatibility of the virtual fundamental class with the forgetful map of interior marked points. This has not yet been fully worked out in the Kuranishi structure formalism.

Set
\begin{align*}
ev:=ev^\b:=(evb_1,...,evb_k):&\mM_{k+1,l}(\b)\to \underbrace{L\times_\W\cdots\times_\W L}_{k\text{ times}},\\E:=E^k:=E_L^k:=&ev^*\left(\underset{j=1}{\overset{k}\boxtimes}\rort\right).
\end{align*}
In Section~\ref{rort_L section}, we construct a family of maps covering $evb_0^{(k,l,\b)}$ of $E^k_L$ to~$\rort_L$ of degree $2-k-2l$
\[
\qor_{k,l}^\b:=\qor_{k,l}^{\left(X,L,J;\b\right)}:E_L^k\to \cort{evb_0}\otimes(evb_0)^*\rort_L
\] indexed by
\[(k,l,\b)\in\Big(\Z_{\geq0}\times\Z_{\geq0}\times \Pi\Big)\setminus\Big\{ (0,0,\b_0),(1,0,\b_0),(2,0,\b_0),(0,1,\b_0)\Big\}.\]
In the exceptional cases, the moduli spaces are empty.
These orientors are used in \cite{non-orientable-A-infty} to construct operators $\fm_k$ which give an $A_\infty$ deformation of the differential graded algebra of differential forms on $L$ with values in $\rort_L$ when $\A$ is an $\R$-algebra. Similar arguments should allow the construction of $A_\infty$ algebras deforming the Morse or singular co-chain complex of $L$ with $\rort_L$ local coefficients. The orientors $\qor_{k,l}^\b$ are surjective for $k\geq 1$ and injective for $k=0, 1$.
Moreover, we prove the following theorems.

Let $k\geq-1,l\geq 0,\b\in\Pi$. 
Let $\pa \mM_{k+1,l}(\b)$ denote the boundary and denote by $\iota:\pa \mM_{k+1,l}(\b)\to \mM_{k+1,l}(\b)$ the canonical map. Fix partitions $k_1+k_2=k+1,\b_1+\b_2=\b$ and $I\dot\cup J=[l]$, such that $k_1>0$ if $k\geq 0$. When $k\geq 0$, let $0< i\leq k_1$. When $k=-1$ let $i=0$. Let
$B_{i,k_1,k_2,I,J}\left(\b_1,\b_2\right)\subset\pa \mM_{k+1,l}(\b)$ denote the closure of the locus of two component stable maps, described as follows. One component has degree $\b_1$ and the other component has degree $\b_2$. The first component carries the boundary marked points labeled $0,...,i-1, i+k_2,...,k,$ and the interior marked points labeled by~$I$. The second component carries the boundary marked points labeled $i,...,i+k_2-1$ and the interior marked points labeled by $J$. The two components are joined at the $i$th boundary marked point on the first component and the $0$th boundary marked point on the second. There exists a canonical local diffeomorphism, called the gluing map,
\[
\vartheta_{i,k_1,k_2,\b_1,\b_2,I,J}:\mM_{k_1+1,I}(\b_1)_{evb_i^{\b_1}}\times_{evb_0^{\b_2}}\mM_{k_2+1,J}(\b_2)\to B_{i,k_1,k_2,I,J}(\b_1,\b_2),
\]
which we abbreviate as $\vartheta$. Denote by $\mO^\vartheta_c$ the canonical relative orientation of $\vartheta$ as a local diffeomorphism. Denote by $p_1,p_2$ the projections on the first and second factors of the fiber product, respectively.

\begin{thm}\label{boundary of q theorem introduction}
Let $k,l\geq0$ and $\b\in\Pi.$ Let $k_1,k_2\geq0$ be such that $k_1+k_2=k+1$, let $I\dot\cup J=[l]$ and let $\b_1+\b_2=\b$. Set\[
E_1=\underset{j=1}{\overset{i-1}\boxtimes}\left(evb_j^{\b_1}\right)^*\rort_L,\qquad E_3=\underset{j=i+k_2}{\overset{k}\boxtimes}\left(evb_{j-k_2+1}^{\b_1}\right)^*\rort_L.
\]
The following equation of \orientor{\left(evb_0^{\b}\circ\iota\circ\vartheta\right)}s of $\vartheta^*\iota^*{ev^\b}^*\left(\underset{j=1}{\overset{k}\boxtimes}\rort_L\right)$
to $\rort_L$ holds.
\begin{equation*}
\expinv{\left(\vartheta,\mO^\vartheta_c\right)}(\pa \qor_{k,l}^\b)=(-1)^s\qor_{k_1,I}^{\b_1}\bu {}^{E_1}\left(\expinv{\left(p_2/evb_i^{\b_1}\right)}\qor_{k_2,J}^{\b_2}\right)^{E_3}
\end{equation*}
Here,
\[
s=i+ik_2+k+\d\mu(\b_1)\mu(\b_2),
\]
where $\d\in \{0,1\}$ is $0$ if exactly when $\mathfrak p$ is a relative $Pin^+$ structure.
\end{thm}

Theorem \ref{boundary of q theorem introduction} is proven in Section \ref{proof of q-relations}. The theorem can be visualized in the following diagram.
An arrow corresponding to a map $f$ and an \orientor{f} $F$ is labeled by a vector $\begin{pmatrix}F\\f\end{pmatrix}$.
\[\begin{tikzcd}
	& {\begin{smallmatrix}\left(ev^\beta\circ\iota \circ \vartheta\right)^*\left(\underset{j=1}{\overset{k}\boxtimes}\rort_L\right)\\\downarrow\\\mathcal M_{k_1+1,I}(\beta_1)\times_L\mathcal M_{k_2+1,J}(\beta_2)\end{smallmatrix}} && {\begin{smallmatrix}\left(ev^{\beta_2}\right)^*\left(\underset{j=i}{\overset{i+k_2-1}\boxtimes}\rort_L\right)\\\downarrow\\\mathcal M_{k_2+1,J}(\beta_2)\end{smallmatrix}} \\
	{\begin{smallmatrix}\left(ev^\beta\circ\iota\right)^*\left(\underset{j=1}{\overset{k}\boxtimes}\rort_L\right)\\\downarrow\\B(\beta)\end{smallmatrix}} \\
	& {\begin{smallmatrix}\left(ev^{\beta_1}\right)^*\left(\underset{j=1}{\overset{k_1}\boxtimes}\rort_L\right)\\\downarrow\\\mathcal M_{k_1+1,I}(\beta_1)\end{smallmatrix}} && {\begin{smallmatrix}\rort_L\\\downarrow\\L\end{smallmatrix}} \\
	{\begin{smallmatrix}\left(ev^\beta\right)^*\left(\underset{j=1}{\overset{k}\boxtimes}\rort_L\right)\\\downarrow\\\mathcal M_{k+1,l}(\beta)\end{smallmatrix}} \\
	& {\begin{smallmatrix}\rort_L\\\downarrow\\L\end{smallmatrix}}
	\arrow[swap,"{\begin{pmatrix}\partial_{evb_0}\\\iota\end{pmatrix}}", from=2-1, to=4-1]
	\arrow["{\begin{pmatrix}\phi^{\mO^\vartheta_c}\\\vartheta\end{pmatrix}}"{description}, from=1-2, to=2-1]
	\arrow["{p_2}"{description}, from=1-2, to=1-4]
	\arrow["{\begin{pmatrix}\qor_{k_2,J}^{\beta_2}\\evb_0^{\beta_2}\end{pmatrix}}"', from=1-4, to=3-4]
	\arrow["{\begin{pmatrix}{}^{E_1}\left(\expinv{\left(p_2/evb^{\beta_1}_i\right)}\qor_{k_2,J}^{\beta_2}\right)^{E_3}\\p_1\end{pmatrix}}"{description}, from=1-2, to=3-2]
	\arrow["{evb_i^{\beta_1}}"{description}, from=3-2, to=3-4]
	\arrow["{\begin{pmatrix}\qor_{k_1,I}^{\b_1}\\evb_0^{\beta_1}\end{pmatrix}}"{description}, from=3-2, to=5-2]
	\arrow["{\begin{pmatrix}\qor_{k,l}^{\beta}\\evb_0^{\beta}\end{pmatrix}}"{description}, from=4-1, to=5-2]
\end{tikzcd}\]
\begin{prp}\label{energy zero q theorem introduction}
In case $(k,l,\b)\in\{(2,0,\b_0),(0,1,\b_0)\}$, the map \[evb_0^{\b_0}=\cdots=evb_k^{\b_0}\]
is a diffeomorphism, and we have
\begin{align*}
    \qor_{2,0}^{\b_0}&=\expinv{\left(evb^{\b_0}_0\right)}m,\\
    \qor_{0,1}^{\b_0}&=\expinv{\left(evb^{\b_0}_0\right)}1_L.
\end{align*}
\end{prp}
The proof of Proposition \ref{energy zero q theorem introduction} is immediate from Lemma~\ref{q for bottom cases b0 lemma}.

Extend the indices of the boundary marked points to $\Z$ cyclicly. In particular, $evb_{k+1}=evb_0$.
Let $c_{ij}:(evb_j)^*\rort_L\to (evb_i)^*\rort_L$ be the isomorphism arising from parallel transport along the boundary of the disk in the direction of the orientation induced from the complex orientation. See Definition~\ref{boundary transport} for details.
Let $Fi:=Fi^\b_{k+1,l}:\mM_{k+1,l+1}(\b)\to \mM_{k+1,l}(\b)$ denote the map that forgets the $l+1$st interior marked point, and stabilizes the resulting map. Similarly, let $Fb:=Fb^\b_{k+1,l}:\mM_{k+2,l}(\b)\to \mM_{k+1,l}(\b)$ denote the map that forgets the $k+1$st boundary marked point, and stabilizes the resulting map. The maps $Fi,Fb,$ have canonical relative orientations $\mO^{Fi}, \mO^{Fb}$, respectively. See Definition~\ref{canonical orientations for cyclic and forgetful maps definition} and Remark~\ref{canonical orientations for cyclic and forgetful maps remark}.
\begin{thm}\label{factorization q through forgetful theorem introduction}
Let $k\geq -1,l\geq 0$ and $\b\in \Pi$. The following equation of \orientor{evb_0^{l+1}}s holds.
\[
\qor_{k,l+1}^\b=\expinv{(Fi,\mO^{Fi})}\qor_{k,l}^\b.
\]
Additionally, for $k\geq0$, the following equation of \orientor{evb_0^{k+1}}s of $E^{k+1}$ to $\rort$ holds.
\[
\qor_{k+1,l}^\b=m\bu \left(\expinv{\left(Fb_{k+1,l}\right)}\qor_{k,l}^\b\right)^\rort\bu \,\,{}^{Fb^*E^k}\left(c_{k+2,k+1}\right)
\]
\end{thm}
Theorem~\ref{factorization q through forgetful theorem introduction} is proven in Remark~\ref{forgetful maps on mI implies factorization of q remark}.

Let $\efield_L:=\pi^L_*\rort_L$ the sheaf pushforward along $\pi^L$, considered as a local system over $\W$. We construct a surjective \orientor{L} of $\rort_L$ to $\efield_L$ of degree $1-n$,
\[
\Otm:=\Otm^L:\rort_L\to \cort{L}\otimes \left(\pi^L\right)^*{\efield_L}.
\]
Informally, $\Otm$ splits off one copy of $\lort_L$, shifts it by degree $1-n$ to $\cort{L}$, and maps the remaining tensor products to $\efield_L$ depending on whether they admit a (vertical) section.
See Definition~\ref{Otm definition} for the precise definition. Denote by $\Otm_{odd}$ the map that equals to $\Otm$ on the odd parts of $\rort_L$ and vanishes on the even parts of $\rort_L$. The orientors $\Otm,\Otm_{odd}$ are used to define integration of forms on $L$ with values in $\rort_L$. In particular, $\Otm_{odd}$ is used to define the cyclic structure on the $A_\infty$ algebra in \cite[Section 5]{non-orientable-A-infty}.

Denote by $f:\mM_{k+1,l}(\b)\to \mM_{k+1,l}(\b)$ the map that cyclicly shifts the boundary points $(z_0,...,z_k)$ as follows $z_k\to z_{k-1}\to \cdots\to z_0\to z_k$. The map $f$ is a diffeomorphism. Set
\begin{align*}
    ev^{cyc}:=(evb_1,...,evb_k, evb_0):&\mM_{k+1,l}(\b)\to \underbrace{L\times_\W\cdots\times_\W L}_{k+1\text{ times}},\\
    E^{cyc}:=&\left(ev^{cyc}\right)^*\rort_L^{\boxtimes k+1}.
\end{align*}
Let $T$ be the \orientor{f} of $E^{cyc}$ to $E^{cyc}$ given as follows.
\[
T(a_0\otimes...\otimes a_k\otimes)=(-1)^{|a_0|\cdot\sum_{j=1}^k|\a_j|}\mO^f_c\otimes a_1\otimes \cdots\otimes a_k\otimes a_0,
\]
where $\mO^f_c$ is the canonical relative orientation of $f$ as a local diffeomorphism. See Section~\ref{local diff orientation}.
\begin{thm}\label{cyclic theorem introduction}
The following equation of \orientor{\pi^{\mM_{k+1,l}(\b)}}s of ${E^{cyc}}$ to $\efield_L$ holds.
\[
\Otm_{odd}\bu m\bu\left(\qor_{k,l}^\b\otimes \Id\right)\bu T=(-1)^k\Otm_{odd} \bu m\bu \left(\qor_{k,l}^\b\otimes \Id\right)
\]
\end{thm}
Theorem \ref{cyclic theorem introduction} is proven in Section \ref{proof of the cyclic equation section}. The theorem can be visualized in the following diagram.
\[\begin{tikzcd}
	{\begin{smallmatrix}E^{cyc}\\\downarrow\\\mM_{k+1,l}(\b)\end{smallmatrix}} && {\begin{smallmatrix}E^{cyc}\\\downarrow\\\mM_{k+1,l}(\b)\end{smallmatrix}} \\
	\\
	{\begin{smallmatrix}\rort_L\otimes \rort_L\\\downarrow\\L\end{smallmatrix}} && {\begin{smallmatrix}\rort_L\otimes\rort_L\\\downarrow\\L\end{smallmatrix}} \\
	& {\begin{smallmatrix}\efield_L\\\downarrow\\\W\end{smallmatrix}}
	\arrow["{\begin{pmatrix}T\\f\end{pmatrix}}", from=1-1, to=1-3]
	\arrow["{\begin{pmatrix}\qor_{k,l}^\b\otimes \Id\\evb_0\end{pmatrix}}", from=1-3, to=3-3]
	\arrow["{\begin{pmatrix}\Otm_{odd}\bu m\\\pi^L\end{pmatrix}}", from=3-3, to=4-2]
	\arrow[swap,"{\begin{pmatrix}\qor_{k,l}^\b\otimes \Id\\evb_0\end{pmatrix}}", from=1-1, to=3-1]
	\arrow[swap,"{\begin{pmatrix}\Otm_{odd}\bu m\\\pi^L\end{pmatrix}}", from=3-1, to=4-2]
	\arrow[equal, from=3-1, to=3-3]
\end{tikzcd}\]
In Section \ref{rort_L section}, we construct a family of \orientor{\pi^{\mM_{0,l}(\b)}}s of $\underline \A$ to $\efield_L$ of degree $4-n-2l$
\[
\qor_{-1,l}^\b:\underline\A\to \cort{\pi^{\mM_{0,l}(\b)}}\otimes \left(\pi^{\mM_{0,l}(\b)}\right)^*\efield_L
\]
indexed by 
\[
(l,\b)\in\left(\Z_{\geq0}\times \Pi\right)\setminus\Big\{(0,\b_0),(1,\b_0)\Big\}.
\]
In the exceptional cases, the moduli spaces are empty. In~\cite{non-orientable-A-infty} these orientors are used to construct an element of the Novikov ring $\fm_{-1}$ and the following theorem gives rise to a relation involving $\fm_{-1},$ the cyclic pairing and $\fm_0.$ In the case that $L$ is orientable, the element $\fm_{-1}$ plays an important role in the definition of open Gromov-Witten invariants given in~\cite{Fukaya-3-folds,Sara1}. We plan to use the element $\fm_{-1}$ defined here along with the $A_\infty$ structure to define open Gromov-Witten invariants for non-orientable $L.$

Fix $l\geq 0$ and $\b\in\Pi$. Fix partitions $\b_1+\b_2=\b$ and $I\dot\cup J=[l]$. Recall the boundary component $B_{-1,I,J}(\b_1,\b_2):=B_{0,0,0,I,J}(\b_1,\b_2)\subset \pa \mM_{0,l}(\b)$, and the gluing map $\vartheta:\mM_{1,I}(\b_1)_{evb_0^{\b_1}}\times_{evb_0^{\b_2}}\mM_{1,J}(\b_2)\to B_{-1,I,J}(\b_1,\b_2)$. Let $\mO^\vartheta_c$ denote the canonical relative orientation of $\vartheta$ as a local diffeomorphism.
Let $\pi^{\mM_{0,I}(\b_1)\times_L\mM_{0,J}(\b_2)}$ denote the composition of $\pi^L$ and the canonical map $\mM_{0,I}(\b_1)\times_L\mM_{0,J}(\b_2)\to L$ given by the fiber product.
\begin{thm}\label{boundary of q k=-1 introduction theorem}
Let $l\geq 0$ and $\b\in \Pi$. Let $I\dot\cup J=[l]$ and $\b_1+\b_2=\b$. The following equation of \eorientor{\pi^{\mM_{0,I}(\b_1)\times_L\mM_{0,J}(\b_2)}}s of $\A$ holds.
\[\expinv{\left(\vartheta,\mO^\vartheta_c\right)}\left(\pa \qor_{-1,l}^\b\right)=-\Otm \bu m\bu
\left(\qor_{0,I}^{\b_1}\right)^{\rort_L}\bu \expinv{\left(p_2/evb_0^{\b_1}\right)}\qor_{0,J}^{\b_2}
\]
\end{thm}
Theorem \ref{boundary of q k=-1 introduction theorem} is proven in Section \ref{proof of q-relations k=-1}. The theorem can be visualized in the following diagram.
\[\begin{tikzcd}
	& {\begin{smallmatrix}\underline\A\\\downarrow\\\mathcal M_{1,I}(\beta_1)\times_L\mathcal M_{1,J}(\beta_2)\end{smallmatrix}} && {\begin{smallmatrix}\underline\A\\\downarrow\\\mathcal M_{1,J}(\beta_2)\end{smallmatrix}} \\
	{\begin{smallmatrix}\underline\A\\\downarrow\\B(\beta)\end{smallmatrix}} \\
	& {\begin{smallmatrix}\left(evb_0^{\beta_1}\right)^*\rort_L\\\downarrow\\\mathcal M_{1,I}(\beta_1)\end{smallmatrix}} && {\begin{smallmatrix}\rort_L\\\downarrow\\L\end{smallmatrix}} \\
	{\begin{smallmatrix}\underline\A\\\downarrow\\\mathcal M_{0,l}(\beta)\end{smallmatrix}} & {\begin{smallmatrix}\rort_L\otimes \rort_L\\\downarrow\\L\end{smallmatrix}} \\
	& {\begin{smallmatrix}\efield_L\\\downarrow\\\W\end{smallmatrix}}
	\arrow[swap,"{\begin{pmatrix}\partial_{\pi^M}\\\iota\end{pmatrix}}", from=2-1, to=4-1]
	\arrow["{\begin{pmatrix}\phi^{\mO^\vartheta_c}\\\vartheta\end{pmatrix}}"{description}, from=1-2, to=2-1]
	\arrow["{p_1}"{description}, from=1-2, to=1-4]
	\arrow[swap,"{\begin{pmatrix}\qor_{0,J}^{\beta_2}\\evb_0^{\beta_2}\end{pmatrix}}"', from=1-4, to=3-4]
	\arrow["{\begin{pmatrix}\expinv{\left(p_1/evb^{\beta_1}_i\right)}\qor_{0,J}^{\beta_2}\\p_1\end{pmatrix}}", from=1-2, to=3-2]
	\arrow["{evb_0^{\beta_1}}"{description}, from=3-2, to=3-4]
	\arrow["{\begin{pmatrix}\left(\qor_{0,I}^{\beta_1}\right)^{\rort_L}\\evb_0^{\beta_1}\end{pmatrix}}", from=3-2, to=4-2]
	\arrow["{\begin{pmatrix}O\bu m\\\pi^L\end{pmatrix}}", from=4-2, to=5-2]
	\arrow["{\begin{pmatrix}\qor_{-1,l}^{\beta}\\\pi^{\mM}\end{pmatrix}}"{description}, from=4-1, to=5-2]
\end{tikzcd}\]

\subsection{Base change}
Recall that our results hold for a general symplectic fibration $\pi^X:X\to \W$ and an exact relatively $Pin^\pm$ Lagrangian subfibration, with an appropriate group $\Pi$ to which the degrees of stable maps belong.
Fix a smooth map $\xi:\W'\to \W$. For any submersion $\pi^M:M\to \W$ we abbreviate $\xi^*M:={\W'}_\xi\times_{\pi^M}M$, so we have a pullback diagram as follows.
\begin{equation}\label{generic pullback families diagram introduction}\begin{tikzcd}
\xi^*M\ar[d,swap,"\pi^{\xi^*M}"]\ar[r,"\xi^M"]&M\ar[d,"\pi^M"]\\\W'\ar[r,swap,"\xi"]&\W.
\end{tikzcd}
\end{equation}
It holds that $\pi^{\xi^*X}:\xi^*X\to \W'$ is a symplectic fibration, and $\xi^*L\subset\xi^*X$ is a relatively $Pin^\pm$ exact Lagrangian subfibration. See Corollary~\ref{pullback of exact manifold is exact naturality corollary}.
Moreover, there exist canonical identifications
\[
\xi^*\mM_{k+1,l}(X,L,J,\b)\simeq\mM_{k+1,l}(\xi^*X,\xi^*L,\xi^*J,\b)
\]
that respect the evaluation maps $evb_i,evi_j$ and the maps $Fb,Fi,f$.

According to Definition~\ref{pullback fiber product orientation convention} applied to diagram~\eqref{generic pullback families diagram introduction} with $M=L$, there exists a canonical isomorphism of local systems of $\A$-algebras
\[
\pull{\xi}_{\rort}:\left(\xi^{L}\right)^*\rort_L\to\rort_{\xi^*L}.
\]
See Section~\ref{proof of q-naturality families}. Set
\[\pull\xi_{\efield}:=\pi^L_*\pull\xi_{\rort}:\xi^*\efield_L\to \efield_{\xi^*L}.\]
We think of $\pull\xi_\rort$ and $\pull\xi_\efield$ as \orientor{\Id_{\xi^*L} \text{ and } \Id_{\W'}}s, respectively.
\begin{prp}\label{naturality of Otm families introduction}
With the above notations, the following diagram is commutative.
\[
\begin{tikzcd}
\xi^*\rort_L\ar[rr,"\pull\xi_{\rort}"]\ar[d,swap,"\expinv\xi \Otm^L"]&&\rort_{\xi^*L}\ar[d,"\Otm^{\xi^*L}"]\\
\cort{\pi^{\xi^*L}}\otimes \xi^*\eort_{L}\ar[rr,swap,"1\otimes (\pi^L)^*\pull\xi_{\efield}"]&&\cort{\xi^*L}\otimes \eort_{\xi^*L}
\end{tikzcd}
\]
That is, the following equation of \orientor{\pi^{\xi^*L}}s holds.
\[
\pull\xi_{\efield}\bu \expinv\xi \Otm^L=\Otm^{\xi^*L}\bu \pull\xi_{\rort}.
\]
\end{prp}

Let $k\geq0,l\geq 0$ and $\b\in \Pi$. Abbreviate \[\mM:=\mM_{k+1,l}(X,L,J,\b),\qquad \mM':=\mM_{k+1,l}(\xi^*X,\xi^*L,\xi^*J,\b).\] Let 
\[
\pull{\xi}_E:=\underset{j=1}{\overset{k}\boxtimes} {\left(evb_j^{\b}\right)}^*\pull\xi_\rort:{\xi^{\mM}}^*E^k_L\to E^k_{\xi^*L}.
\]
We think of $\pull\xi_E$ as an \orientor{\Id_{\mM'}}.

\begin{thm}\label{naturality of Q families introduction}
The following diagram is commutative.
\[
\begin{tikzcd}
{\xi^{\mM}}^*E^k_L\ar[rr,"\pull\xi_E"]\ar[d,swap,"\expinv{\xi}\qor^{\left(X,L,J;\b\right)}_{k+1,l}"]&&E^k_{\xi^*L}\ar[d,"\qor^{\left(\xi^*X,\xi^*L,\xi^*J;\b\right)}_{k+1,l}"]\\
\cort{evb_0^{\W'}}\otimes {\xi^{\mM}}^*\left(evb_0^{\W}\right)^*\rort_L\ar[rr,swap,"1\otimes evb_0^*\left(\pull{\xi}_\rort\right)"]&&\cort{evb_0^{\W'}}\otimes\left(evb_0^{\W'}\right)^*\rort_{\xi^*L}
\end{tikzcd}
\]
That is, the following equation of orientors holds.
\[
\pull\xi_\rort\bu\expinv{\xi}\qor^{\left(X,L,J;\b\right)}_{k+1,l}=\qor^{\left(\xi^*X,\xi^*L,\xi^*J;\b\right)}_{k+1,l}\bu \pull\xi_E.
\]
Moreover, the following diagram is commutative.
\[
\begin{tikzcd}
\underline\F\ar[drr,"\qor_{-1,l}^{\left(\xi^*X,\xi^*L,\xi^*J;\b\right)}"]\ar[d,swap,"\expinv{\xi}\qor^{\left(X,L,J;\b\right)}_{-1,l}"]\\\cort{\pi^{\mM'}}\otimes {\pi^{\mM'}}^*\xi^*\efield_L\ar[rr,swap,"1\otimes {\pi^{\mM'}}^*\pull\xi_{\efield}"]
&&\cort{\pi^{\mM'}}\otimes {\pi^{\mM'}}^*\efield_{\xi^*L}
\end{tikzcd}
\]That is, the following equation of orientors holds.
\[
\pull\xi_{\efield}\bu \expinv\xi \qor^{\left(X,L,J;\b\right)}_{-1,l}=\qor^{\left(\xi^*X,\xi^*L,\xi^*J;\b\right)}_{-1,l}
\]
\end{thm}
Theorem~\ref{naturality of Q families introduction} and Proposition~\ref{naturality of Otm families introduction} are proven in Section~\ref{proof of q-naturality families}.

\subsection{Outline}
In Section~\ref{conventions section} we review algebraic notations, orbifolds background and orientation conventions. Section~\ref{orientors complete section} develops the general theory of orientors and orientor calculus. Section~\ref{moduli spaces section} is devoted to the discussion of families of Lagrangian submanifolds in symplectic manifolds and related moduli spaces of stable maps. Sections~\ref{Moduli orientation - Conventions}-\ref{Gluing map signs section} construct basic \orientor{evb_0^\b}s and state relations between them for different triplets $(k,l,\b)$. The proofs of these relations appear in Sections~\ref{conventions for boundary section}-\ref{naturality of J families section}. Section~\ref{orientors of local systems of algebras section} builds on the results of Section~\ref{Moduli orientation} and extends the basic \orientor{evb_0^\b}s to \orientor{evb_0^\b}s which resemble $A_\infty$ operators.

\subsection{Acknowledgments}
The authors are grateful to M.~Abouzaid, E.~Kosloff, P.~Seidel and S.~Tukachinsky, for helpful conversations. The authors were partially funded by ERC starting grant 337560 as well as ISF grants 569/18 and 1127/22.

\section{Conventions}
\label{conventions section} 
\subsection{Notations}
In the following sections we work in the category of orbifolds with corners, indicated by the Latin capital letters $M,N,P,X,Y$, and smooth amps between them, indicated by $f,g, h$ etc. For a comprehensive guide for the category of orbifolds with corners, we recommend \cite{Sara-corners}.
Throughout this paper, we fix a commutative ring $\A$.

\begin{notation}[Abuse of notation in equations of natural numbers]
Let $M,N$ be manifolds and $f:M\to N$ be a smooth map.
Let $Q,S$ be graded local systems over $M$ and let $F:Q\to S$ be a morphism of degree $\deg F$ and let $q\in Q$ be of degree $\deg q$. 
As stated in the introduction, a local system of graded $\A$-modules will be referred to as a local system. A morphism of local systems might be referred to as a map.
Let $\a\in A(M;Q)$ be a differential form. Let $\b$ be a homology class of a symplectic manifold $X$ relative to a Lagrangian $L$. 
In integral expressions (mostly used as exponents of the number $-1$):
\begin{enumerate}
    \item we write $m$ (or $M$) for the dimension of the corresponding capital-letter orbifold $M$;
    \item we write $f$ for  ${\text{rdim}\,f=\dim M-\dim N}$, the relative dimension of $f$;
    \item we write $q$ for $\deg q$ and we write $F$ for $\deg F$;
    \item we write $\a$ for $|\a|$ which is the degree of $\a$;
    \item we write $\b$ for the Maslov Index $\mu(\b)$.
\end{enumerate}
\end{notation}
\subsection{Graded algebra}
Throughout the paper we write $x=_2y$ to denote $x\equiv y\mod 2.$
\begin{definition}[Tensor product]
\label{tensor product}
Let $\A$ be a ring.
Let $A,B,C,D$ be graded $\A$-modules with valuations (or local systems of graded $\A$-modules over an orbifold with corners). Let $F:A\to C, G:B\to D$ be linear maps of degrees $|F|,|G|$. Let $a,b$ be homogeneous elements of $A,B,$ respectively.
\begin{enumerate}
\item The sign $\otimes$ means the completed tensor product with respect to the valuations.
\item The tensor product of differential graded algebras with valuations is again a differential graded algebra with valuation in the standard way.
For \[a\in A,b\in B\] the differential is 
\[
d_{A\otimes B}(a_0\otimes b_0)=(d_A a_0)\otimes b_0+(-1)^{a_0}a_0\otimes d_Bb_0,
\]
and for \[\{a_i\}_{i=1}^\infty\in A,\{b_j\}_{j=1}^\infty\in B,\]
  the valuation is defined as follows\begin{align*}
\\\nu_{A\otimes B}\left(\sum_{i,j}a_i\otimes b_j\right)=\inf_{a_i\otimes b_j\neq 0}\left(\nu(a_i)+\nu(b_j)\right).
\end{align*}

\item The symmetry isomorphism $\tau_{A,B}$ is given by
\[
A\otimes B\overset{\t_{A,B}}\to B\otimes A,\quad a\otimes b\mapsto (-1)^{ab}b\otimes a.
\]
\item \textit{Tensor product of $\A$-algebras:}\\If $A,B$ are graded $\A$-algebras (or local systems of graded $\A$-algebras) with multiplication $(\cdot_A,\cdot_B)$ then the graded $\A$-algebra $A\otimes B:=A\otimes_\A B$ is defined as the graded tensor product, with multiplication
    \[
    (a_1\otimes b_1)\cdot (a_2\otimes b_2)=(-1)^{b_1a_2}(a_1\cdot_A a_2)\otimes (b_1\cdot_B b_2).
    \]
    \item \textit{Functoriality of tensor product:}\\ 
    The tensor product of two maps is given by
    \[
    F\otimes G:A\otimes B\to C\otimes D,\qquad F\otimes G(a\otimes b)=(-1)^{|G|a}Fa\otimes Gb.
    \]
    
\end{enumerate}
\end{definition}

\begin{lemma}
\label{symmetry isomorphism distributivity lemma}
Let $A,B,C$ be graded $\A$-modules.
Then as maps $A\otimes B\otimes C\to B\otimes C\otimes A$ the following equation holds.
\[
\tau_{A,B\otimes C}=(\Id_B\otimes\tau_{A,C})\circ (\tau_{A,B}\otimes \Id)
\]
\end{lemma}
\begin{proof}
Let $a\in A,b\in B,c\in C$ be of homogeneous degree. Then
\[
\tau_{A,B\otimes C}(a\otimes b\otimes c)=(-1)^{a(b+c)}b\otimes c\otimes a,
\]
and
\[
(\Id_B\otimes\tau_{A,C})\circ (\tau_{A,B}\otimes \Id)(a\otimes b\otimes c)=(-1)^{ab}(\Id_B\otimes\tau_{A,C})(b\otimes a\otimes c)=(-1)^{ab+ac}b\otimes c\otimes a.
\]
\end{proof}
\begin{proposition}[Koszul signs]\label{Koszul signs}
With the previous notation, if $F':C\to C',G':D\to D'$ are maps leaving $C,D$ respectively, of degrees $|F'|,|G'|$, then
\begin{align}
(G\otimes F)\circ\tau_{A,B}=(-1)^{|F||G|}\tau_{C,D}\circ(F\otimes G).\\
(F'\otimes G')\circ(F\otimes G)=(-1)^{|F||G'|}(F'\circ F)\otimes (G'\circ G).
\end{align}
\end{proposition}
\begin{proof}
We evaluate at $a,b\in A\otimes B$.
For the first equality, we have
\[
G\otimes F\circ\tau_{A,B}(a\otimes b)=(-1)^{ab}G\otimes F(b\otimes a)=(-1)^{(F+a)b}Gb\otimes Fa,
\]
and
\[
\tau_{C,D}\circ F\otimes  G(a\otimes b)=(-1)^{Ga}\tau_{C,D}Fa\otimes Gb=(-1)^{Ga+(F+a)(G+b)}Gb\otimes Fa.
\]
However, 
\[
(F+a)b+Ga+(F+a)(G+b)=_2FG.
\]
For the second equality,
\[
(F'\otimes G')\circ(F\otimes G)(a\otimes b)=(-1)^{Ga}(F'\otimes G') Fa\otimes Gb=(-1)^{Ga+G'(F+a)}F'Fa\otimes G'Gb,
\]
and
\[
(F'\circ F)\otimes (G'\circ G)(a\otimes b)=(-1)^{(G'+G)a}F'Fa\otimes G'Gb.
\]
However, 
\[
Ga+G'(F+a)+(G'+G)a=_2FG'.
\]
\end{proof}
\begin{definition}
\label{dual graded vector space definition}
Let $A=\bigoplus_{i\in \Z}A_i$ be a graded $\A$-module. \textbf{The dual space $A^\vee$ of $A$} is given by
\[
A^\vee :=\bigoplus_{i\in \Z}A_i^\vee, 
\]where $A_i^\vee$ is the space of linear maps from $A_i$ to $\A$. Denote by $\nu_A:A\otimes A^\vee\to \A$ the pairing $a\otimes a^\vee\to a^\vee(a)$. 
\end{definition}
\begin{remark}
Treating $\a$ as the free module of dimension one over $\A$, $\A^\vee\simeq \A$. The pairing $\A\otimes\A^\vee\to \A$ is an isomorphism.
\end{remark}
\begin{definition}
Let $T,K$ be graded $\A$-modules and let $S$ be a graded $\A$-algebra. Assume that $\mu:S\otimes K\to K$ is a module-structure. We define \textbf{the left (resp. right) $T$-extension of $\mu$} to be a module-structure of $T\otimes K$ (resp. $K\otimes T$) as follows.
\[
{}^T\mu(s\otimes t\otimes k)=(-1)^{st}t\otimes \mu(s\otimes k),
\]
\[
\mu^T(s\otimes k\otimes t)=\mu(s\otimes k)\otimes t.
\]
We further define an $S$-module structure on $K^\vee$ by
\[
\left(\mu^\vee(s\otimes v^\vee)\right)\otimes v=(-1)^{sv^\vee}v^\vee\left(\mu(s\otimes v)\right).
\]
\end{definition}
\begin{definition}
\label{dual graded linear map definition}
Let $F:X\to Y$ be a graded linear map. \textbf{The dual map $F^\vee$ of $F$} is the graded linear map 
\begin{align*}
    F^\vee:Y^\vee&\to X^\vee\\
    \left(F^\vee y^\vee\right)(x)&=(-1)^{|F||x|}y^\vee(Fx).
\end{align*}
\end{definition}
\begin{remark}
\label{dual map diagram definition remark}
Let $F:X\to Y$ be a graded linear map. Then the following diagram is commutative.
\[
\begin{tikzcd}
X\otimes Y^\vee\ar[r,"F\otimes \Id"]\ar[d,"\Id\otimes F^\vee"]&Y\otimes Y^\vee\ar[d,"\nu_Y"]\\X\otimes X^\vee\ar[r,"\nu_X"]&\A
\end{tikzcd}
\]
\end{remark}
\begin{proof}
Use the Koszul signs of Lemma \ref{Koszul signs}.
\end{proof}

For a set $A$, denote the constant map by $\pi^A:A\to *$.
For two sets $A,B$, we denote their product and corresponding projections as follows.
\[
\begin{tikzcd}
A\times B\ar[r,"\pi^{A\times B}_B"]\ar[d,swap,"\pi^{A\times B}_A"]&B\ar[d,"\pi^B"]\\A\ar[r,swap,"\pi^A"]&*
\end{tikzcd}
\]
When it causes no confusion, we might write $\pi_A,\pi_B$ for the projections.

For two lists $B_1=(v_1,...,v_n),B_2=(w_1,...,w_m)$, denote by $B_1\circ B_2$ the concatenation $(v_1,...,v_n,w_1,...,w_m).$

\subsection{Orbifolds with corners}
\label{orbifolds}
We use the definition of orbifolds with corners from~\cite{Sara-corners}. We also use the definitions of smooth maps, strongly smooth maps, boundary and fiber products of orbifolds with corners given there. In particular, for $M$ an orbifold with corners, $\pa M$ is again an orbifold with corners, and it comes with a natural map $\iota_M:\pa M\to M$. In the special case of manifolds with corners, our definition of boundary coincides with~\cite[Definition 2.6]{Joyce}, our smooth maps coincide with weakly smooth maps in~\cite[Definition 2.1(a)]{Joyce-generalization}, and our strongly smooth maps are as in~\cite[Definition 2.1(e)]{Joyce-generalization}, which coincides with smooth maps in~\cite[Definition 3.1]{Joyce}. We say a map of orbifolds is a submersion if it is a strongly smooth submersion in the sense of~\cite{Sara-corners}. In the special case of manifolds with corners, our submersions coincide with submersions in~\cite[Definition 3.2(iv)]{Joyce} and with strongly smooth horizontal submersions in~\cite[Definition 19(a)]{}. For a strongly smooth map $f:M\to N$, we use the notion of vertical boundary $\pa_f M\subset\pa M$ defined in~\cite[Section 2.1.1]{Sara-corners}, which extends to orbifolds with corners the definition of~\cite[Section 4]{Joyce} for manifolds with corners. We write $\iota_f:\pa_fM\to M$ for the restriction of $\iota$ to $\pa_fM$. When $f$ is a submersion, the vertical boundary is the fiberwise boundary, that is, $\pa_fM=\coprod_{y\in N}\pa(f^{-1}(y))$. If $\pa N=\emp$, then $\pa_fM=\pa M$. A strongly smooth map of orbifolds $f:M\to N$ induces a strongly smooth map $f|_{\pa_fM}=f\circ \iota_f:\pa_fM\to N$, called the restriction to the vertical boundary. If $f$ is a submersion, then the restriction $f|_{\pa_fM}$ is also a submersion. As usual, diffeomorphisms are smooth maps with a smooth inverse. We use the notion of transversality from~\cite[Section 3]{Sara-corners}, which is induced from transversality of maps of manifolds with corners as defined in~\cite[Definition 6.1]{Joyce}. In particular, any smooth map is transverse to a submersion. Weak fiber products of strongly smooth transverse maps exist by~\cite[Lemma 5.3]{Sara-corners}. Below, we omit the adjective `weak' for brevity.

Extending the definition of corners of~\cite{Joyce} to corners over a map, we continue with the following definition.
\begin{definition}[Vertical corners]\label{vertical corners definition}
Let $r\geq0$, and write $S_r$ for the symmetric group on $r$ elements, the group of bijections $\s:\{1,...,r\}\to \{1,...,r\}$. Let $f:M\to N$ be a strongly smooth map. By abuse of notation, in the current definition, we also denote by $f$ the restriction of $f$ to $\pa^rM$. There is a natural free action of $S_r$ on $\pa_f^rM$ by permuting the order of local vertical boundary components. We define the \textbf{vertical $r-$corners} $C^r_f(M)$ of $M$ to be
\[
C^r_f(M):=\pa_f^rM/S_r.
\]
\end{definition}
\begin{definition}[Orbifold over a manifold]
Let $\W$ be a manifold with corners. An \textbf{orbifold (resp. manifold) with corners over $\W$} is an orbifold (resp. manifold) with corners $A$ and a submersion $\pi^A:A\to \W$. An \textbf{orbifold (resp. manifold) over $\W$} is an orbifold (resp. manifold) with corners $A$ and a submersion $\pi^A:A\to \W$ with empty vertical boundary $\pa_{\pi^A}A=\emp.$
\end{definition}
\begin{definition}\label{boundary factorization}
Let $M\overset{f}\to P\overset{g}\to N$ be such that $g\circ f$ is a proper submersion. In particular, $g$ is a proper submersion. we say \textbf{$f$ factorizes through the boundary of~$g$} if there exists a map $\iota_g^*f$ such that the following diagram is a fiber-product.
\[
\begin{tikzcd}
\pa_{g\circ f}M\ar[r,"\iota_{g\circ f}"]\ar[d,swap,dashed, "\iota_g^* f"]&M\ar[d,"f"]\\
\pa_g P\ar[r,swap,"\iota_g"]&P
\end{tikzcd}
\]
\end{definition}
\begin{remark}
If $f$ is a (proper) submersion, and $f$ factorizes through the boundary of $g$, then $\iota_g^*f$ is also a (proper) submersion.
\end{remark}
\begin{definition}
Let $M$ be a topological space. A subset $M'\subset M$ is called \textbf{essential} if it is open, dense and for any $x\in M$ there exists a local basis consisting of open sets $U$ such that the inclusion map $\iota:U\cap M'\to U$ induces a bijection on connected components.
\end{definition}
\begin{remark}
If $M$ is locally connected, then the local basis in the above definition may be taken to consist of connected open sets $U$, and thus $U\cap M'$ is connected.
\end{remark}
\begin{example}
Let $M$ be an orbifold with corners and let $D\subset M$ denote either a suborbifold of codimension $\geq 2$, or the image of the inclusion of the boundary $\iota:\pa M\to M$. Then $M\setminus D$ is essential.
\end{example}
\begin{definition}
Consider a square diagram of orbifolds with corners, as follows.
\[
\begin{tikzcd}
S\ar[d,"q"]\ar[r,"r"]&P\ar[d,"g"]\\M\ar[r,"f"]&N
\end{tikzcd}
\]
$S$ is called an \textbf{essential fiber-product of $M,P$ over $N$} if there exist essential subsets $S',M',P',N'$ of $S,M,P,N$, respectively, such that the images of $S',M',P'$ under $q,r,f,g$ are in $M',P',N'$, and the diagram
\[
\begin{tikzcd}
S'\ar[d,"q"]\ar[r,"r"]&P'\ar[d,"g"]\\M'\ar[r,"f"]&N'
\end{tikzcd}
\]
is a fiber-product.
\end{definition}
\begin{definition}
\label{essential boundary factorization}
we say \textbf{$f$ essentially factorizes through the boundary of~$g$} if Definition \ref{boundary factorization} holds with an \textit{essential fiber-product}.
\[
\begin{tikzcd}
\pa_{g\circ f}M\ar[r,"\iota_{g\circ f}"]\ar[d,swap,dashed, "\iota_g^* f"]&M\ar[d,"f"]\\
\pa_g P\ar[r,swap,"\iota_g"]&P
\end{tikzcd}
\]
\end{definition}

We make an extensive use of the following lemma in Section \ref{Moduli orientation}, to allow orientation calculations on essential subsets of the moduli spaces.
\begin{lemma}[Unique extension of $\Z/2$-maps]
\label{unique extension of bundle maps}
Let $M$ be a locally connected topological space, and let $P,Q$ be $\Z/2$-bundles over $M$. Let $M'\subset M$ be an essential subset. Let $f:P|_{M'}\to Q|_{M'}$ be a map of $\Z/2$ bundles over $M'$. Then $f$ extends uniquely to a map $f:P\to Q$ of bundles over $M$.

In particular, if $f,g:P\to Q$ satisfy $f|_{M'}=g|_{M'}$ then $f=g$.
\end{lemma}
\begin{proof}
Let $x\in M\setminus M'$ be a point. There exists a connected neighborhood $x\in U\subset~M$ such that $P|_U,Q|_U$ are trivial and that the intersection $U\cap M'$ is connected. Therefore, under trivializations of $P|_U,Q|_U$, the map $f|_{U\cap M'}$ is constant. Therefore, it extends uniquely to $U$.
\end{proof}
\begin{remark}
By Lemma \ref{unique extension of bundle maps}, in what follows all Lemmas that concern fiber-products or boundary factorization hold for essential fiber-products and essential boundary factorization. \end{remark}
\begin{notation}
generally, for a set $X$ and a topological space $M$, we write $\underline X$ for the trivial bundle over $M$ with fiber $X$.
\end{notation}


\subsection{Orientation conventions}\label{orientations section}
We follow the conventions of \cite{Sara-corners} concerning manifolds with corners. In particular, we relatively orient boundary and fiber products as detailed in the following. For an orbifold with corners $M$, we consider the orientation double cover $\tilde M$ as a graded $\Z/2$-bundle, concentrated in degree $\deg \tilde M=\dim M$.
\begin{definition}
 Let $M\overset f\to N$ be a map. We define the \textbf{relative orientation bundle} of $f$ to be the $\Z/2$-bundle over $M$ given by
\[
\zcort{f}:=  Hom_{\Z/2}(\tilde M,f^*\tilde N).
\]

A \textbf{local relative orientation} is a section $\mO:U\to \zcort{f}|_U$ over an open subset $U\subset M$.
A \textbf{relative orientation} is a global section $\mO:M\to\zcort{f}$.
\end{definition}
Note that it is concentrated in degree $-\text{rdim}\, f=-m+n$.

\begin{definition}
 The \textbf{orientation bundle} of an orbifold $M$ is defined to be the relative orientation bundle of the constant map $M\to pt$,
\[
\zcort{M}:=Hom_{\Z/2}(\tilde M,\underline{\Z/2})=\tilde M^\vee,
\]
A \textbf{(local) orientation} for $M$ is (local) orientation relative to the constant map $M\to pt$.
\end{definition}
Note that it is concentrated in degree $-\dim M$.

We now relatively-orient chosen operations on orbifolds. 
\subsubsection{Local diffeomorphism}
    \label{local diff orientation}
    \begin{definition}
    Let $f:M\to N$ be a local diffeomorphism. The differential $df$ is regarded as a bundle map $df:TM\to f^*TN$. Its exterior power induces a $\Z/2$-bundle map $[\La^{top}df]:\tilde M\to f^*\tilde N$. It can be thought of as a section $\mO^f_c\in Hom(\tilde M,f^*\tilde N)$ called \textbf{the canonical relative orientation of} $f$. In particular, $\zcort{f}$ is canonically trivial.
    
    Moreover, given a map $g:N\to P$, there is a \textbf{pullback map}
    \[
        \pull{f}:f^*\zcort{g}\to \zcort{g\circ f}
    \]
    given by composition on the right with $\mO^f_c$.
\end{definition}

\subsubsection{Composition}
\label{Composition orientation convention}
\begin{definition}
\label{composition isomorphism}
Let $M\overset{f}\longrightarrow P\overset{g}\longrightarrow N$ be two maps. There is a canonical isomorphism
    \[
    \zcort{f}\otimes f^*\zcort{g}\simeq \zcort{g\circ f}
    \] called the \textbf{composition isomorphism}
    given by
    \begin{align*}
    Hom_{\Z/2}(\tilde M,f^*\tilde P)\otimes f^*Hom_{\Z/2}(\tilde P, g^*\tilde N)&\to Hom_{\Z/2}(\tilde M,(g\circ f)^*\tilde N),\\ \mO^f\otimes f^*\mO^g&\mapsto f^*\mO^g\circ\mO^f.
    \end{align*}
\end{definition}
\begin{remark}\label{canonical orientation of local diffeomorphisms composition remark}
If $f,g$ are local diffeomorphisms, then by the chain rule
\[
\mO_c^{g\circ f}=f^*\mO_c^g\circ\mO_c^f.
\]
That is $\pull{f}\circ\pull{g}=\pull{(g\circ f)}.$
\end{remark}
\begin{notation}
By abuse of notation, we may omit the pullback notation $f^*$ if it causes no confusion, such as
\[
\mO^f\otimes \mO^g=\mO^f\otimes f^*\mO^g,\qquad \mO^g\circ\mO^f=f^*\mO^g\circ \mO^f.
\]
\end{notation}
\begin{remark}
\label{pullback by diffeomorphism base independent}
In the setting of Definition \ref{Composition orientation convention}, if $t:T\to M$ is a local diffeomorphism, then the following diagram is commutative.
\[
\begin{tikzcd}
a^*\zcort{f}\otimes t^*f^*\zcort{g}\ar[d,swap,"\pull{t}\otimes\Id"]\ar[r,"comp."]&t^*\zcort{g\circ f}\ar[d,"\pull{t}"]\\
\zcort{f\circ t}\otimes (f\circ t)^*\zcort{g}\ar[r,swap,"comp."]&\zcort{g\circ f\circ t}
\end{tikzcd}
\]
This follows from the associativity of composition.
\end{remark}

\begin{notation}
To ease notations, we notate this isomorphism as equality.
\end{notation}

\subsubsection{Relative orientation of boundary}
\label{relative orientation of boundary}
    Let $M\overset{f}\longrightarrow N$ be a proper submersion. As explained in Section \ref{orbifolds}, the boundary of $M$ can be divided into horizontal and vertical components with respect to $f$. Let $p\in \pa_fM$ be a point in the vertical boundary and $x_1,...,x_{m-1}\in T_p\pa_fM$ be a basis, such that 
    \[
    df_{\iota_f(p)}\circ {(d\iota_f)}_p(x_{i})=0,  \quad i=n+1,...,m-1.
    \]Let $x_1^\vee,...,x_{m-1}^\vee$ be the dual basis. Let $\nu_{out}$ be an outwards-pointing vector in $T_pM$.
    We define the \textbf{canonical relative orientation of the boundary} to be
    \begin{align}
    \label{Canonical relative orientation of boundary}\mO_c^{\iota_f}|_p:=\Big[x_1^\vee\we...&\we x_{m-1}^\vee\bigotimes {(d\iota_f)}_p(x_1)\we...\we {(d\iota_f)}_p(x_n)\we\\ \notag&\we\nu_{out}\we {(d\iota_f)}_p(x_{n+1})\we...\we {(d\iota_f)}_p(x_{m-1})\Big].
    \end{align}
    \begin{remark}\label{boundary orientation comparison}
Consider the constant map $pt:M\to pt$. The total boundary $\pa M$ is the same as $\pa_{pt}M$. Therefore, we get a relative orientation for the inclusion of the total boundary $\iota_M:\pa M\to M$. By virtue of Definition \ref{local diff orientation}, it is possible to relatively orient $\pa_fM$ as an open subset of $\pa M$. That is indeed done in \cite{Sara-corners}: with the previous notations, the relative orientation that is used there is
\[
o_c^{\iota_f}|_p:=\left[x_1^\vee\we...\we x_{m-1}^\vee\bigotimes \nu_{out}\we {(d\iota_f)}_p(x_1)\we...\we {(d\iota_f)}_p(x_{m-1})\right].
\]
These relative orientations satisfy 
\[
o_c^{\iota_f}=(-1)^{N}\mO_c^{\iota_f}.
\]
\end{remark}
\begin{remark}
\label{composition with canonical orientation of boundary}
Given a relative orientation $\mO^f\in \zcort{f}|_p$ at point $p\in M$, for vectors $x_1,...,x_{m-1}\in T_p\pa_fM$ that form a basis, we have
\[
\mO^f\circ\mO_c^{\iota_f}(x_1\we ...\we x_{m-1})=(-1)^{f-1}\mO^f\left({(d\iota_f)}_p(x_1)\we ...\we {(d\iota_f)}_p(x_{m-1})\we \nu_{out}\right)
\]
\end{remark}
\subsubsection{Essential fiber product}
\begin{definition}\label{pullback fiber product orientation convention}
Let $M\overset{f}\rightarrow N\overset{g}\leftarrow P$ be transversal smooth maps of orbifolds with corners. Consider the following fiber product diagram.
\begin{equation}
\label{generic pullback diagram}
\begin{tikzcd}
M\times_N P\ar[r,"r"]\ar[d,"q"]&P\ar[d,"g"]\\M\ar[r,"f"]&N
\end{tikzcd}
\end{equation}

There is a canonical isomorphism from the relative orientation bundle of $q$ to the pullback of the relative orientation bundle of $g$. It is called \textbf{the pullback by r over f} and denoted
\[
\pull{(r/f)}:r^*\zcort{g}\simeq \zcort{q}.
\]
It is given as follows.
Let $(m,p)\in M\times P$ be such that $f(m)=g(p)$. 
Let $\mO^N,\mO^M,\mO^g$ be local orientations of $N,M,g$ in neighborhoods of $f(m),m,p$, respectively.
Define $\mO^P:=\mO^N\circ \mO^g$.
By the transversality assumption,
\[
F:=df_m\oplus -dg_p:T_mM\oplus T_pP\to T_{f(m)}N
\]
is surjective, and by definition of fiber product, there is a canonical isomorphism
\[
\psi:=dq_{(m,p)}\oplus dr_{(m,p)}:T_{(m,p)}(M\times_N P)\to \ker(F).
\]
Therefore, there exists a short exact sequence
\[
\begin{tikzcd}
0\ar[r]&T_{(m,p)}(M\times_NP)\ar[r,"\psi"]&T_mM\oplus T_pP\ar[r,"F"]&T_{f(m)}N\ar[r]&0.
\end{tikzcd}
\]
Splitting the short exact sequence, we get an isomorphism
\[
T_mM\oplus T_pP\xrightarrow{\Psi} T_{(m,p)}(M\times_NP)\oplus T_{f(m)}N.
\]
We define a local orientation $\mO^{M}\times \mO^g$ of $M\times_NP$ at $(m,p)$ to be the orientation for which $\Psi$ has sign $(-1)^{NP}$, and subsequently we define a local orientation $\pull{(r/f)}(\mO^g)$ of $q$ to satisfy the following equation.
\[
\mO^{M}\times\mO^g=\mO^M\circ \pull{(r/f)}(\mO^g).
\]
\end{definition}
\begin{definition}
If
\[
\begin{tikzcd}
S\ar[r,"r"]\ar[d,"q"]&P\ar[d,"g"]\\M\ar[r,"f"]&N
\end{tikzcd}
\]
is an essential fiber product, the pullback by $r$ over $f$ that is defined on the essential part extends uniquely to a map $\zcort{q}\to r^*\zcort{g}$ by Lemma \ref{unique extension of bundle maps}. This map is called the \textbf{pullback by $r$ over $f$.} 
\end{definition}
\subsubsection{Relatively oriented commutative diagram}
\begin{definition}\label{enhanced pullback fiber product orientation convention}
Let 
\[
\begin{tikzcd}
O\ar[r,"r"]\ar[d,"q"]&P\ar[d,"g"]\\M\ar[r,"f"]&N
\end{tikzcd}
\] be a commutative diagram such that $r,f$ are relatively oriented, with relative orientations $\mO^r,\mO^f$. There exists a canonical isomorphism from the relative orientation bundle of $q$ to the pullback of the relative orientation bundle of $g$, with respect to $\mO^r,\mO^f$. It is called \textbf{the pullback by $r$ over $f$ (with respect to $\mO^r,\mO^f$)}, and denoted
\[
\pull{\left(r/f\right)}:r^*\zcort{g}\simeq \zcort{q}.
\]
It is given as follows. Let $\mO^g$ be a local section of $r^*\zcort{g}$. The pullback $\pull{\left(r/f\right)}\mO^g$ is the unique local relative orientation of $q$ satisfying
\[
\mO^f\circ \pull{\left(r/f\right)}\mO^g=(-1)^{fg}\mO^g\circ \mO^r.
\]
Note that we omit the relative orientations of $r,f$ in the notation, as they will be fixed and declared in any situation in which we use this notation.
\end{definition}
\begin{remark}
In the case the above diagram is a pullback diagram, and $\mO^r=\pull{\left(q/g\right)}\mO^f$, Lemma~\ref{vertical pullback vs horizontal pullback} implies that both definitions coincide.
\end{remark}
\subsubsection{Base extension}
The following definition is an important example of Definition \ref{pullback fiber product orientation convention}
\begin{definition}
\label{pulled to product oreintation definition}
Let $\W$ be a manifold with corners, let $A, X, Y$ be orbifolds with corners over $\W$ and let $f:X\to Y$ be a map over $\W$. Consider the following fiber product diagrams.
\[
\begin{tikzcd}
A\times_\W X\ar[r,"\pi^{A\times X}_X"]\ar[d,"\Id\times f"]&X\ar[d,"f"]\\A\times_\W Y\ar[r,"\pi^{A\times Y}_Y"]&Y
\end{tikzcd}\qquad 
\begin{tikzcd}
 X\times_\W A\ar[r,"\pi^{X\times A}_X"]\ar[d,"f\times\Id"]&X\ar[d,"f"]\\ Y\times_\W A\ar[r,"\pi^{Y\times A}_Y"]&Y
\end{tikzcd}
\]
Assume $f$ is relatively oriented. Then $\Id\times f$ and $f\times\Id$ are relatively oriented, and we denote the pullbacks by $\pi^{A\times X}_X$ over $\pi^{A\times Y}_Y$ and by $\pi^{X\times A}_X$ over $\pi^{Y\times A}_Y$ of a relative orientation $\mO^f$ of $f$ by
\[
1\times \mO^f:=\pull{\left(\pi^{A\times X}_X/\pi^{A\times Y}_Y\right)}\mO^f, \qquad  \mO^f\times 1:=\pull{\left(\pi^{X\times A}_X/\pi^{Y\times A}_Y\right)}\mO^f, 
\]respectively.
\end{definition}
\subsubsection{Quotient}
\label{Quotient orientation convention}
Let $G$ be an oriented Lie group with orientation $\mO^G$ and Lie algebra $\mathfrak g$. Denote by $\mO^\mathfrak g:=\mO^G|_1$ the orientation of $\mathfrak g$ at the unit element $1\in G$. For a point $g\in G$, denote by $R_g:G\to G$ be multiplication on the right by $g$. Then $\mO^G|_g=R_{g^{-1}}^*\mO^\mathfrak g$. 

Let $M$ be an orbifold with corners. Let $\Phi:M\times G\to M$ be a smooth proper right-action with finite stabilizers and let ${q:M\to M/G}$ be the quotient. The map $\pi_M:M\times G\to M$ carries the pullback orientation
\[
\mO^\pi_c:=\pull{\left(\pi^{M\times G}_G/\pi^M\right)}\mO^G.
\]
For a point $p\in M$, there is a short exact sequence
\begin{equation}
0\rightarrow\mathfrak{g}\rightarrow T_pM\xrightarrow{dq_p}T_{q(p)}(M/G)\rightarrow0.
\label{quotient exact sequence orientation}
\end{equation}
Moreover, the following diagram is a fiber-product.
\begin{equation}\label{action quotient fiber-product diagram}
\begin{tikzcd}
M\times G\ar[r,"\Phi"]\ar[d,"\pi_M"]&M\ar[d,"q"]\\M\ar[r,"q"]&M/G
\end{tikzcd} 
\end{equation}

\begin{definition}\label{action quotient orientation definition}
\textbf{The canonical relative orientation $\mO^q_c$} of the quotient map $q$ is defined to be the relative orientation provided by the short exact sequence~\eqref{quotient exact sequence orientation}, with respect to the orientation $\mO^G|_1$ of $\mathfrak g$. One may check with methods analog to the proofs that appear in Section~\ref{orientations calculus proofs} that
\begin{equation}
\mO^\pi_c=\pull{\left(\Phi/q\right)}\mO^q_c.
\label{projection orientation is pullback of quotient equation}
\end{equation}

\textbf{The canonical relative orientation of the action $\Phi$} is defined to be the pullback \[
\mO^\Phi_c:=(-1)^{G}\pull{\left(\pi_M/q\right)}\mO^q_c.
\]
\end{definition}

\begin{remark}
Let $\mO^{\mathfrak g}\in\zcort{\mathfrak g}$ be an orientation of $G$. The orientation $\mO^q_c$ is calculated as follows. Choose a basis $v_1,...,v_g\in \mathfrak g$, such that $\mO^{\mathfrak g}=[v_1^\vee\we\cdots\we v_g^\vee]$ where $v_1^\vee,...,v_g^\vee$ is the dual basis. For a point $p\in M$ and tangent vectors  ${x_1,...,x_{m-g}\in T_pM}$ such that
 \[
 (x_1,..,x_{m-g},d\Phi_p(v_1),...,d\Phi_p(v_g)) 
 \]is a basis of $T_pM$, we denote the dual basis as $x_1^\vee,...,x_{m-g}^\vee,d\Phi_p(v_1)^\vee,...,d\Phi_p(v_g)^\vee$. 
Then \[\mO^q_c|_p:=\left[x_1^\vee\we...\we x_{m-g}^\vee\we d\Phi_p(v_1)^\vee\we ...\we d\Phi_p(v_g)^\vee\bigotimes d  q_p (x_1)\we...\we d  q_p (x_{m-g})\right].\]
\end{remark}
One may check the following remark using methods analog to the proofs that appear in Section~\ref{orientations calculus proofs}. The first equation is simply equation~\eqref{projection orientation is pullback of quotient equation}.
\begin{remark}\label{action orientation pullback}
Let $\mO^{\mathfrak g}\in\zcort{\mathfrak g}$ be an orientation of $G$. The relative orientations $\mO^\pi_c,\mO^\Phi_c$ are calculated as follows. Let $\mO^\mathfrak g\in \zcort{\mathfrak g}$ be an orientation of the Lie algebra. Choose a basis $v_1,...,v_g\in \mathfrak g$ such that $\mO^\mathfrak g=[v_1^\vee\we\cdots\we v_g^\vee]$ where $v_1^\vee,...,v^\vee_g$ is the dual basis. For a point $(p,g)\in M\times G$ and a basis $x_1,...,x_m\in T_pM$, we have
\begin{align*}
\mO^\pi_c|_{(p,g)}(x_1,...,x_m,v_1,...,v_g)&=\left(d\pi_M(x_1),...,d\pi_M(x_m)\right),
\\\mO^\Phi_c|_{(p,g)}(x_1,...,x_m,v_1,...,v_g)&=\left(d\Phi_g(x_1),...,d\Phi_g(x_m)\right).
\end{align*}
where $\Phi_g:M\to M$ is multiplication by $g$ on the right.
\end{remark}

\subsection{Orientations calculus}
\label{orientations calculus section}
The following lemmas describe, in particular, how to orient complex structures from simpler ones. The proofs are provided in Section \ref{orientations calculus proofs}
\begin{lemma}
\label{pullback by diffeomorphism of composition raw lemma}
In the setting of Definition \ref{pullback fiber product orientation convention}, assume $f$ is a local diffeomorphism. Then $r$ is a local diffeomorphism as well, and if $a:N\to A$ is a map, then the following diagram is commutative.
\begin{equation}
\label{pullback by diffeomorphism}
\begin{tikzcd}
r^*\zcort{g}\otimes q^*f^*\zcort{a}\ar[d,swap,"\pull{(r/f)}\otimes q^*\pull{f}"]\ar[r,equal,"comp."]&r^*\zcort{a\circ g}\ar[d,"\pull{r}"]\\
\zcort{q}\otimes q^*\zcort{a\circ f}\ar[r,equal,"comp."]
&\zcort{a\circ g\circ r}
\end{tikzcd}
\end{equation}
\end{lemma}

\begin{lemma}
\label{pullback of vertical local diffeomorphism lemma}
In the setting of Definition \ref{pullback fiber product orientation convention}, assume $g$ is a local diffeomorphism. Then $q$ is a local diffeomorphism as well, and
\begin{equation}
\label{pullback of vertical diffeomorphism}
\pull{(r/f)}(\mO^g_c)=\mO^q_c.
\end{equation}
\end{lemma}
\begin{remark}
\label{pullback to product of vertical local diffeomorphism  remark}
As a consequence of Lemma \ref{pullback of vertical local diffeomorphism lemma}, in the setting of Definition \ref{pulled to product oreintation definition}, the following equations hold.
\[
1\times\mO_c^f=\mO^{\Id\times f}_c, \qquad \mO^f_c\times 1=\mO^{f\times\Id}_c.
\]
\end{remark}

\begin{lemma}
\label{vertical pullback vs horizontal pullback}
In the setting of Definition \ref{pullback fiber product orientation convention}, let $\mO^f,\mO^g$ be local orientations of $f,g$ at points $m,p$ in $M,P$, respectively. Then
\[
\mO^f\circ \pull{(r/f)}\mO^g=(-1)^{fg}\mO^g\circ \pull{(q/g)}\mO^f.
\]
\end{lemma}

\begin{lemma}\label{horizontal pullback composition}
Assume that the squares in the following diagram are pullbacks.
\[
\begin{tikzcd}
M\times_NP\ar[r,"r"]\ar[d,"q"]&P\ar[r,"t"]\ar[d,"g"]&T\ar[d,"b"]\\
M\ar[r,"f"]&N\ar[r,"a"]&K
\end{tikzcd}
\]
Then, the rectangle is also a pullback diagram, and
\[
\pull{\left((t\circ r)/(a\circ f)\right)}=\pull{(r/f)}\circ r^*\pull{(t/a)}.
\]
\end{lemma}
\begin{lemma}\label{vertical pullback composition}
Assume that the squares in the following diagrams are pullbacks.
\[
\begin{tikzcd}
M\times_NP\ar[r,"r"]\ar[d,"q"]&P\ar[d,"g"]\\
M\ar[r,"f"]\ar[d,"t"]&N\ar[d,"a"]\\
T\ar[r,"b"]&K   
\end{tikzcd}
\]
Then the rectangle is also a pullback diagram, and the following diagram is commutative.
\[
\begin{tikzcd}
r^*\zcort{a\circ g}\ar[rr,"\pull{(r/b)}"]\ar[d,equal]&&\zcort{t\circ q}\ar[d,equal]\\
r^*\zcort{g}\otimes r^*g^*\zcort{a}\ar[rr,"\pull{(r/f)}\otimes \pull{(f/b)}"]&&\zcort{q}\otimes q^*\zcort{t}
\end{tikzcd}
\]
\end{lemma}
\begin{lemma}[Boundary factorization]\label{boundary factorization orientation}
Let $M\overset{f}\to P\overset{g}\to N$ be such that $f$ and $g\circ f$ are proper submersions. Assume that $f$ factorizes through the boundary of $g$ in the sense of Definition~\ref{boundary factorization}. That is, we have the following pullback diagram.
\[
\begin{tikzcd}
\pa_{g\circ f}M\ar[r,"\iota_{g\circ f}"]\ar[d,swap,dashed, "{\iota_g}^* f"]&M\ar[d,"f"]\\
\pa_g P\ar[r,swap,"\iota_g"]&P
\end{tikzcd}
\]
Then 
\begin{equation}
\mO_c^{\iota_g}\circ \pull{\left(\iota_{g\circ f}/\iota_g\right)}(\mO^f)=\mO^f\circ\mO_c^{\iota_{g\circ f}}.
\label{orientation boundary commutation}    
\end{equation}
\end{lemma}
\begin{lemma}\label{orientation of boundary of product}
Let $A,X, Y$ be orbifolds with corners over $\W$ and let $f:X\to Y$ be a submersion over $\W$. Consider the following pullback diagram.
\[
\begin{tikzcd}
\pa_{(\Id\times f)}(A\times_\W X)\ar[r,"\iota_f^*\pi_X"]\ar[d,swap,"\iota_{(\Id\times f)}"]&\pa_f X\ar[d,"\iota_f"]\\
A\times_\W X\ar[d,swap,"\Id\times f"]\ar[r,"\pi_X"]&X\ar[d,"f"]\\A\times_\W Y\ar[r,swap,"\pi_Y"]&Y
\end{tikzcd}
\]
Then
\[
\pull{\left(\iota^*_f\pi_X/\pi_X\right)}\mO^{\iota_f}_c=\mO^{\iota_{(\Id\times f)}}_c.
\]
\end{lemma}

\subsection{Proofs of results in orientations calculus}
\label{orientations calculus proofs}
\begin{proof}[Proof of Lemma \ref{pullback by diffeomorphism of composition raw lemma}]
The commutativity of diagram \eqref{pullback by diffeomorphism} amounts to showing that if $\mO^g$ and $\mO^a$ are local relative orientations of $g$ and $a$, respectively, then
\begin{equation}
\mO^a\circ\mO^g\circ\mO^r_c=\mO^a\circ\mO^f_c\circ \pull{(r/f)}(\mO^g).   
\end{equation}
It is enough to show that
\begin{equation}
\label{local orientation equality pullback diffeomorphism proof}
\mO^g\circ\mO^r_c=\mO^f_c\circ \pull{(r/f)}(\mO^g), 
\end{equation}which proof follows.

Note that by assumption, $df_m$ is an isomorphism. There is an isomorphism
\begin{align*}
    \eta:\ker(F)&\to T_pP,\\
    (x,y)&\overset{\eta}\mapsto y, 
\end{align*}
whose inverse is $\left(df_m^{-1}\left(dg_p(y)\right),y\right)\mapsfrom y$. Note that
\[
\eta\circ\psi=dr_{(m,p)},
\]
which suggests a splitting as follows.
\[
T_mM\oplus T_pP\xrightarrow{\begin{pmatrix}0&dr^{-1}\\df&-dg\end{pmatrix}}T_{(m,p)}(M\times_NP)\oplus T_{f(m)}N
\] 
Let $\mO^N,\mO^g$ be local orientations of $N,g$. Providing $M$ and $M\times_NP$ with the local orientations
\[
\pull{f}(\mO^N)=\mO^N\circ \mO^f_c\quad,\qquad\pull{r}(\mO^N\circ\mO^g)=\mO^N\circ\mO^g\circ \mO_c^r,
\]
results in $sgn\begin{pmatrix}
0&dr^{-1}\\df&-dg
\end{pmatrix}=(-1)^{NP}.$ Therefore, 
\[
\left(\mO^N\circ\mO^f_c\right)\times \mO^g=\mO^N\circ\mO^g\circ\mO^r_c.
\]
Therefore, $\pull{(r/f)}(\mO^g)$ is given by the equation
\[
\mO^N\circ\mO^f_c\circ\pull{(r/f)}(\mO^g)=\mO^N\circ\mO^g\circ\mO^r_c.
\]
In particular, 
\[
\mO^f_c\circ\pull{(r/f)}(\mO^g)=\mO^g\circ\mO^r_c,
\]
which was to be shown.
\end{proof}
\begin{proof}[Proof of Lemma \ref{pullback of vertical local diffeomorphism lemma}]
Note that by assumption, $dg_p$ is an isomorphism. There is an isomorphism
\begin{align*}
\eta:\ker(F)&\to T_mM,\\
(x,y)&\mapsto x,
\end{align*}
whose inverse is $\left(x,dg_p^{-1}(df_m(x))\right)\mapsfrom x$. Note that
\[
\eta\circ \psi=dq_{(m,p)},
\]
which suggests a splitting as follows.
\[
T_mM\oplus T_pP\xrightarrow{\begin{pmatrix}dq^{-1}&0\\df&-dg\end{pmatrix}}T_{(m,p)}(M\times_NP)\oplus T_{f(m)}N
\] 
Let $\mO^N,\mO^M$ be local orientations of $N,M$ at $m,f(m)$, respectively. 
Equipping $P,M\times_NP$ with the local orientations 
\[
\pull{g}(\mO^N)=\mO^N\circ \mO^g_c\quad,\qquad \pull{q}(\mO^M)=\mO^M\circ \mO_c^q,
\]
results in $sgn\begin{pmatrix}
dq^{-1}&0\\df&-dg
\end{pmatrix}=(-1)^{N}=(-1)^{NP}$ since $\dim N=\dim P.$
Therefore, 
\[
\mO^M\times \mO^g_c=\mO^M\circ\mO^q_c.
\]
This finishes the proof of the lemma.
\end{proof}
\begin{proof}[Proof of Lemma \ref{vertical pullback vs horizontal pullback}]
Let $\mO^N$ be a local orientation of $N$ and let $\mO^M=\mO^N\circ \mO^f,\mO^P=\mO^N\circ \mO^g$ be local orientations of $M,P$, respectively. Assume 
\[
TM\oplus TP\xrightarrow{\begin{pmatrix}
\a&\b\\df&-dg
\end{pmatrix}}T\left(M\times_NP\right)\oplus TN
\]
is a splitting used in the definition of the pullback by $r$ over $f$. Then
\[
TP\oplus TM\xrightarrow{\begin{pmatrix}
\b&\a\\dg&-df
\end{pmatrix}}TM\times_NP\oplus TN
\]
is a splitting used in the definition of the pullback by $q$ over $g$,
and with respect to the orientation $\mO^M\times\mO^g$ of $M\times_NP$,
\[sgn\begin{pmatrix}
\b&\a\\dg&-df
\end{pmatrix}=(-1)^{N+MP}
sgn\begin{pmatrix}
\a&\b\\df&-dg
\end{pmatrix}=(-1)^{N+NP+MP}.
\]
On the other hand, with respect to the orientation $\mO^P\times \mO^f$ of $M\times_NP$, by definition,
\[
sgn\begin{pmatrix}
\b&\a\\dg&-df
\end{pmatrix}=(-1)^{MN}.
\]
Therefore, 
\[
\mO^N\circ \mO^g\circ\pull{(q/g)}\mO^f=\mO^P\times\mO^f=(-1)^{s}\mO^M\times \mO^g=(-1)^{s}\mO^N\circ \mO^f\circ \pull{(r/f)}\mO^g,
\]
where 
\[
s=N+NP+MP+NM\equiv_2(M-N)(P-N)=fg.
\]
\end{proof}
\begin{proof}[Proof of Lemma \ref{horizontal pullback composition}]
Assume there are the following splittings.
\begin{align*}
A:\begin{pmatrix}x&y\\da&-db\end{pmatrix}:TN\oplus TT\xrightarrow{}&TP\oplus TK,\qquad &x\circ dg+y\circ dt=\Id_{TP}\\
B:=\begin{pmatrix}z&w\\df&-dg\end{pmatrix}:TM\oplus TP\xrightarrow{}&T(M\times_NP)\oplus TN,\qquad & z\circ dq+w\circ dr=\Id_{T(M\times_NP)}
\end{align*}
Then, there is the following splitting,
\[
C:=\begin{pmatrix}\a&\b\\da\circ df&-db\end{pmatrix}:TM\oplus TT\xrightarrow{}T(M\times_NP)\oplus TK
\]
where
\[
\a=z+w\circ x\circ df,\qquad \b=w\circ y.
\]
Indeed, 
\[
\a\circ dq+\b\circ dt\circ dr=z\circ dq+w\circ (x\circ dg+y\circ dt)\circ dr=z\circ dq+w\circ dr=\Id_{T(M\times_NP)}.\]

We want to relate the determinant of $C$ to those of $A$ and $B$. However, there is no natural composition possible to make, and therefore we turn to a space on which all $A,B$ and $C$ act. Consider the following maps
\[
TN\oplus TP\oplus TM\oplus TT\to TN\oplus TP\oplus T(M\times_NP)\oplus TK
\]given by the matrices
\begin{align*}
X:=&\begin{pmatrix}
-1&0&0\\0&-1&0\\0&0&C
\end{pmatrix}=\begin{pmatrix}
-1&0&0&0\\0&-1&0&0\\0&0&z+w\circ x\circ df&w\circ y\\0&0&da\circ df&-db
\end{pmatrix},\\ 
Y:=&\begin{pmatrix}
-1&-dg&df&0\\x&0&0&y\\0&w&z&0\\da&0&0&-db
\end{pmatrix}.
\end{align*}
We show that the $Y$ follows from the $X$ by elementary rows/columns actions.
\begin{align*}
\begin{pmatrix}
-1&0&0&0\\0&-1&0&0\\0&0&z+w\circ x\circ df&w\circ y\\0&0&da\circ df&-db
\end{pmatrix}\xrightarrow{\begin{smallmatrix}
C3-=C1\circ df\\
C4-=C2\circ y
\end{smallmatrix}}
\begin{pmatrix}
-1&0&df&0\\0&-1&0&y\\0&0&z+w\circ x\circ df&w\circ y\\0&0&da\circ df&-db
\end{pmatrix}\\
\xrightarrow{\begin{smallmatrix}
R4-=da\circ R1\\
R3-=w\circ R2
\end{smallmatrix}}
\begin{pmatrix}
-1&0&df&0\\0&-1&0&y\\0&w&z+w\circ x\circ df&0\\da&0&0&-db
\end{pmatrix}\xrightarrow{\begin{smallmatrix}
C3-=C2\circ x\circ df
\end{smallmatrix}}
\begin{pmatrix}
-1&0&df&0\\0&-1&x\circ df&y\\0&w&z&0\\da&0&0&-db
\end{pmatrix}\\\xrightarrow{\begin{smallmatrix}
R2-=x\circ R1
\end{smallmatrix}}
\begin{pmatrix}
-1&0&df&0\\x&-1&0&y\\0&w&z&0\\da&0&0&-db
\end{pmatrix}\xrightarrow{\begin{smallmatrix}
C2+=C1\circ dg+C4\circ dt
\end{smallmatrix}}
\begin{pmatrix}
-1&-dg&df&0\\x&0&0&y\\0&w&z&0\\da&0&0&-db
\end{pmatrix}
\end{align*}
Let $\mO^K,\mO^N,\mO^M,\mO^b$ be local orientations of $K,N,M,b$, respectively. Let
\begin{align*}
\mO^T&:=\mO^K\circ \mO^b\quad,\\
\mO^P&:=\mO^N\circ \pull{(t/a)}\mO^b\quad,
\\\mO^{M\times_NP}&:=\mO^M\circ \pull{(r/f)}\pull{(t/a)}\mO^b.
\end{align*}
We investigate $sgn(X)$ with respect to these orientations.
By definition of $X$, it follows that $sgn(X)=(-1)^{N+P}\cdot sgn(C).$ Since $X$ is carried to $Y$ by elementary transformations, we know that $sgn(X)=sgn(Y)$. Denote by $D=\begin{pmatrix}
0&0\\1&0
\end{pmatrix}$ and let
\[
Z:=\begin{pmatrix}
A&0\\
D&B
\end{pmatrix}:TN\oplus TT\oplus TM\oplus TP\to TP\oplus TK\oplus T(M\times_NP)\oplus TN
\]
Then the following equation holds.
\[
Y=\begin{pmatrix}
&&&1\\1&&&\\&&1&\\&1&&
\end{pmatrix}\cdot Z\cdot \begin{pmatrix}
1&&&\\&&&1\\&&1&\\&1&&
\end{pmatrix}.
\]
Recalling that $sgn(A)=(-1)^{KT}$ and $sgn(B)=(-1)^{PN},$ we obtain that $sgn(Y)=(-1)^s$, with
\begin{align*}
s\equiv_2&\uline{TM+TP}+MP+\cancel{NP}+NK+\uwave{N(M\times_NP)+K(M\times_NP)}+KT+\cancel{PN}\\
\equiv_2&\uline{T(M\times_NP)-TN}+NK+KT+MP
+\uwave{(N+K)(M\times_NP)}\\
\equiv_2&(T+N+K)(M+N+P)+TK+N(T+K)+MP\\
\equiv_2&P(\cancel{M}+N+P)+TK+N(T+K)+\cancel{MP}\\
\equiv_2&(P+N)(T+K)+TK\equiv_2(T+K)(T+K)+TK=T+K+TK.
\end{align*}
Therefore, 
\[
sgn(C)=(-1)^{N+P}sgn(Y)=(-1)^{N+P+T+K+TK}=(-1)^{TK}.
\]
Therefore, 
\[
\mO^M\times\mO^b=\mO^M\circ\pull{(r/f)}\pull{(t/a)}\mO^b,
\]
i.e.
\[
\pull{(r/f)}\mO^b=\pull{(r/f)}\pull{(t/a)}\mO^b.
\]
This concludes the proof.
\end{proof}
\begin{proof}[Proof of Lemma \ref{vertical pullback composition}]
Assume there are the following maps.
\[
\begin{tikzcd}
TM\times_NP\ar[from=d,"z"]&TP\ar[l,"w"]\ar[d,"dg"]\\
TM\ar[from=d,"x"]&TN\ar[l,"y"]\ar[d,"da"]\\
TT\ar[r,"db"]&TK   
\end{tikzcd}
\]
Assume that the following maps are the required splittings.
\begin{align*}
TT\oplus TN\xrightarrow{\begin{pmatrix}x&y\\db&-da\end{pmatrix}}&TM\oplus TK,\qquad &x\circ dt+y\circ df=\Id_{TM}\\
TM\oplus TP\xrightarrow{\begin{pmatrix}z&w\\df&-dg\end{pmatrix}}&T(M\times_NP)\oplus TN,\qquad & z\circ dq+w\circ dr=\Id_{T(M\times_NP)}
\end{align*}
Then there is the following splitting,
\[
TT\oplus TP\xrightarrow{\begin{pmatrix}\a&\b\\db&-da\circ dg\end{pmatrix}}T(M\times_NP)\oplus TK
\]
where
\[
\a=z\circ x,\qquad \b=z\circ y\circ dg+w.
\]
Indeed, 
\[
\a\circ dt\circ dq+\b\circ dr=z\circ \left(x\circ dt\circ dq+y\circ df\circ dq
\right)+w\circ dr=z\circ dq+w\circ dr=\Id_{T(M\times_NP)}.\]
Consider the following maps between the spaces
\[
TM\oplus TN\oplus TT\oplus TP\rightarrow TM\oplus TN\oplus T(M\times_NP)\oplus TK,
\]given by the matrices
\[
X:=\begin{pmatrix}
-1&0&0&0\\0&1&0&0\\0&0&\a&\b\\0&0&db&-da\circ dg
\end{pmatrix},\qquad 
Y:=\begin{pmatrix}
0&y&x&0\\df&1&0&-dg\\z&0&0&w\\0&-da&db&0
\end{pmatrix}
\]
We show that the $Y$ follows from the $X$ by elementary rows/columns actions.
\begin{align*}
\begin{pmatrix}
-1&0&0&0\\0&1&0&0\\0&0&\a&\b\\0&0&db&-da\circ dg
\end{pmatrix}\xrightarrow{C1+=C3\circ dt\circ dq\circ z+C4\circ dr\circ z}
\begin{pmatrix}
-1&0&0&0\\0&1&0&0\\z&0&\a&\b\\0&0&db&-da\circ dg
\end{pmatrix}\\
\xrightarrow{\begin{smallmatrix}C3-=C1\circ x\\C4-=C1\circ y\circ dg-C2\circ dg\end{smallmatrix}}
\begin{pmatrix}
-1&0&x&y\circ dg\\0&1&0&-dg\\z&0&0&w\\0&0&db&-da\circ dg
\end{pmatrix}\xrightarrow{\begin{smallmatrix}
R4-=da\circ R2\\
R1+=y\circ R2
\end{smallmatrix}}
\begin{pmatrix}
-1&y&x&0\\0&1&0&-dg\\z&0&0&w\\0&-da&db&0
\end{pmatrix}\\
\xrightarrow{\begin{smallmatrix}
C1+=C3\circ dt+C2\circ df
\end{smallmatrix}}
\begin{pmatrix}
0&y&x&0\\df&1&0&-dg\\z&0&0&w\\0&-da&db&0
\end{pmatrix}.
\end{align*}
Denote by 
\[
A:=\begin{pmatrix}
x&y\\db&-da
\end{pmatrix},\quad B:=\begin{pmatrix}
z&w\\df&-dg
\end{pmatrix},\quad C:=\begin{pmatrix}
\a&\b\\db&-da\circ dg
\end{pmatrix},
\]
Let $\mO^K,\mO^T,\mO^a,\mO^g$ be local relative orientations of $K,T,a,g$, respectively. Let 
\begin{align*}
\mO^M&:=\mO^T\circ\pull{(f/b)}(\mO^a)\quad,\\\mO^{M\times_NP}&:=\mO^M\circ \pull{(r/f)}(\mO^g)=\mO^T \circ\pull{(f/b)}(\mO^a)\circ\pull{(r/f)}(\mO^g)
\end{align*}
We investigate the sign of the map $X$ with respect to these orientations.
By definition of $X$, it follows that $sgn(X)=(-1)^M\cdot sgn(C).$
Since $X$ is carried to $Y$ by elementary transformations, we know that $sgn(X)=sgn(Y)$.
Denote by $D=\begin{pmatrix}
0&1\\0&0
\end{pmatrix}$ and set
\[Z:=\begin{pmatrix}
B&0\\D&A
\end{pmatrix}
:TM\oplus TP\oplus TT\oplus TN\to T(M\times_NP)\oplus TN\oplus TM\oplus TK.
\]
Then the following equation holds.
\[
Y=\begin{pmatrix}
&&1&\\&1&&\\1&&&\\&&&1
\end{pmatrix}\cdot Z\cdot\begin{pmatrix}
1&&&\\&&&1\\&&1&\\&1&&
\end{pmatrix}
\]Recalling that $sgn(A)=(-1)^{KN}$ and $sgn(B)=(-1)^{NP}$, we obtain that
$sgn(Y)=(-1)^s,$ for
\begin{align*}
s&=NT+\cancel{NP}+TP+MN+M(M\times_{\bcancel{N}}P)+N(\bcancel{M}\times_NP)+KN+\cancel{NP}\\
&=NT+TP+MN+M+MP+N+NP+KN\\
&=(N+P)(T+M)+M+NP+N+KN\\
&=(\cancel{N}+P)(K+N)+M+NP+\cancel{N+KN}\\
&=PK+M.
\end{align*}
We obtain that $sgn(C)=(-1)^{PK}$, and thus
\[
\mO^{T}\times(\mO^{a}\circ \mO^{g})=\mO^T\circ \pull{(f/b)}(\mO^a)\circ\pull{(r/f)}(\mO^g).
\]
Therefore,
\[
\pull{(r/b)}(\mO^a\circ\mO^g)=\pull{(f/b)}(\mO^a)\circ\pull{(r/f)}(\mO^g),
\]
which was to be shown.
\end{proof}
\begin{proof}[Proof of Lemma \ref{boundary factorization orientation}]
Since $f$ is a submersion, so is $\iota^*f$. Let $y_1,...,y_{m-1}$ be a basis for $T\pa_{g\circ f} M$ at a point, such that $x_i=d(\iota^*f)(y_i),i=1,...,p-1$ form a basis of $T\pa_gP$. Define a section
\begin{align*}
s:T\pa_g P&\to T\pa_{g\circ f}M,\\
\sum \a_ix_i&\mapsto \sum \a_iy_i.
\end{align*}
Note that $\left(d\iota_{g\circ f}(y_i)\right)_{i=1}^{m-1}$ together with $\nu_{out}$ form a basis of $TM$, and define a map
\begin{align*}
t:TM&\to T\pa M\\
\sum \b_j d\iota_{g\circ f}(y_j)+\g\nu_{out}&\mapsto \sum \b_jy_j.
\end{align*}
Note that if $(x,z)\in \ker(F)$, then $z\in Im(d\iota_{g\circ f})$, and thus $z=d\iota_{g\circ f}(t(z))$. further,
\[
d\iota_g\circ d(\iota^*f)\left(s(x)-t(z)
\right)=d\iota_g(x)-df(z)=0,
\]
and therefore $s(x)-t(z)\in \ker d(\iota^*f).$
Consider the map
\[
C:=\begin{pmatrix}
0&t\\d\iota_g&-df
\end{pmatrix}:T\pa P\oplus TM\xrightarrow{}T\pa M\oplus TP.
\]It is injective. Indeed, if 
\[
\begin{pmatrix}
0&t\\d\iota_g&-df
\end{pmatrix}\begin{pmatrix}
x\\z
\end{pmatrix}=0
\]
then $(x,z)\in \ker(F)$, and thus 
\[
\ker(d\iota^*f)\ni s(x)-t(z)=s(x),\]
which implies
\[
0=d(\iota^*f)s(x)=x.
\]Therefore, $z\in \ker(df)$ and thus $z\in Im(d\iota_{g\circ f})$. So
\[
z=d\iota_{g\circ f}(t(z))=d\iota_{g\circ f}(0)=0.
\]Since the dimensions of the source and target of $C$ are the same, it is an isomorphism, and in fact, a splitting of
\[
\begin{tikzcd}
0\ar[r]&T\pa M\ar[r,"{\begin{pmatrix}
d(\iota^*f)\\d\iota_{g\circ f}
\end{pmatrix}}"]&T\pa P\oplus TM\ar[r,"{\left(
d\iota_g,\,\,-df
\right)}"]&TP\ar[r]&0
\end{tikzcd}
\]

Let $\mO^P,\mO^f$ be local orientations of $P,f$, respectively. Set
\begin{align*}
\mO^M&=\mO^P\circ\mO^f,\\
\mO^{\pa P}&=\mO^P\circ\mO^{\iota_g}_c,\\
\mO^{\pa M}&=\mO^M\circ \mO^{\iota_{g\circ f}}_c.
\end{align*}
With these orientations, the maps
\begin{align*}
A:\R\oplus T\pa M&\to TM,\qquad (\a,z)\mapsto \a\nu_{out}+d\iota_{g\circ f}(z) \\
B:\R\oplus T\pa P&\to TP,\qquad (\b,x)\mapsto \b\nu_{out}+d\iota_g(x)
\end{align*}
have $sgn(A)=sgn(B)=(-1)^{N}$. Note that $A$ and $B$ are isomorphisms, and
\[
B^{-1}(y)=y-d\iota_{g\circ f}(t(y)),t(y).
\]
Therefore, the map
\[
D:=\begin{pmatrix}
1&0\\0&C
\end{pmatrix}:\R\oplus T\pa P\oplus TM\to \R\oplus T\pa M\oplus TP,
\]
which is equivalent to
\[
\begin{pmatrix}
1&0&y-d\iota_{g\circ f}(t(y))\\
0&0&t(y)\\
\nu_{out}&d\iota_g&-df
\end{pmatrix}
\]
has $sgn(C)=sgn(D)=1=(-1)^{P\cdot \pa P}$. Therefore, 
\[
\mO^{\pa P}\times \mO^f=\mO^M\circ\mO^{\iota_g\circ f},
\]
which breaks down to the equation
\[
\mO^{P}\circ \mO^{\iota_g}_c\circ\pull{(\iota_{g\circ f}/\iota_g)}\mO^f=\mO^P\circ\mO^f\circ\mO^{\iota_g\circ f}.
\]
The lemma follows.
\end{proof}
\begin{proof}[Proof of Lemma \ref{orientation of boundary of product}]
For ease, we assume $\W$ is a point. The general case is similar.
We identify $\pa_{(\Id\times f)}(A\times_\W X)$ with $A\times_\W \pa_fX$, such that $\iota^*_f\pi_X$ reduces to the projection $\pi_{\pa_fX}$ and $\iota_{(\Id\times f)}$ reduces to the inclusion $\Id\times \iota_f$.
Consider the map
\[
C:=\begin{pmatrix}
1&0&0\\
0&0&1\\0&1&d\iota_f
\end{pmatrix}:TA\oplus TX\oplus T\pa_fX\to TA\oplus T\pa_fX\oplus TX
\] 
which splits the following sequence.
\[
0\rightarrow\oplus T\pa_fX\xrightarrow{\begin{pmatrix}1&0\\0&d\iota_f\\0&1\end{pmatrix}}TA\oplus TX\oplus T\pa_fX\xrightarrow{\begin{pmatrix}
0&1&-d\iota_f
\end{pmatrix}}TX\rightarrow 0
\]

With whichever orientations chosen for $A,X,\pa_fX$, we have $sgn(C)=(-1)^{X\cdot \pa_fX}$.
But $C$ can be interpreted as a map
\[
T(A\times X)\oplus T\pa_fX\to T\left(\pa_{(\Id\times f)}(A\times X)\right)\oplus TX.
\]
Let $\mO^X,\mO^A$ be local orientations of $X,A$, respectively. Set
\begin{align*}
\mO^{\pa X}&:=\mO^X\circ\mO^{\iota_f}_c,\\\mO^{A\times X}&:=\mO^{A}\times \mO^X,\\
\mO^{\pa(A\times X)}&:=(\mO^A\times\mO^X)\circ \mO^{\iota_{(\Id\times f)}}_c.
\end{align*}
With these orientations, the identifying map
\[
T\pa_{(\Id\times f)}(A\times X)\to TA\oplus T\pa_fX
\]
is orientation preserving. Therefore, 
\[
\mO^A\times \mO^{\iota_f}_c=(\mO^{A}\times\mO^X)\circ\mO^{\iota_{(\Id\times f)}}_c,
\]
which implies by definition that
\[
\mO^{\iota_{(\Id\times f)}}_c=\pull{\left(\iota^*_f\pi_X/\pi_X\right)}(\mO^{\iota_f}_c).
\]
\end{proof}

\section{Orientors}
\label{orientors complete section}
\subsection{Orientors}
\label{orientors section}
\subsubsection{Definition and basic examples}
In this paper, we will concentrate mostly on bundle maps of the following form.
\begin{definition}\label{orientor}
Let $g:M\to N$ be a map and let $Q,K$ be $\Z/2$ bundles over $M,N$, respectively. A \textbf{\orientor{g} of $Q$ to $K$} is a graded bundle map \[
G:Q\to \zcort{g}\otimes_{\Z/2} g^*K.
\] Its \textbf{degree} is the usual degree as a bundle map, where $\zcort{g}\otimes g^*K$ is, as usual, the graded tensor product and $\zcort{g}$ is concentrated in degree $-\text{reldim } g$.
A \eorientor{g} of $K$ is a \orientor{g} of $g^*K$ to $K$.
\end{definition}
\begin{terminology}
if $g=\pi^M:M\to *$ is the constant map, then we say \orientor{M} for \orientor{g}.
\end{terminology}
\begin{definition}[Orientation as an orientor]
\label{orientation as orientor}
Let $f:M\to N$ be a relatively orientable map of orbifolds with corners. The section $\mO^f:M\to \zcort{f}$ can be extended uniquely to a $\Z/2$ equivariant map 
\[
\phi^{\mO^f}:\underline{\Z/2}\to \zcort{f}
\]
which satisfies
\[\phi^{\mO^f}(1)=\mO^f.\]
The map $\phi^{\mO^f}$ can be considered as an \eorientor{f} of $\underline{\Z/2}$. If $f$ is a local diffeomorphism, we denote by
\[
\phi_f:=\phi^{\mO^f_c}.
\]
If $N=*$ and $M$ is oriented with orientation $\mO^M$, then we abbreviate
\[
\phi_M:=\phi^{\mO^M}.
\]
\end{definition}
\begin{example}\label{symmetry isomorphism as orientor example}
Let $A,B$ be $\Z/2$ vector bundles over an orbifold $M$. The symmetry operator $\tau_{A,B}:A\otimes B\to B\otimes A$ of Definition \ref{tensor product} may be considered as an \orientor{\Id_M}. More generally, any bundle map of bundles over an orbifold $M$ may be considered as an \orientor{\Id_M}.
\end{example}
\begin{definition}\label{tensored orientor extension}
Let $M,N,g,Q,K,G$ be as in Definition \ref{orientor} and let $T$ be a $\Z/2$ bundle over $N$. Then the \textbf{right $T$ extension of $G$} is the \orientor{g} of $Q\otimes g^*T$ to $K\otimes T$ given by
\[
Q\otimes g^*T\overset{G\otimes\Id}{\xrightarrow{\hspace*{1cm}}}\zcort{g}\otimes g^*(K\otimes T).
\]
It is denoted by $G^T.$
Similarly, the \textbf{left $T $ extension of $G$} is the \orientor{g} of $g^*T\otimes Q$ to $T\otimes K$ given by
\[
g^*T\otimes Q\xrightarrow{\Id \otimes G}g^*T\otimes \zcort{g}\otimes g^*K\xrightarrow{\tau\otimes\Id}\zcort{g}\otimes g^*(T\otimes K).
\]
It is denoted by ${}^TG.$
\end{definition}
\begin{remark}[Repeated extension]
In the previous setting, if $T_1,T_2$ are $\Z/2$ bundles over $N$, then
\[
{\left(G^{T_1}\right)}^{T_2}=G^{T_1 \otimes T_2},\qquad {}^{T_2}{\left({}^{T_1}G\right)}={}^{T_2\otimes T_1}G.
\]
\end{remark}
\subsubsection{Composition}
\begin{definition}\label{orientor composition}
Let $M\overset{f}\to P\overset{g}\to N$ be maps and let $Q,K,R$ be {$\Z/2$~bundles} over $M,P,N$, respectively. Let $F:Q\to \zcort{f}\otimes f^*K$ be a \orientor{f} of $Q$ to $K$ and $G:K\to \zcort{g}\otimes g^*R$ be a \orientor{g} of $K$ to $R$. \textbf{The composition $G\bu F$} is the \orientor{g\circ f} of $Q$ to $R$ given as follows.
\[
Q\overset{F}\to\zcort{f}\otimes f^*K\overset{\Id \otimes f^*G}\to \zcort{f}\otimes f^*\zcort{g}\otimes f^*g^*P\overset{comp.}=\zcort{g\circ f}\otimes {(g\circ f)}^*R
\]
\end{definition}
\begin{lemma}\label{composition of orientors is associative}
The composition of orientors is associative.
\end{lemma}
\begin{proof}[Proof of Lemma \ref{composition of orientors is associative}]
In addition to the data in Definition \ref{orientor composition}, let $h:N\to L$ be a map, let $T$ be a $\Z/2$ bundle over $L$ and let $H$ be a \orientor{h} of $R$ to $T$.
Then the equality\[H\bu (G\bu F)=(H\bu G)\bu F\] can be seen from the following diagram.
\begin{equation*}
\begin{tikzcd}
Q\ar[rr,"G\bu F"]\ar[dd,swap,"F"]&&\zcort{g\circ f}\otimes (g\circ f)^*R\ar[dd,"\Id \otimes (g\circ f)^*H"]\\\\
\zcort{f}\otimes f^*K\ar[uurr,"\Id \otimes f^*G"]\ar[rr,swap,"\Id \otimes f^*(H\bu G)"]&&\zcort{h\circ g\circ f}\otimes (h\circ g\circ f)^*T.
\end{tikzcd}
\end{equation*}
\end{proof}
\subsubsection{Pullback along a relatively oriented map}

\begin{lemma}\label{composition of relative orientation orientors lemma}
Let $M\xrightarrow{f}P\xrightarrow{g}N$ be maps, and let $\mO^f,\mO^g$ be relative orientations of $f,g,$ respectively. Then
\[
\phi^{\mO^g}\bu\phi^{\mO^f}=(-1)^{fg}\phi^{\mO^g\circ\mO^f}.
\]
\end{lemma}
\begin{remark}\label{composition of local diffeomorphisms orientors lemma}
In particular, if $f,g$ are local diffeomorphisms then Remark~\ref{canonical orientation of local diffeomorphisms composition remark} implies 
\[
\phi_g\bu\phi_f=\phi_{g\circ f}.
\]
\end{remark}
\begin{proof}[Proof of Lemma \ref{composition of relative orientation orientors lemma}]
We calculate
\[
\phi^{\mO^g}\bu\phi^{\mO^f}(1)\overset{\text{Def. \ref{orientor composition}}}=\left(1_{\zcort{f}}\otimes \phi^{\mO^g}\right)\left(\mO^f\otimes 1\right)\overset{\text{Koszul~\ref{Koszul signs}}}=(-1)^{fg}\mO^g\circ \mO^f.\]
\end{proof}
\begin{definition}\label{pullback of orientor}
Let $M\overset{f}\rightarrow P\overset{g}\rightarrow N$, and suppose that $f$ is relatively oriented with relative orientation $\mO^f$. Let $K,R$ be $\Z/2$-bundles over $P,N$, respectively. Let $G$ be a \orientor{g} of $K$ to $R$. The \textbf{pullback of $G$ by $(f,\mO^f)$} is the \orientor{g\circ f} of $f^*K$ to $R$, given by
\[
\expinv{(f,\mO^f)}G=(-1)^{fG}G\bu \left(\phi^{\mO^f}\right)^{K},
\]
where $\phi^{\mO^f}$ is the orientor from Definition \ref{orientation as orientor}.
If $f$ is a local diffeomorphism, we write
\[\expinv{f}G=\expinv{(f,\mO^f_c)}G=G\bu \left(\phi_f\right)^K.
\]
\end{definition}
\begin{lemma}
\label{pullback by composition relatively oriented functoriality lemma}
Let $M\xrightarrow{f}P\xrightarrow{g}N\xrightarrow{h}L$ be maps, let $\mO^f,\mO^g$ be relative orientations of $f,g,$ respectively, and let $H$ be an \orientor{h}. Then
\[
\expinv{(f,\mO^f)}\expinv{(g,\mO^g)}H=\expinv{(g\circ f,\mO^g\circ \mO^f)}H.
\]
\end{lemma}
\begin{proof}
This follows immediately from Definition~\ref{pullback of orientor} and Lemma~\ref{composition of relative orientation orientors lemma}.
\end{proof}
\subsubsection{Orientors restoration}
\begin{lemma}\label{orientor invertion lemma}
Let $M\xrightarrow{f}P\xrightarrow{g}N$, and let $Q,K,R$ be $\Z/2$-bundles over $M,P,N$, respectively. Let $H$ be an \orientor{g\circ f} of $Q$ to $R$, and $G$ be an \orientor{g} of $K$ to $R$. Assume $G$ is invertible as a bundle map
\[
G:K\to \zcort{g}\otimes g^*R.
\]
Then there exists a unique \orientor{f} $F$ of $Q$ to $K$ satisfying
\[
H=G\bu F.
\]
\end{lemma}
\begin{proof}[Proof of Lemma \ref{orientor invertion lemma}]
Assume such $F$ exists. Then
\[
\left(\Id_{\zcort{f}}\otimes f^*G\right)\circ F=H.
\]
Since $G$ is invertible, this equation defines $F$ uniquely.
\end{proof}
\begin{lemma}\label{orientor on pullback is pulled back}
Let $M\xrightarrow{f}P\xrightarrow{g}N$, and assume that $f$ is relatively oriented, surjective submersion with connected fibers and with relative orientation $\mO^f$. Let $K,R$ be $\Z/2$-bundles over $P,N$, respectively. Set $h=g\circ f$ and let $H$ be an \orientor{h} of $f^*K$ to $R$. Then there exists a unique \orientor{g} of $K$ to $R$, denoted by $G$ such that 
\[
H=\expinv{(f,\mO^f)}G.
\]
\end{lemma}
The connectedness of the fibers of $f$ in the previous lemma is necessary, as is shown in the following examples.
\begin{example}
In the situation of Lemma~\ref{orientor on pullback is pulled back}, let $M=P=N=S^1$ modeled as the group of unit elements in $\C^*$. Let $f(z)=z^2$ and $g(z)=z$. Let $R=\underline{\Z/2}$ and let $K$ be a nontrivial double-cover of $P$. Since $f$ is a local diffeomorphism, it is relatively oriented. Then $f^*K$ is trivial, $\zcort{f\circ g}$ is trivial and $f^*g^*R$ is trivial. Therefore, there exists an \orientor{g\circ f} of $f^*K$ to $R$. However, there does not exist a \orientor{g} of $K$ to $R$ since $K$ is homeomorphic to $S^1$ while $R$ is homeomorphic to two copies of $S^1$.
\end{example}
\begin{example}\label{connected fiber necessity example points}
In the situation of Lemma~\ref{orientor on pullback is pulled back}, let $M=\{0,1\}, P=N=\{0\}$. Let $R=K=\Z/2$. The map $pt:M\to P$ is a local diffeomorphism. Let 
\[
F:\underline{\Z/2}\to \zcort{pt}\otimes \underline{\Z/2}
\]
be the map that over a point $x\in M$ sends
\[
k\mapsto k+x.
\]
Then under any identification $\zcort{pt}\simeq\Z/2$, the set map $F$ has two fixed points, while
$\expinv{(pt)}G$ has either four fixed points or no fixed points, for any equivariant map $G:\Z/2\to \Z/2$. Therefore, $F$ cannot be written as $\expinv{(pt)}G$ for any \eorientor{\Id_{pt}} $G$ of $\Z/2$.
\end{example}
\begin{proof}[Proof of Lemma \ref{orientor on pullback is pulled back}]
Under the adjunction of the functors $f^*$ and $f_*$ of sheaves, the map $H$ corresponds to a map
\[
\tilde H:K\to f_*\left(\zcort{h}\otimes h^*R\right).
\]
By the composition isomorphism of Definition~\ref{composition isomorphism} and the projection formula for sheaves, 
\[
f_*\left(\zcort{h}\otimes h^*R\right)=f_*\zcort{f}\otimes \zcort{g}\otimes g^*R.
\]
However, since the fibers of $f$ are connected, $f_*\underline{\Z/2}=\underline{\Z/2}$ and thus the map $f_*\phi^{\mO^f}:\underline{\Z/2}\to f_*\zcort{f}$ is an isomorphism.
Therefore, there exists a unique \orientor{g} $G$ of $K$ to $R$ satisfying
\[
\tilde H=f_*\phi^{\mO^f}\otimes G.
\]
By the abovementioned adjunction, it satisfies
\[
H=\expinv{\left(f,\mO^f\right)}G.
\]
\end{proof}

\subsubsection{Pullback along a fiber-product}
\begin{definition}\label{pullback of orientor by pullback-diagram}
Let
\[
\begin{tikzcd}
M\times_NP\ar[r,"r"]\ar[d,"q"]&P\ar[d,"g"]\\M\ar[r,"f"]&N
\end{tikzcd}
\] be a pullback square. Let $K,R$ be $\Z/2$-bundles over $P,N$ respectively, and let $G$ be a \orientor{g} from $K$ to $R$. \textbf{The pullback of $G$ by $r$ over $f$} is the \orientor{q} of $r^*K$ to $f^*R$ given by the following composition.
\[
r^*K\xrightarrow{r^*G}r^*\zcort{g}\otimes r^*g^*R\xrightarrow{\pull{(r/f)} \otimes\Id}\zcort{q}\otimes q^*f^*R.
\]
It is denoted by $\expinv{(r/f)}G$.
\end{definition}
\begin{example}\label{trivial expinv is pullback}
Let $f:M\to N$ be a map, $Q,K$ be $\Z/2$-bundles over $N$ and let $G:Q\to K$ be a \orientor{\Id_N} of $Q$ to $K$. Consider the following pullback diagram.
\[
\begin{tikzcd}
M\ar[d,swap,"\Id_M"]\ar[r,"f"]&N\ar[d,"\Id_N"]\\M\ar[r,"f"]&N
\end{tikzcd}
\]
Then under the canonical isomorphism $\zcort{\Id_X}\simeq \underline{\Z/2}$ it holds that
\[
\expinv{(f/f)}G=f^*G.
\]
This follows from Lemma \ref{pullback of vertical local diffeomorphism lemma}.
\end{example}
\begin{lemma}\label{pullback of composition of orientors by relatively oriented}
Consider the following diagram, in which the square is a pullback square.
\[
\begin{tikzcd}
M\times_NP\ar[r,"r"]\ar[d,"q"]&P\ar[d,"g"]\\M\ar[r,"f"]&N\ar[r,"h"]&L
\end{tikzcd}
\]
Let $K,R,T$ be $\Z/2$-bundles over $P,N,L$, respectively. Assume $G:K\to \zcort{g}\otimes g^*R$ and $H:R\to \zcort{h}\otimes h^*T$ are orientors. Let $\mO^f$ be a relative orientation of $f$. Set $\mO^r:=\pull{(q/g)}\mO^f$. Then $\mO^r$ is a relative orientation of $r$ and it follows that the following diagram is commutative.
\[
\begin{tikzcd}
r^*K\ar[r,"\expinv{\left(r/f\right)}(G)"]
\ar[dr,swap,"\expinv{\left(r,\mO^{r}\right)}(H\bu G)"]&\zcort{q}\otimes q^*f^*R\ar[d,"\Id \otimes q^*\expinv{\left(f,\mO^{f}\right)}(H)"]\\
&\zcort{q}\otimes q^*\zcort{h\circ f}\otimes q^*f^*a^*T
\end{tikzcd}
\]
That is,
\begin{equation}
    \label{pullback of orientor by lifted diffeomorphisms}
\expinv{\left(r,\mO^{r}\right)}(H\bu G)=\expinv{\left(f,\mO^{f}\right)}(H)\bu \expinv{(r/f)}(G).
\end{equation}
\end{lemma}
\begin{proof}
Recalling Definition \ref{pullback of orientor}, this lemma amounts to
\[
(H\bu G)\bu \left(\phi^{\mO^r}\right)^{K}=(-1)^{fG}\left(H\bu \left(\phi^{\mO^f}\right)^{R}\right)\bu \expinv{(r/f)}(G).
\]Using the Associativity Lemma \ref{composition of orientors is associative}, it is enough to prove that
\[
G\bu \left(\phi^{\mO^r}\right)^{K}=(-1)^{fG}\left(\phi^{\mO^f}\right)^{R}\bu \expinv{(r/f)}(G).
\] To see that, let $k,t,\mO^g$ be sections of $K, R, \zcort{g}$, respectively, such that $G(k)=\mO^g\otimes g^*t$. Then
\begin{align*}
    G\bu \left(\phi^{\mO^r}\right)^K(r^*k)=\left(\Id \otimes r^*G\right)\left(\mO^r\otimes r^*k\right)=(-1)^{rG}\mO^r\otimes r^*\mO^g\otimes r^*g^*t,
\end{align*}
and
\begin{align*}
    \left(\phi^{\mO^f}\right)^R\bu \expinv{\left(r/f\right)}(G)( r^*k)&=\left(\Id \otimes q^*\phi^{\mO^f}\right)\left(\pull{(r/f)}\mO^{g}\otimes r^*g^*t\right)\\&=(-1)^{fg}\left(\pull{(r/f)}\mO^{g}\otimes q^*\mO^f\otimes r^*g^*t\right).
\end{align*}
However, by Lemma \ref{vertical pullback vs horizontal pullback} it follows that
\[
\mO^r\otimes r^*\mO^g=(-1)^{fg}\pull{(r/f)}\mO^{g}\otimes q^*\mO^f,
\]
and the lemma follows, since $\dim f=\dim r$.
\end{proof}
\begin{remark}\label{pullback of composition of orientors by local diffeomorphism}
In the setting of Lemma \ref{pullback of composition of orientors by relatively oriented}, if $f$ is a local diffeomorphism, then $r$ is a local diffeomorphism, and
\[
\expinv{r}(H\bu G)=\expinv{f}(H)\bu \expinv{\left(r/f\right)}G.
\]
\end{remark}
\begin{example}
\label{expinv by diffeomorphism over Id}
Let $M\xrightarrow{f}P\xrightarrow{g}N$ be maps, $K,R$ be $\Z/2$-bundles over $P,N$, respectively, and let $G$ be a \orientor{g} of $K$ to $R$. Assume $f$ is a diffeomorphism. Consider the following pullback diagram. 
\[
\begin{tikzcd}
M\ar[r,"f"]\ar[d,"g\circ f"]&P\ar[d,"g"]\\N\ar[r,"\Id"]&N
\end{tikzcd}
\]
By applying Remark \ref{pullback of composition of orientors by local diffeomorphism} to the preceding diagram extended with $h=\Id_N$ and $H=\Id_R$, we get
\[
\expinv{(f/\Id)}G=\expinv{f}G.
\]
\end{example}
\begin{lemma}\label{commutativity vertical pullback vs horizontal pullback orientors}
Consider the following fiber-product diagram.
\[
\begin{tikzcd}
M\times_NP\ar[r,"r"]\ar[d,"q"]&P\ar[d,"g"]\\M\ar[r,"f"]&N.
\end{tikzcd}
\]
Let $Q,K,R,T$ be $\Z/2$-bundles over $N$.
Let $F$ be \orientor{f} of $f^*Q$ to $R$ and let $G$ be a \orientor{g} of $g^*K$ to $T$.
Then we have the following equality of \orientor{(f\circ q)}s of $(f\circ q)^*\left(K\otimes Q\right)$ to $T\otimes R$.
\[
{}^TF\bu \expinv{(r/f)}G^{f^*Q}=(-1)^{FG}G^R\,\,\bu\,\, {}^{g^*K}\expinv{(q/g)}F
\]
\end{lemma}
\begin{proof}[Proof of Lemma \ref{commutativity vertical pullback vs horizontal pullback orientors}]
Let $\mO^f,\mO^g$ be local orientations of $f,g$, respectively. Let $q,k,r,t$ be local sections of $Q,K,R,T$, respectively. Assume \[F(f^*q)=\mO^f\otimes f^*r, G(g^*k)=\mO^g\otimes g^*t.\] Then
\begin{align*}
{}^TF\bu\expinv{(r/f)}G^{f^*Q}\left((f\circ q)^*(k\otimes q)\right)=&
\left(\Id_{\zcort{q}}\otimes  q^*\left({}^TF\right)\right)\left(\pull{(r/f)}\mO^g\otimes q^*f^*(t\otimes q)\right)\\
=&(-1)^{Fg+(F-f)t}\mO^f\circ \pull{(r/f)}\mO^g\otimes q^*f^*(t\otimes r).
\end{align*}
There is no sign from passing $\mO^f$ by $\mO^g$ by the definition of the composition isomorphism~\ref{composition isomorphism}.
On the other hand,
\begin{align*}
G^R\bu{}^{g^*K}\expinv{(q/g)}F\left((g\circ r)^*(k\otimes q)\right)=&(-1)^{(F-f)k}
\left(1_{\zcort{r}}\otimes  r^*\left(G^R\right)\right)\left(\pull{(q/g)}\mO^f\otimes r^*g^*(k\otimes r)\right)\\
=&(-1)^{Gf+(F-f)k}\mO^g\circ \pull{(q/r)}\mO^f\otimes r^*g^*(t\otimes r).
\end{align*}
By Lemma~\ref{vertical pullback vs horizontal pullback} we see that
\[
{}^TF\bu \expinv{(r/f)}G^{f^*Q}=(-1)^{s}G^R\,\,\bu\,\, {}^{g^*K}\expinv{(q/g)}F,
\]
with 
\begin{align*}
s=Fg+(F-f)t+Gf+(F-f)k+fg.
\end{align*}
However, 
$G+k\equiv_2 g+t,$
and thus, 
\begin{align*}
s\equiv_2Fg+(F-f)(G-g)+Gf+fg\equiv_2FG.
\end{align*}
\end{proof}
\begin{lemma}\label{horizontal expinv composition}
Assume that the squares in the following diagram are pullbacks.
\[
\begin{tikzcd}
M\times_NP\ar[r,"r"]\ar[d,"q"]&P\ar[r,"t"]\ar[d,"g"]&T\ar[d,"b"]\\
M\ar[r,"f"]&N\ar[r,"a"]&K
\end{tikzcd}
\]
Then the rectangle is also a pullback diagram, and
\[
\expinv{\left((t\circ r)/(a\circ f)\right)}=\expinv{(r/f)}\circ r^*\expinv{(t/a)}.
\]
\end{lemma}
\begin{proof}[Proof of Lemma \ref{horizontal expinv composition}]
This is a consequence of Lemma \ref{horizontal pullback composition}.
\end{proof}
\begin{lemma}\label{vertical pullback expinv composition}
Assume that the squares in the following diagrams are pullbacks.
\[
\begin{tikzcd}
M\times_NP\ar[r,"r"]\ar[d,"q"]&P\ar[d,"g"]\\
M\ar[r,"f"]\ar[d,"t"]&N\ar[d,"a"]\\
T\ar[r,"b"]&K   
\end{tikzcd}
\]
Then the rectangle is also a pullback diagram. Let $Q,R,S$ be $\Z/2$-bundles over $P,N,K,$ respectively. Assume $G$ is a \orientor{g} of $Q$ to $R$, and $A$ is a \orientor{a} of $R$ to $S$. Then
\[
\expinv{(r/b)}(A\bu G)=\left(\expinv{(f/b)}A\right)\bu \left(\expinv{(r/f)}G\right).
\]
\end{lemma}
\begin{proof}[Proof of Lemma \ref{vertical pullback expinv composition}]
This is a consequence of Lemma \ref{vertical pullback composition}.
\end{proof}
\subsubsection{Pullback along a relatively oriented commutative diagram}
\begin{definition}\label{pullback along a relatively oriented commutative diagram definition}
Let 
\[
\begin{tikzcd}
O\ar[r,"r"]\ar[d,"q"]&P\ar[d,"g"]\\M\ar[r,"f"]&N
\end{tikzcd}
\]
be a commutative diagram such that $r,f$ are relatively oriented, with relative orientations $\mO^r,\mO^f$. Let $K,R$ be $\Z/2$-bundles over $P,N,$ respectively, and let $G$ be a \orientor{g} from $K$ to $R$. \textbf{The pullback of $G$ by $r$ over $f$} is the unique \orientor{q} of $r^*K$ to $f^*R$, denoted $\expinv{(r/f)}G$, given by the following equation.
\[
\expinv{(r,\mO^r)}G=\left(\phi^{\mO^f}\right)^{R}\bu \expinv{\left(r/f\right)}G.
\]Its existence and uniqueness are guaranteed by Lemma~\ref{orientor invertion lemma}.
\end{definition}
\begin{remark}\label{enhanced pullback alternative equivalent definition remark}
Recall the isomorphism $\pull{(r/f)}:r^*\zcort{g}\to \zcort{q}$ from Definition~\ref{enhanced pullback fiber product orientation convention}.
It is easy to check by hand that the composite map
\[
r^*K\xrightarrow{r^*G}r^*\zcort{g}\otimes r^*g^*R\xrightarrow{\pull{(r/f)}\otimes1}\zcort{q}\otimes q^*f^*R
\]
satisfies the defining property of $\expinv{(r/f)}G$.
\end{remark}
\begin{remark}
In addition to the above remark, if the diagram is a fiber-product and $\mO^r=\pull{\left(q/g\right)}\mO^f$, then both definitions of $\expinv{\left(r/f\right)}G$ coincide. This follows by the uniqueness in the above definition and Lemma~\ref{pullback of composition of orientors by relatively oriented} applied to the following diagram,
\[
\begin{tikzcd}
O\ar[r,"r"]\ar[d,"q"]&P\ar[d,"g"]\\M\ar[r,"f"]&N\ar[r,"\Id_N"]&N
\end{tikzcd}
\]
noting that $\expinv{(f,\mO^f)}\Id_{R}=\left(\phi^{\mO^f}\right)^R.$
\end{remark}
\begin{lemma}\label{pullback of composition of orientors by relatively oriented enhanced}
Consider the following commutative diagram.
\[
\begin{tikzcd}
O\ar[r,"r"]\ar[d,"q"]&P\ar[d,"g"]\\M\ar[r,"f"]&N\ar[r,"h"]&L
\end{tikzcd}
\]
Let $\mO^f,\mO^r$ be relative orientations of $f,r$.
Let $K,R,T$ be $\Z/2$-bundles over $P,N,L$, respectively. Assume $G:K\to \zcort{g}\otimes g^*R$ and $H:R\to \zcort{h}\otimes h^*T$ are orientors. Let $\mO^f$ be a relative orientation of $f$. Set $\mO^r:=\pull{(q/g)}\mO^f$. Then
\begin{equation}
\label{pullback of orientor by lifted diffeomorphisms enhanced}
\expinv{\left(r,\mO^{r}\right)}(H\bu G)=(-1)^{(r-f)H}\expinv{\left(f,\mO^{f}\right)}(H)\bu \expinv{(r/f)}(G).
\end{equation}
\end{lemma}
\begin{lemma}\label{horizontal expinv composition enhanced}
Assume that the squares in the following diagram are commutative, and that the horizontal arrows are relatively oriented, with relative orientations $\mO^r,\mO^f, \mO^t,\mO^a$.
\[
\begin{tikzcd}
O\ar[r,"r"]\ar[d,"q"]&P\ar[r,"t"]\ar[d,"g"]&T\ar[d,"b"]\\
M\ar[r,"f"]&N\ar[r,"a"]&K
\end{tikzcd}
\]
Then
\[
\expinv{\left((t\circ r)/(a\circ f)\right)}=(-1)^{a(r-f)}\expinv{(r/f)}\circ r^*\expinv{(t/a)}.
\]
\end{lemma}
\begin{proof}
This is a consequence of the uniqueness in Definition~\ref{pullback along a relatively oriented commutative diagram definition}, applying Lemma~\ref{composition of relative orientation orientors lemma}.
\end{proof}

\subsubsection{Base extension}
\begin{definition}\label{pulled orientation orientor}
Let $\W$ be a manifold with corners. Let $A,X,Y$ be orbifolds with corners over $\W$ and let $f:X\to Y$ be a map over $\W$. Let $Q,K$ be $\Z/2$-bundles over $X,Y$, respectively, and let $F$ be a \orientor{f} of $Q$ to $K$. Consider the following Cartesian diagrams.
\[
\begin{tikzcd}
A\times_\W X\ar[d,swap,"\Id\times f"]\ar[r,"\pi^{A\times X}_X"]&X\ar[d,"f"]\\A\times_\W Y\ar[r,swap,"\pi^{A\times Y}_Y"]&Y
\end{tikzcd} \qquad
\begin{tikzcd}
X\times_\W A\ar[d,swap," f\times\Id"]\ar[r,"\pi^{X\times A}_X"]&X\ar[d,"f"]\\Y\times_\W A\ar[r,swap,"\pi^{Y\times A}_Y"]&Y
\end{tikzcd}
\]
We denote the pullbacks
\[
F^A:=\expinv{(\pi^{A\times X}_X/\pi^{A\times Y}_Y)}F,\qquad {}^AF:=\expinv{(\pi^{X\times A}_X/\pi^{Y\times A}_Y)}F.
\]
\end{definition}
\begin{lemma}\label{base extension is multiplicative expinv}
Let $\W$ be a parameter space, let $A,B,X,Y$ be orbifolds with corners over $\W$, and let $f:X\to Y$ be a map over $\W$.
If $Q,K$ are $\Z/2$-bundles over $X,Y$, respectively, $F$ is an \orientor{f} of $Q$ to $K$, then
\begin{align*}
F^{A\times_\W B}&=\left(F^{B}\right)^A,\\
{}^{A\times_\W B}F&={}^B\left({}^AF\right).
\end{align*}
\end{lemma}
\begin{proof}[Proof of Lemma \ref{base extension is multiplicative expinv}]
This is an immediate consequence of Lemma \ref{horizontal expinv composition} applied to the following diagrams of fiber-products.
\[
\begin{tikzcd}
A\times_\W B\times_\W X\ar[r,"\pi_{B\times X}"]\ar[d,swap,"\Id\times\Id\times f"]&B\times_\W X\ar[r,"\pi_X"]\ar[d,"\Id\times f"]&X\ar[d,"f"]\\
A\times_\W B\times_\W Y\ar[r,swap,"\pi_{B\times Y}"]&B\times_\W Y\ar[r,swap,"\pi_Y"]&Y
\end{tikzcd}\qquad \begin{tikzcd}
X\times_\W A\times_\W B\ar[r,"\pi_{ X\times A}"]\ar[d,swap,"f\times\Id\times\Id"]& X\times_\W A\ar[r,"\pi_X"]\ar[d,"f\times\Id"]&X\ar[d,"f"]\\
Y\times_\W A\times_\W B\ar[r,swap,"\pi_{Y\times A}"]&Y\times_\W A\ar[r,swap,"\pi_Y"]&Y
\end{tikzcd}
\]
\end{proof}
\begin{lemma}\label{base extension is distributive expinv}
Let $A,X,Y,Z$ be orbifolds with corners over a parameter space $\W$, and $X\overset{f}\to Y\overset{g}\to Z$ be maps over $\W$. 
If $Q,K,R$ are $\Z/2$-bundles over $X,Y,Z$, respectively, $F$ is an \orientor{f} of $Q$ to $K$ and $G$ is a \orientor{g} of $K$ to $R$, then
\begin{align*}
\left(G\bu F\right)^A&=G^A\bu F^A,\qquad
{}^A\left(G\bu F\right)={}^AG\bu \,\,{}^AF.    
\end{align*}
\end{lemma}
\begin{proof}[Proof of Lemma \ref{base extension is distributive expinv}]
This is an immediate consequence of Lemma \ref{vertical pullback expinv composition} applied to the following diagrams of fiber-products.
\[
\begin{tikzcd}
A\times_\W X\ar[d,swap,"\Id\times f"]\ar[r,"\pi_X"]&X\ar[d,"f"]\\A\times_\W Y\ar[d,swap,"\Id\times g"]\ar[r,"\pi_Y"]&Y\ar[d,"g"]\\A\times_\W Z\ar[r,swap,"\pi_Z"]&Z
\end{tikzcd}\qquad
\begin{tikzcd}
X\times_\W A\ar[d,swap,"f\times\Id"]\ar[r,"\pi_X"]&X\ar[d,"f"]\\Y\times_\W A\ar[d,swap,"g\times\Id"]\ar[r,"\pi_Y"]&Y\ar[d,"g"]\\Z\times_\W A\ar[r,swap,"\pi_Z"]&Z
\end{tikzcd}
\]
\end{proof}
\subsubsection{Base extension and relatively oriented pullback}
\begin{lemma}\label{pulled orientation orientor is pullback orientor of orientation}
In the setting of Definition \ref{pulled orientation orientor}, if $\mO^f$ is a relative orientation of $f$, then $1\times\mO^f$ and $\mO^f\times 1$ are relative orientations for $\Id\times f$ and $f\times\Id$, respectively, and
\[
\phi^{1\times \mO^f}=\left(\phi^{\mO^f}\right)^A,\qquad \phi^{\mO^f\times 1}={}^A\left(\phi^{\mO^f}\right).
\]
\end{lemma}
\begin{remark}\label{pulled orientation orientor is pullback orientor of orientation local diffeomorphism}
In the setting of Definition \ref{pulled orientation orientor}, if $f$ is a local diffeomorphism, then $\Id\times f$ and $f\times\Id$ are local diffeomorphisms and
\[
\phi_{\Id\times f}=\left(\phi_f\right)^A,\qquad \phi_{f\times\Id}={}^A\left(\phi_f\right).
\]
This follows at once from Lemmas \ref{pulled orientation orientor is pullback orientor of orientation}, \ref{pullback to product of vertical local diffeomorphism  remark}.
\end{remark}
\begin{proof}[Proof of Lemma \ref{pulled orientation orientor is pullback orientor of orientation}]
For the first equation, by definition, both sides are given by
\[
\underline{\Z/2}\ni\underline 1\mapsto 1\times\mO^f = \pull{\left(\pi^{A\times X}_X/\pi^{A\times Y}_Y\right)}\mO^f\in \zcort{1\times \mO^f}.
\]
For the second equation, by definition, both sides are given by
\[
\underline{\Z/2}\ni\underline 1\mapsto \mO^f\times 1 = \pull{\left(\pi^{X\times A}_X/\pi^{Y\times A}_Y\right)}\mO^f\in \zcort{\mO^f\times 1}.
\]
\end{proof}
\begin{lemma}\label{base extension of pullback is extended pullback of base extension lemma}
With the setting of Lemma \ref{base extension is distributive expinv}, assume $\mO^f$ is a relative orientation of $f$. The $1\times \mO^f$ and $\mO^f\times 1$ are relative orientations of $\Id\times f$ and $f\times\Id,$ respectively. Let $G$ be a \orientor{g}. Then
\[
\left(\expinv{\left(f,\mO^f\right)}G\right)^A=\expinv{\left(\Id\times f,1\times\mO^f\right)}\left(G^A\right), \qquad {}^A\left(\expinv{\left(f,\mO^f\right)}G\right)=\expinv{\left( f\times\Id,\mO^f\times 1\right)}\left({}^AG\right),
\]
\end{lemma}
\begin{proof}[Proof of Lemma \ref{base extension of pullback is extended pullback of base extension lemma}]
We show the first equality. The second equality is left to the reader. First, observe,
\begin{equation}
\label{base extension bundle extension commutativity}
\left(\left(\phi^{\mO^f}\right)^{K}\right)^A=\left(\left(\phi^{\mO^f}\right)^A\right)^{\pi_Y^*K}\overset{\text{Lem. \ref{pulled orientation orientor is pullback orientor of orientation}}}=\left(\phi^{1\times \mO^f}\right)^{\pi_Y^*K}.
\end{equation}
Therefore, the lemma is obtained by the following calculation.
\begin{align*}
(\expinv{\left(f,\mO^f\right)}G)^A\overset{\text{Def. \ref{pullback of orientor}}}=&(-1)^{fG}\left(G\bu \left(\phi^{\mO^f}\right)^K\right)^A\\
\overset{\text{Rem. \ref{base extension is distributive expinv}}}=&(-1)^{fG}G^A\bu \left(\left(\phi^{\mO^f}\right)^{K}\right)^A
\\
\overset{\text{eq. \eqref{base extension bundle extension commutativity}}}=&(-1)^{fG}G^A\bu \left(\phi^{1\times \mO^f}\right)^{\pi_Y^*K}
\\
\overset{\text{Def. \ref{pullback of orientor}}}=&\expinv{\left(\Id\times f,1\times\mO^f\right)}(G^A).
\end{align*}
\end{proof}
\begin{lemma}\label{pullback of pullback by diffeomorphism of orientor}
With the setting of Lemma \ref{base extension is distributive expinv}, assume $f$ is a local diffeomorphism. Then, $\Id\times f$ and $f\times\Id$ are local diffeomorphisms. Let $G$ be a \orientor{g}. Then,
\[
\left(\expinv{f}G\right)^A=\expinv{(\Id\times f)}\left(G^A\right), \qquad {}^A\left(\expinv{f}G\right)=\expinv{(f\times\Id)}\left({}^AG\right).
\]
\end{lemma}
\begin{proof}[Proof of Lemma \ref{pullback of pullback by diffeomorphism of orientor}]
This is an immediate consequence of Lemma~\ref{base extension of pullback is extended pullback of base extension lemma} and Lemma~\ref{pullback of vertical local diffeomorphism lemma}.
\end{proof}
The proof of the following lemma is similar to that of Lemma~\ref{base extension of pullback is extended pullback of base extension lemma} and thus is omitted.
\begin{lemma}\label{base extension of pullback is pullback by base extension orientors lemma}
Consider the following diagram in which all faces are fiber products.
\[\begin{tikzcd}
	{A\times_\W M\times_NP} && {A\times_\W P} \\
	& {M\times_N P} && P \\
	{A\times_\W M} & {} & {A\times_\W N} \\
	& M && N
	\arrow["g", from=2-4, to=4-4]
	\arrow["f"', from=4-2, to=4-4]
	\arrow[from=3-1, to=4-2]
	\arrow[from=1-1, to=3-1]
	\arrow["q"'{pos=0.2}, from=2-2, to=4-2]
	\arrow[from=3-3, to=4-4]
	\arrow[from=3-1, to=3-3]
	\arrow[from=1-1, to=1-3]
	\arrow["r"{pos=0.2}, from=2-2, to=2-4]
	\arrow[from=1-1, to=2-2]
	\arrow[from=1-3, to=2-4]
	\arrow[from=1-3, to=3-3]
\end{tikzcd}\]
Let $G$ be an \orientor{g}. Then
\[
\left(\expinv{(r/f)}G\right)^A=\expinv{\left(\Id_A\times r/\Id_A\times f\right)}\left(G^A\right).
\]
\end{lemma}
\subsubsection{Quotient orientors}
For this section, fix an oriented Lie group $G$ with orientation $\mO^G$, and an orbifold with corners $M$ with a right free proper $G$ action $\Phi:M\times G\to M$. Denote by $q:M\to M/G$ the quotient map and by $\pi:M\times G\to M$ the projection. Denote by $\mO^{\Phi}_c,\mO_c^ {q},\mO_c^\pi$ the relative orientations of $\Phi,q,\pi$ from Definition~\ref{action quotient orientation definition} with respect to $\mO^G$. Let $G',\mO^{G'},M',\Phi',q',\pi'$, etc. be another collection as the above. Let $\phi:G\to G'$ be a homomorphism of groups.
\begin{definition}\label{Invariant orientor group action definition}
With the above setting, a map $h:M\to M'$ is called \textbf{equivariant} if 
\[h\circ \Phi=\Phi'\circ (h\times \phi).\]
That is, the following diagram is commutative.
\[
\begin{tikzcd}
M\times G\ar[r,"\Phi"]\ar[d,"h\times \phi"]&M\ar[d,"h"]\\M'\times G'\ar[r,"\Phi'"]&M'
\end{tikzcd}
\]
Consider also the following commutative diagram.
\[
\begin{tikzcd}
M\times G\ar[r,"\pi"]\ar[d,"h\times \phi"]&M\ar[d,"h"]\\M'\times G'\ar[r,"\pi'"]&M'
\end{tikzcd}
\]
For a $G$-equivariant map $h:M\to M'$ and bundles $K,R$ over $M/G,M'/G'$, respectively, an \orientor{h} $H$ of $ q^*K$ to $q'^*R$ is said to be \textbf{$G-$equivariant} if 
\[
\expinv{\left(\Phi/\Phi'\right)}H=\expinv{\left(\pi/\pi'\right)}H.
\]
\end{definition}
\begin{remark}\label{orientor from quotient pulled by quotient is equivariant remark}
An equivariant map $h:M\to M'$ descends to the quotient $\overline h:M/G\to M'/G'$. That is, there exists a unique map $\overline h:M/G\to M'/G'$ such that $q'\circ h=\overline h\circ  q.$ In diagrammatic form, the following diagram is commutative.
\[
\begin{tikzcd}
M\ar[r,"q"]\ar[d,"h"]&M/G\ar[d,"\overline h"]\\M'\ar[r,"q'"]&M'/G'
\end{tikzcd}
\]
Let $K,R$ be bundles over $M/G,M'/G',$ respectively. Let ${\overline H}$ be an \orientor{\overline h} of $q^*K$ to $q'^*R$. Then $\expinv{\left(q/q'\right)}{\overline H}$ is $G$-equivariant.
Indeed, we have
\begin{align*}
\expinv{\left(\Phi/\Phi'\right)}\expinv{\left(q/q'\right)}{\overline H}&\overset{\text{Lem. \ref{horizontal expinv composition enhanced}}}=(-1)^{G'(G-G')}\expinv{\left((q\circ\Phi)/(q'\circ\Phi')\right)}{\overline H}\\&\qquad\,\,=(-1)^{G'(G-G')}\expinv{\left((q\circ \pi_M)/(q'\circ\pi_{M'})\right)}{\overline H}\overset{\text{Lem. \ref{horizontal expinv composition enhanced}}}=\expinv{\left(\pi/\pi'\right)}\expinv{\left(q/q'\right)}{\overline H}.
\end{align*}
The fact that every $G$-equivariant orientor is of the form $\expinv{(q/q')}{\overline H}$ for some ${\overline H}$ is formalized in the following lemma.
\end{remark}
\begin{lemma}\label{orientor restoration group action lemma}
Let $h:M\to M'$ be a $G-$invariant map, let $K,R$ be $\Z/2$ bundles over $M/G,M'/G',$ respectively, and let $H$ be a $G-$equivariant \orientor{h} of $q^*K$ to $q'^*R$. Let $\overline h:M/G\to M'/G'$ be the map satisfying $q'\circ h=\overline h\circ  q.$ There exists a unique \orientor{\overline h} ${\overline H}$ of $K$ to $R$ satisfying
\[
H=\expinv{(q/q')}{\overline H}.
\]
\end{lemma}
\begin{proof}
The proof is similar to that of Lemma~\ref{orientor on pullback is pulled back}. For a sheaf $\mathcal F$ with a $G$-action, denote by $\mathcal F^G\subset \mathcal F$ its invariants. Since $h$ is equivariant, $G$ acts on $q_*\left(\zcort{h}\otimes h^*q'^*R\right)$. Since $H$ is equivariant, under the adjunction of $q^*$ and $q_*$ it corresponds to a map
\[
\tilde H:K\to \left(q_*\left(\zcort{h}\otimes h^*q'^*R\right)\right)^G=\left(q_*\left(\zcort{h}\otimes q^*\overline h^*R\right)\right)^G.
\]
By the composition isomorphism of Definition~\ref{composition isomorphism} and the projection formula for sheaves,
\[
\left(q_*\left(\zcort{h}\otimes q^*\overline h^*R\right)\right)^G=\left(q_*\zcort{h}\otimes \overline h^*R\right)^G=\left(q_*\zcort{h}\right)^G\otimes\overline h^*R.
\]
Consider the map $(q/q')_\star:\zcort{\overline h}\to \left(q_*\zcort{h}\right)^G$ corresponding to 
$\pull{\left(q/q'\right)}:q^*\zcort{\overline h}\to \zcort{h}$ under the adjunction of $q^*$ and $q_*$. Since $G$ acts transitively on the fibers of $q$, $(q/q')_\star$ is an isomorphism.
Denote by $\overline H$ the unique \orientor{\overline h} of $K$ to $R$ satisfying
\[
\tilde H=\left((q/q')_\star\otimes 1\right)\circ \overline H.
\]
By the naturality of the adjunction it follows that
\[
H=\left(\pull{(q/q')}\otimes 1\right)\circ q^*\overline H.
\]
Recalling Remark~\ref{enhanced pullback alternative equivalent definition remark}, this is exactly
\[
H=\expinv{(q/q')}\overline H.
\]
\end{proof}

\subsubsection{Boundary orientors}
\begin{definition}\label{boundary-operator for relative orientation}
The \textbf{boundary orientor} is the \eorientor{\iota_f} of $\underline{\Z/2}$
\[
\pa_f:\underline{\Z/2}\to\zcort{\iota_f}\otimes \underline{\Z/2}
\]
given by
\[
\pa_f(1)=(-1)^f\mO_c^{\iota_f}.
\]
\end{definition}
\begin{remark}\label{boundary orientor is contraction by minus out pointing vector}
The composition of ${\pa_f}^{\zcort{f}}$ and the composition isomorphism,
\[
{\iota_f}^*\zcort{f}\overset{{\pa_f}^{\zcort{f}}}\to\zcort{\iota_f}\otimes{\iota_f}^*\zcort{f} \overset{comp.}=\zcort{f\circ\iota_f},
\]
is given by
\[
\mO^f\mapsto(-1)^{f}\mO^f\circ\mO_c^{\iota_f}.
\]
By abuse of notation, we often denote this composition by ${\pa_f}^{\zcort{f}}.$ Explicitly, it is given by the contraction with $-\nu_{out}$ on the right. Indeed, 
if $x_1,...,x_{m-1}\in~T_p\pa_fM$ form a basis, then Remark \ref{composition with canonical orientation of boundary} gives
\[
\pa_f^{\zcort{f}}\mO^f(x_1\we ...\we x_{m-1})=-\mO^f({(d\iota_f)}_px_1\we ...\we {(d\iota_f)}_p x_{m-1}\we \nu_{out}).
\]
\end{remark}
\begin{definition}\label{restriction of orientor to the boundary}
Let $M\overset{f}\to P\overset{g}\to N$ be such that $g\circ f$ is a surjective submersion. Assume that $f$ factorizes through the boundary of $g$ in the sense of Definition~\ref{boundary factorization}. That is, we have the following pullback diagram.
\[
\begin{tikzcd}
\pa_{g\circ f}M\ar[r,"\iota_{g\circ f}"]\ar[d,swap,dashed, "{\iota_g}^* f"]&M\ar[d,"f"]\\
\pa_g P\ar[r,swap,"\iota_g"]&P
\end{tikzcd}
\]
Let $Q,K$ be $\Z/2$-bundles over $M,P$ respectively, and let $F$ be a \orientor{f} from $Q$ to $K$. The \textbf{restriction of $F$ to the boundary} is $\expinv{\left(\iota_{g\circ f}/\iota_g\right)}F$, that is, the \orientor{{\iota_g}^*f} of $\iota_{g\circ f}^*Q$ to $\iota_g^*K$
\[
\iota_{g\circ f}^*Q\overset{\iota_{g\circ f}^*F}{\xrightarrow{\hspace*{1cm}}} \iota_{g\circ f}^*\zcort{f}\otimes \iota_{g\circ f}^*f^*K\overset{\pull{\left(\iota_{g\circ f}/\iota_g\right)}\otimes\Id}{\xrightarrow{\hspace*{2cm}}}\zcort{{\iota_g}^*f}\otimes {({\iota_g}^*f)}^*\iota_g^*K.
\]
\end{definition}
\begin{definition}\label{Boundary of orientor}
Let $M,N,g,Q,K,G$ be as in Definition \ref{orientor}. Recall that the \eorientor{\iota_g}
${\pa_g}^Q:\iota_{g}^*Q\to \zcort{\iota_g}\otimes \iota_g^*Q$ is the $Q$-extension from Definition \ref{tensored orientor extension} of the boundary orientor from Definition \ref{boundary-operator for relative orientation}. The \textbf{boundary of $G$} is the \orientor{\zcort{g\circ\iota_g}} of $\iota_g^*Q$ to $K$,
\[
\pa G=(-1)^{|G|}G\bu {\pa_g}^Q.
\]
\end{definition}
\begin{remark}
Recall that the degree of the boundary operator is $|\pa_g|=1$. Thus, if the degree of $G$ is $|G|$, then the degree of $\pa G$ is $|G|+1$.
\end{remark}
\begin{example}[Boundary of orientation]
\label{boundary of orientation as orientor}
Let $(M,\mO^M)$ be an oriented orbifold with corners. Recall the \eorientor{M} of $\underline{\Z/2}$ denoted by $\phi^{\mO^M}$ from Definition \ref{orientation as orientor}.
Then 
\[
\pa\phi^{\mO^M}=\phi^{\pa_{M}^{\zcort{M}} \left(\mO^M\right)}=\pa_M^{\zcort{M}}\circ \phi^{\mO^M},
\]
as \eorientor{\pa M}s of $\underline{\Z/2}$. 
Indeed, acting on the section $\underline{1}$ of $\underline{\Z/2}$ we have
\begin{multline*}
\pa\phi^{\mO^M}(\underline{1})=(-1)^{M}\left(\Id_{\zcort{\iota}}\otimes\phi^{\mO^M}\right)\left(\pa_M(\underline{1})\right)\\=\left(\Id_{\zcort{\iota}}\otimes\phi^{\mO^M}\right)\left(\mO_c^{\iota_M}\otimes \underline{1}\right)\overset{Kozsul}=(-1)^{M}\mO_c^{\iota_M}\otimes \mO^M,
\end{multline*}
and
\[
\pa_M^{\zcort{M}}\circ\phi^{\mO^M}(\underline{1})=\pa_M^{\zcort{M}}(\mO^M)=(-1)^M\mO^{\iota_M}_c\otimes\mO^M.
\]
\end{example}
\begin{lemma}\label{boundary of composition orientor}
Let $M\overset{f}\to P\overset{g}\to N$ be such that $g\circ f$ is a surjective submersion and let $Q,K,R$ be {$\Z/2$~bundles} over $M,P,N$, respectively. Let $F:Q\to \zcort{f}\otimes f^*K$ be a \orientor{f} of $Q$ to $K$ and $G:K\to \zcort{g}\otimes g^*R$ be a \orientor{g} of $K$ to $R$. Assume that $f$ factorizes through the boundary of $g$ as ${\iota_g}^*f$, as in Definition \ref{boundary factorization}. 
Then we have an equality of \orientor{g\circ f\circ \iota_{g\circ f}} of $\iota^*_{g\circ f}Q$ to $R$,
\[
\iota_{g\circ f}^*Q\longrightarrow \zcort{ g\circ f\circ \iota_{g\circ f}}\otimes {\left(g\circ f\circ\iota_{g\circ f}\right)}^*R
\]
as follows.
\begin{equation}
\pa(G\bu F)=\pa G\bu \expinv{\left(\iota_{g\circ f}/\iota_g\right)}F
\label{boundary of composition orientor equation}   
\end{equation}
\end{lemma}
The proof of Lemma \ref{boundary of composition orientor} requires the following remark.
\begin{remark}
\label{boundary factorization orientation diagram}
Let $M\overset{f}\to P\overset{g}\to N$ be such that $g\circ f$ is a surjective submersion. Assume that $f$ factorizes through the boundary of $g$ in the sense of Definition~\ref{boundary factorization}. That is, we have the following pullback diagram.
\[
\begin{tikzcd}
\pa_{g\circ f}M\ar[r,"\iota_{g\circ f}"]\ar[d,swap,dashed, "{\iota_g}^* f"]&M\ar[d,"f"]\\
\pa_g P\ar[r,swap,"\iota_g"]&P
\end{tikzcd}
\]
With the notion of the boundary orientor, Lemma \ref{boundary factorization orientation} can be rephrased as the commutativity of the following diagram.
\begin{equation}
\begin{tikzcd}
&\iota_{g\circ f}^*\zcort{f}\otimes \underline{\Z/2}\ar[dl,swap,"\pull{\left(\iota_{g\circ f}/\iota_g\right)}\otimes {({\iota_g}^*f)}^*{\pa_g}"]\ar[dr,"\pa_{g\circ f}^{\zcort{ f}}"]\\
\zcort{{\iota_g}^*f}\otimes {({\iota_g}^*f)}^*\zcort{\iota_g}\ar[rd,equal,swap,"comp."]&&
\zcort{\iota_{g\circ f}}\otimes \iota_{g\circ f}^*\zcort{f}\ar[ld,equal,"comp."]
\\
&\zcort{\iota_g\circ \iota_g^*f}=\zcort{f\circ \iota_{g\circ f}}
\end{tikzcd}
\end{equation}
That is,
\begin{equation}
\label{boundary factorization orientation for Kf}
\pa_{g\circ f}^{\zcort{f}}=\pull{\left(\iota_{g\circ f}/\iota_g\right)}\otimes {({\iota_g}^*f)}^*\pa_g.
\end{equation}
\end{remark}
\begin{proof}
With the notations of Lemma \ref{boundary factorization orientation}, let $\mO^g$ be a local relative orientation of $g$.
Recall the notation
\begin{align*}
\mO^{{\iota_g}^*\!f}=\pull{\left(\iota_{g\circ f}/\iota_g\right)}\left(\mO^f\right).
\end{align*}
By the Koszul sign convention and since $\deg \pa_g=1$,
\begin{align*}
\left(\pull{\left(\iota_{g\circ f}/\iota_g\right)}\otimes {({\iota_g}^*f)}^*{\pa_g}\right)\left(\mO^f\otimes\Id\right)=\,\,\quad&(-1)^{f}\pull{\left(\iota_{g\circ f}/\iota_g\right)}\left(\mO^f\right)\otimes {({\iota_g}^*f)}^*{\pa_g}\left(1\right)
\\
=\,\,\quad&(-1)^{f+g}\mO^{{\iota_g}^*\!f}\otimes \left(\mO_c^{\iota_g}\right)\\
\overset{comp.}=\quad&(-1)^{f+g}\mO_c^{\iota_g}\circ\mO^{{\iota_g}^*\!f}\\
\overset{\text{Lemma \ref{boundary factorization orientation}}}=&(-1)^{f+g}\mO^f\circ\mO_c^{\iota_{g\circ f}}={\pa_{g\circ f}^{\zcort{ f}}}\left(\mO^f\right).
\end{align*}
\end{proof}
\begin{proof}[Proof of Lemma \ref{boundary of composition orientor}]
First, we show that
\begin{equation}
    \label{commutativity of orientor and boundary operator proof}
    \left((\Id \otimes \iota^*_{g\circ f}F)\circ {\pa_{g\circ f}}^Q\right)=(-1)^{|F|}\left({\pa_{g\circ f}}^{\zcort{f}\otimes f^*K}\circ \iota^*_{g\circ f}F\right).
\end{equation}
Indeed, let $q$ be a local section of $\iota^*_{g\circ f}Q$. It follows from the Koszul signs of Proposition \ref{Koszul signs}
that
\begin{align*}
\left((\Id \otimes \iota^*_{g\circ f}F)\circ {\pa_{g\circ f}}^Q\right)(q)&=
(-1)^{f}(\Id \otimes \iota^*_{g\circ f}F)\left(\mO^{\iota_{g\circ f}}_c\otimes q\right)=(-1)^{f+|F|}\left(\mO^{\iota_{g\circ f}}_c\otimes F(q)\right),\\
\left({\pa_{g\circ f}}^{\zcort{f}\otimes f^*K}\circ \iota^*_{g\circ f}F\right)(q)&={\pa_{g\circ f}}^{\zcort{f}\otimes f^*K}\left(F(q)\right)=(-1)^f\mO_c^{\iota_{g\circ f}}\otimes F(q).
\end{align*}
Second, we show that there is an equality of maps
\[
\iota_{g\circ f}^*Q\longrightarrow \zcort{f\circ \iota_{g\circ f}}\otimes (f\circ\iota_{g\circ f})^*K=\zcort{\iota_g\circ ({\iota_g}^*f)}\otimes {\left(\iota_g\circ ({\iota_g}^*f)\right)}^*K
\]
as follows.
\begin{equation}
\label{commutativity of orientor and boundary as orientors}
F\bu{\pa_{g\circ f}}^Q=(-1)^{|F|}{\pa_g}^K\bu\expinv{\left(\iota_{g\circ f}/\iota_g\right)}F
\end{equation}
Indeed,
\begin{align*}
F\bu {\pa_{g\circ f}}^{Q}=\quad&(\Id \otimes \iota_{g\circ f}^*F)\circ {\pa_{g\circ f}}^{Q}\\
\overset{\text{eq. \eqref{commutativity of orientor and boundary operator proof}}}=&(-1)^{|F|}{\pa_{g\circ f}}^{\zcort{f}\otimes f^*K}\circ\iota_{g\circ f}^*F\\
\overset{\text{Rmk. \ref{boundary factorization orientation diagram}}
}=&(-1)^{|F|}\left[\pull{(\iota_{g\circ f}/\iota_g)}\otimes {({\iota_g}^*f)}^*{\pa_g}^K\right]\circ \iota_{g\circ f}^*F
\\\overset{\text{Koszul \ref{Koszul signs}}}=&(-1)^{|F|}\left
[\left(\Id \otimes {({\iota_g}^*f)}^*{\pa_g}^{K}\right)\circ \left(\pull{\left(\iota_{g\circ f}/\iota_g\right)}\otimes\Id\right)\right]\circ \iota_{g\circ f}^*F\\
\overset{\text{Def. \ref{restriction of orientor to the boundary}}}=&(-1)^{|F|}\left(\Id \otimes {({\iota_g}^*f)}^*{\pa_g}^{K}\right)\circ \expinv{\left(\iota_{g\circ f}/\iota_g\right)}F\\
\overset{\text{Def. \ref{orientor composition}}}=&(-1)^{|F|}{\pa_g}^{K}\bu \expinv{\left(\iota_{g\circ f}/\iota_g\right)}F.
\end{align*}
To prove \eqref{boundary of composition orientor equation}, we compile the definitions and use the above equation, as follows.
\begin{align*}
    \pa(G\bu F)\overset{\text{Def. \ref{Boundary of orientor}}}=&(-1)^{|G|+|F|} (G\bu F)\bu {\pa_{g\circ f}}^{Q}\\
    \overset{\text{Lem. \ref{composition of orientors is associative}}}=&(-1)^{|G|+|F|}G\bu \left(F\bu {\pa_{g\circ f}}^{Q}\right)\\
    \overset{{\text{eq. \eqref{commutativity of orientor and boundary as orientors}}}}=&(-1)^{|G|}G\bu({\pa_g}^K\bu\expinv{\left(\iota_{g\circ f}/\iota_g\right)}F)\\
    \overset{\text{Lem. \ref{composition of orientors is associative}}}=&
    (-1)^{|G|}\left(G\bu{\pa_g}^K\right)\bu\expinv{\left(\iota_{g\circ f}/\iota_g\right)}F\\
    \overset{\text{Def. \ref{Boundary of orientor}}}=&
    \pa G\bu\expinv{\left(\iota_{g\circ f}/\iota_g\right)}F.
\end{align*}
\end{proof}
There is a canonical identification of $\pa_{(\Id\times f)}(A\times X)\simeq A\times \pa_f X$.
That is, $\pi_X$ factorizes through the boundary of $f$ as in Definition \ref{boundary factorization}.
\begin{lemma}\label{boundary of pulled orientation orientor}
Consider the following diagram of Cartesian squares.
\[
\begin{tikzcd}
\pa_{(\Id\times f)}(A\times X)\ar[r,"\iota_f^*\pi_X"]\ar[d,swap,"\iota_{(\Id\times f)}"]&\pa_f X\ar[d,"\iota_f"]\\
A\times X\ar[d,swap,"\Id\times f"]\ar[r,"\pi_X"]&X\ar[d,"f"]\\A\times Y\ar[r,swap,"\pi_Y"]&Y
\end{tikzcd}
\]
Let $Q,K$ be $\Z/2$-bundles over $X,Y,$ respectively, and let $F$ be a \orientor{f} of $Q$ to $K$.
Under the identification $\pa_{(\Id\times f)}(A\times X)\simeq A\times \pa_f X$, it holds that
\[
(\pa F)^A=\pa(F^A).
\]
\end{lemma}
\begin{proof}[Proof of Lemma \ref{boundary of pulled orientation orientor}]
Notice that 
\[
\expinv{\left(\iota_f^*\pi_X/\pi_Y\right)}\pa_f=\pa_{(\Id\times f)},
\]
as they are both given by contraction by $-\nu_{out}$ on the right, as explained in Remark \ref{boundary orientor is contraction by minus out pointing vector}. Therefore, 
recalling Definition \ref{pulled orientation orientor} and Definition \ref{Boundary of orientor}, the lemma follows by applying Lemma \ref{vertical pullback expinv composition} to the above diagram.
\end{proof}
\subsubsection{Miscellanea}
\begin{lemma}\label{orientors tensor distributivity}
With the setting of Definition \ref{orientor composition}, let $T$ be a $\Z/2$-bundle over $N$. Then
\begin{align*}
\left(G\bu F\right)^T&=G^{T}\bu F^{g^*T},\\
{}^T\left(G\bu F\right)&={}^TG\bu {}^{g^*T}F.
\end{align*}
\end{lemma}
\begin{proof}[Proof of Lemma \ref{orientors tensor distributivity}]
For the first equation, we calculate
\begin{align*}
(G\bu F)^T=&\left(G\bu F\right)\otimes \Id_{f^*g^*T}=\left(\left(\Id_{\zcort{f}}\otimes f^*G\right)\circ F\right)\otimes \Id_{f^*g^*T}\\\overset{\text{Koszul \ref{Koszul signs}}}=& \left(\Id_{\zcort{f}}\otimes f^*G\otimes \Id_{f^*g^*T}\right)\circ \left(F\otimes \Id_{f^*g^*T}\right)=G^T\bu F^{g^*T}.
\end{align*}
Similarly, for the second equation, we calculate
\begin{align*}
{}^T(G\bu F)=&\left(\tau_{T,\zcort{g\circ f}}\otimes \Id_{f^*g^*R}\right)\circ \left(\Id_{f^*g^*T}\otimes\left(G\bu F\right)\right)\\
\overset{\text{Lem. \ref{symmetry isomorphism distributivity lemma}}}=&\left(\left(\left(\Id_{\zcort{f}}\otimes \tau_{T,\zcort{g}}\right)\circ \left(\tau_{T,\zcort{f}}\otimes \Id_{f^*\zcort{g}}\right)\right)\otimes \Id_{f^*g^*R}\right)\circ\\&\circ \left(\Id_{f^*g^*T}\otimes\left(\left(\Id_{\zcort{f}}\otimes f^*G\right)\circ F\right)\right)\\\overset{\text{Koszul \ref{Koszul signs}}}=&
\left(\Id_{\zcort{f}}\otimes \tau_{T,\zcort{g}}\otimes \Id_{f^*g^*R}\right)\circ\left(\tau_{T,\zcort{f}}\otimes \Id_{f^*\zcort{g}}\otimes \Id_{f^*g^*R}\right)\circ\\
&\circ \left(\Id_{f^*g^*T}\otimes \Id_{\zcort{f}}\otimes f^*G\right)\circ \left(\Id_{f^*g^*T}\otimes F\right)\\
\overset{\text{Koszul \ref{Koszul signs}}}=&\left(\Id_{\zcort{f}}\otimes \tau_{T,\zcort{g}}\otimes \Id_{f^*g^*R}\right)\circ
\left(\Id_{\zcort{f}}\otimes \Id_{f^*g^*T}\otimes f^*G\right)\circ\\
&\circ\left(\tau_{T,\zcort{f}}\otimes \Id_{f^*K}\right)
\circ\left(\Id_{f^*g^*T}\otimes F\right)\\
&=\left(\Id_{\zcort{f}}\otimes f^*\left({}^TG\right)\right)\circ {}^{g^*T}F={}^TG\bu {}^{g^*T}F.
\end{align*}
\end{proof}
\begin{lemma}\label{orientors symmetry koszul}
Let $f:M\to N$ be a map of orbifolds with corners, and let $A,B$ be $\Z/2$-bundles over $M$ and $C$ be a $\Z/2$-bundle over $N$. If $F$ is an \orientor{f} of $B$ to $C$, then the following diagrams are commutative.
\[
\begin{tikzcd}
A\otimes B\ar[r,"{}^AF"]\ar[d,"\tau_{A,B}"]&\zcort{f}\otimes A\otimes f^*C\ar[d,"\Id \otimes \tau_{A,f^*C}"]\\
B\otimes A\ar[r,"F^A"]&\zcort{f}\otimes f^*C\otimes A
\end{tikzcd}\qquad\begin{tikzcd}
B\otimes A\ar[r,"F^A"]\ar[d,"\tau_{B,A}"]&\zcort{f}\otimes f^*C\otimes A\ar[d,"\Id \otimes \tau_{f^*C,A}"]\\
A\otimes B\ar[r,"{}^AF"]&\zcort{f}\otimes A\otimes f^*C
\end{tikzcd}
\]
That is,
\begin{align*}
\tau_{A,f^*C}\bu \,\,{}^AF&=F^A\bu \tau_{A,B},\\
\tau_{f^*C,A}\bu F^A&={}^AF\bu \tau_{B,A}.
\end{align*}
\end{lemma}
\begin{proof}[Proof of Lemma \ref{orientors symmetry koszul}]
The second equation is obtained by the first one by composing with $\tau_{B,A}$ on the right and $\tau_{f^*C,A}$ on the left. Assume $a,b,c$ are local sections of $A,B,C$ and $\mO^f$ is a local relative orientation of $f$, such that $F(b)=\mO^f\otimes f^*c.$ Then
\[
\tau_{A,f^*C}\bu\,\,{}^AF(a\otimes b)=(-1)^{(F-f)a}(\Id \otimes \tau_{A,f^*C})(\mO^f\otimes a\otimes f^*c)=(-1)^{(F-f+c)a}\mO^f\otimes f^*c\otimes a,
\]
while
\[
F^A\bu \tau_{A,B}(b\otimes a)=(-1)^{ab}F^A(b\otimes a)=(-1)^{ab}\mO^f\otimes f^*c\otimes a.
\]
However, $F+b\equiv_2f+c$ by the equality of the degrees of $F(b)=\mO^f\otimes f^*c$. This proves the commutativity of the first diagram.
\end{proof}

\subsection{Extension to essential subspaces}
By Lemma \ref{unique extension of bundle maps}, All Lemmas of the above section that concern fiber-products or boundary factorization hold for essential fiber-products and essential boundary factorization.
\subsection{Extension to arbitrary commutative rings}

\begin{definition}
Let $\A$ be a commutative ring. 
Let $f:M\to N$ be a map. Consider the $\A$-representation of $\Z/2$ given by negation, $(-1)\cdot a=-a$ for $a\in\A$. Then the \textbf{$\A-$relative orientation bundle} $\cort{f}\to M$ is the local system associated to the negation representation,
\[
\cort{f}=\zcort{f}\times_{\Z/2} \A.
\]
\end{definition}
\begin{remark}
As $\zcort{f}$ is concentrated in degree $-\dim f$, so is $\cort{f}$.
\end{remark}
If $Q,K$ are local systems over $M,N$ and $g:M\to N$ is a map, then a \textbf{\orientor{g} of $Q$ to $K$} is a morphism of local systems over $M$
\[
G:Q\to \cort{g}\otimes_\A g^*K.
\]
All definitions, equations and lemmas about $\Z/2$-orientors extend naturally to orientors of local systems over $\A$.

\section{Moduli Spaces}
\label{moduli spaces section}
\subsection{Open stable maps}
\label{Moduli spaces}
Let $(X_0,\w_0)$ be a symplectic manifold of dimension $2n$ and let $L_0\subset X_0$ be a Lagrangian. Let $\mu_0:H_2(X_0,L_0;\Z)\to \Z$ be the Maslov index \cite{Maslov}. The symplectic form $\w_0$ induces a map $\w_0:H_2(X_0,L_0;\Z)\to \R$ given by integration, $\b\mapsto \int_\b \w_0$. Let $\Pi_0$ be a quotient of $H_2(X_0,L_0;\Z)$ by a subgroup that is contained in the kernel of $(\mu_0,\w_0):H_2(X_0,L_0;\Z)\to \Z\oplus \R$. Thus, $\mu_0,\w_0$ descend to $\Pi_0$. Let $J_0$ be an $\w_0-$tame almost target structure on $X_0$.
A $J_0$-holomorphic genus-0 open stable map to $(X_0,L_0)$ of degree $\b\in \Pi_0$ with one boundary component, $k+1$ boundary marked points, and $l$ interior marked points, is a quadruple $\mathfrak u:=(\S,u,\vec z,\vec w$) as follows. The domain $\S$ is a genus-0 nodal Riemann surface with boundary consisting of one connected component. 
The map of pairs
\[
u:(\S,\pa \S)\to (X_0,L_0)
\]
is continuous, and $J_0$-holomorphic on each irreducible component of $\S$, satisfying
\[
u_*([\S,\pa \S])=\b.
\]
The boundary marked points and the interior marked points
\[
\vec z=(z_0,...,z_k),\qquad \vec w=(w_1,...,w_l),
\]
where $z_j\in \pa \S,w_j\in \overset{\circ}\S$, are distinct from one another and from the nodal points. The labeling of the marked points $z_j$ respects the cyclic order given by the orientation of $\pa \S$ induced by the complex orientation of $\S$. Stability means that if $\S_i$ is an irreducible component of $\S$, then either $u|_{\S_i}$ is non-constant or it satisfies the following requirement: If $\S_i$ is a sphere, the number of marked points and nodal points on $\S_i$ is at least 3; if $\S_i$ is a disk, the number of marked and nodal boundary points plus twice the number of marked and nodal interior points is at least 3. An \textbf{isomorphism of open stable maps} \[\phi:(\S,u,\vec z,\vec w)\to(\S',u',\vec z',\vec w')\] is a homeomorphism $\phi:\S\to \S'$, biholomorphic on each irreducible component, such that 
\[
u=u'\circ \phi,\qquad\qquad z_j'=\phi(z_j),\quad j=0,...,k,\qquad\qquad w_j'=\phi(w_j),\quad j=1,...,l.
\]
We denote $\mathfrak u\sim \mathfrak u'$ if there exists an isomorphism of open stable maps $\phi:\mathfrak u\to \mathfrak u'$.
Denote by $\mM_{k+1,l}(X_0,L_0,J_0;\b)$ the moduli space of $J_0$-holomorphic genus-$0$ open stable maps to $(X_t,L_t)$ of degree $\b$ with one boundary component, $k+1$ marked boundary points and $l$ marked interior points.

\subsection{Families}\label{families target definition section}
Let $\W$ be a manifold with corners. An orbifold with corners $M$ over $\W$ is a submersion $\pi^M:M\to \W$. We denote by $T^vM=\ker \left(d\pi^M\right)$ the vertical tangent bundle along $\pi^M$. Let $\pi^N:N\to \W$ be another orbifold with corners over $\W$ and $f:M\to N$ be a map over $\W$. Let $\xi:\W'\to \W$ be a smooth map of manifolds with corners. As $\pi^M$ is a submersion, the fiber product $\xi^*M:=\W'_\xi\times_{\pi^M}M$ exists. We also get an induced map $\xi^*f:\xi^*M\to \xi^*N$. The situation is summed up in the following diagram.
\begin{equation}
\label{pullback along families basic diagram}    
\begin{tikzcd}
\xi^*M\ar[r,"\xi^M"]\ar[d,swap,"\xi^*f"]&M\ar[d,"f"]\\\xi^*N\ar[r,"\xi^N"]\ar[d,swap,"\xi^*\pi^{N}"]&N\ar[d,"\pi^N"]\\\W'\ar[r,"\xi"]&\W
\end{tikzcd}
\end{equation}
Moreover, for a fiber-product
\[
\begin{tikzcd}
M\times_\W P\ar[r]\ar[d]&P\ar[d,"\pi^P"]\\M\ar[r,"\pi^M"]&\W
\end{tikzcd}
\]
of orbifolds with corners over $\W$, we write $\pi^{M\times P}_M,\pi^{M\times P}_P$ for the corresponding projections.
For many purposes, one may assume $\W$ is a point.
\begin{definition}
Let $f:M\to N$ be a map of smooth manifolds. A \textbf{vector field along $f$} is a section $u$ of the bundle $f^*TN\to M$. A vector field $u$ along $f$ determines a linear map
\begin{align*}
    i_u:A^{k}(N)\to& A^{k-1}(M)\\
    i_u\rho\left(v_1,...,v_{k-1}\right)|_{x\in M}=&\rho_{f(x)}\left(u(x),df_x(v_1(x)),\ldots,df_x(v_{k-1}(x))\right).
\end{align*}
called interior multiplication.
\end{definition}
\begin{definition}\label{horizontal form definition}
Let $\pi^M:M\to \W$ be a manifold over $\W$. A differential form $\xi\in A^*(M)$ is called \textbf{horizontal} with respect to $\pi^M$ if its restriction to vertical vector fields vanishes.
\end{definition}
\begin{definition}\label{exact submersion definition}
Let $\pi^M:M\to \W$ be a manifold over $\W$ and let $\w\in A^2(M)$. The submersion $\pi^M:M\to \W$ is called \textbf{exact} with respect to $\w$ if $\w$ is horizontal with respect to $\pi^M$ and for every vector field $u$ on $\W$ there exists a function $f_u:M\to \R$ such that for all vector fields $\tilde u$ on $M$ such that $d\pi^M(\tilde u)=u$, the $1$-form 
\[
i_{\tilde u}\w-df_u
\]
is horizontal with respect to $\pi^M$.
\end{definition}
\begin{remark}\label{exact submersion independent of lift remark}
When checking whether a submersion is exact with respect to a horizontal $2$-form, given a vector field $u$ on $\W$, it suffices to construct a lift $\tilde u$ of $u$ to $M$, and a function $f_{u}:M\to \R$ such that $i_{\tilde u}\w-df_{u}$ is horizontal. It follows that for any lift $\tilde u'$, the form $i_{\tilde u'}\w - df_{u}$ is horizontal. Indeed,
\[
i_{\tilde u'}\w-i_{\tilde u}\w=i_{(\tilde u'-\tilde u)}\w
\]
is horizontal.
\end{remark}
\begin{lemma}\label{naturality of exact submersion condition lemma}
Consider the following commutative diagram,
\[
\begin{tikzcd}
M'\ar[r,"\xi^M"]\ar[d,"\pi'"]&M\ar[d,"\pi"]\\\W'\ar[r,"\xi"]&\W
\end{tikzcd}
\]
where $\pi,\pi'$ are submersions. Denote by $\xi^*(d\pi):\xi^*TM\to \xi^*T\W$ the map given by 
\[
\left(\xi^*(d\pi)\right)(v)|_x=d\pi_{\xi^M(x)}(v(x)).
\]
Assume $\pi$ is exact with respect to the form $\w\in A^2(M)$. Then for every vector field $u$ along $\xi$, there exists a function $f_u:M'\to \R$ such that for all vector fields $\tilde u$ along $\xi^M$ satisfying $(\xi^*d\pi)(\tilde u)=u\circ\pi'$, it holds that the $1$-form $i_{\tilde u}\w-df_u\in A^1(M')$ is horizontal.
\end{lemma}
\begin{proof}
By a standard partition of unit argument, we may assume that $T\W$ admits a frame $(u_1,...,u_k)$. For $i=1,..., k$, let $f_i:M\to \R$ be the functions corresponding to $u_i$ according to Definition~\ref{exact submersion definition}. Let $u$ be a vector field along $\xi$. There exist unique functions $\a_i:\W\to \R$ for $i=1,...,k$ such that
\[
u(x)=\sum_{i=1}^k\a_i(x)\cdot u_i(\xi^M(x)).
\]
Define \[f_u:={\left(\sum_{i=1}^k(\a_i\circ \pi')\cdot (f_i\circ\xi^M)\right)}.\]
Similarly to Remark~\ref{exact submersion independent of lift remark}, it is enough to prove that $i_{\tilde u}\w-df_u$ is horizontal for a specific $\tilde u$. We proceed to construct $\tilde u$ for which we can compute $i_{\tilde u}\w$. For $i=1,...,k$, let $\tilde u_i$ be vector fields on $M$ such that $d\pi(\tilde u_i)=u_i.$
Then the vector field $\tilde u$ along $\xi^M$ defined by
\[
\tilde u(x):=\sum_{i=1}^k(\a_i\circ\pi')(x)\cdot \tilde u_i(\xi^M(x))
\] satisfies
\[
(\xi^*d\pi)(\tilde u)=u.
\]

We calculate, using the Leibniz's rule,
\begin{align*}
i_{\tilde u}\w-df_u=\sum_{i=1}^k(\a_i\circ\pi')\cdot\left( {\xi^M}^*\left(i_{u_i}\w-df_i\right)\right)-\sum_{i=1}^kd(\a_i\circ\pi')\cdot (f_i\circ \xi^M).
\end{align*}
Keeping in mind that pullbacks of horizontal forms are horizontal, all the forms on the right are horizontal.

\end{proof}
\begin{corollary}\label{pullback of exact manifold is exact naturality corollary}
Recall the notation from diagram~\eqref{pullback along families basic diagram}. Let $\w\in A^2(M)$ and assume $\pi^M$ is exact with respect to $\w$. It holds that $\xi^*\pi^{M}$ is exact with respect to $\left({\xi^M}\right)^*\w\in A^2(\xi^*M)$.
\end{corollary}

\begin{definition}\label{symplectic fibration definition}
Let $\pi^X:X\to \W$ be a manifold with corners over $\W$, and let $\w$ be a closed $2$-form on $X$. $\pi^X$ is called a \textbf{symplectic fibration} if it is a locally trivial fibration such that, for all $t\in \W$, $(\pi^{-1}(t),\w|_{\pi^{-1}(t)})$ is a symplectic manifold and the vertical boundary with respect to $\pi^X$ is empty. Let $L\subset X$ be a subfibration, that is, the restriction $\pi^L:=\pi^X|_L$ is a locally trivial fibration. We say that $L$ is a \textbf{Lagrangian subfibration} if $\w|_L$ is horizontal with respect to $\pi^L$. That is, the fibers of $\pi^L$ are Lagrangian submanifolds in the fibers of $\pi^X$. A Lagrangian subfibration is called \textbf{exact} if $\pi^L:=\pi^X|_L$ is exact with respect to $\w|_L$.

For a vector bundle $V\to B$, define the characteristic classes $p^\pm(V)\in H^2(B;\Z/2)$ by
\[
p^+(V)=w_2(V),\qquad p^-(V)=w_2(V)+w_1(V)^2.
\]
According to \cite{Kirby-Pin-structures}, $p^\pm(V)$ is the obstruction to the existence of a $Pin^{\pm}$ structure on $V$. See \cite{Kirby-Pin-structures} for a detailed discussion of the definition of the groups $Pin^{\pm}$ and the notion of $Pin^\pm$ structures.
We say that the fibration $X\supset L\to \W$ is \textbf{relatively $Pin^{\pm}$} if $p^{\pm}(T^vL)\in \Ima\left(i^*:H^2(X)\to H^2(L)\right)$, and $Pin^{\pm}$ if $p^{\pm}(T^vL)=0$. A \textbf{relative $Pin^{\pm}$ structure $\fp$ on $L$} is a relative $Pin^{\pm}$ structure on $T^vL$.
\end{definition}
\begin{remark}
The condition that the vertical boundary with respect to $\pi^X$ is empty may be replaced with an appropriate convexity property.
\end{remark}

We fix a symplectic fibration $(X,\w,\W,\pi^X)$ with an exact Lagrangian subfibration $L$ whose fibers are connected. For $t\in \W$, we write $X_t,L_t$ for the fibers of $\pi^X,\pi^L$, respectively, and $\w_t$ for the restriction of $\w$ to $X_t$. Set
\[
\zort{L}:=\zcort{\pi^L}[1-n].
\]
\begin{definition}\label{vertical orientability definition}
We say that the fibration $L$ is \textbf{vertically orientable} if $(\pi^L_*\zort{L})\neq\emp$. This is equivalent to the fiber being orientable.
\end{definition}
\begin{definition}
Let $b\in \Z$. We define a sheaf on $\W$
\[
\mathcal{X}_L^{b}:=
\pi_*^L\left(\zort{L}^{\otimes b}\right).
\]
\end{definition}
\begin{definition}
$b\in \Z$ is called \textbf{an exponent for $L$} if $\mathcal X_L^{b}$ is nonempty. In this case, the canonical map $\pi_L^*\mathcal X_L^{b}\to \zort{L}^{\otimes b}$ is an isomorphism, since both are $\Z/2$ local systems.
\end{definition}
\begin{remark}\label{exponent for L cases remark}
$b\in\Z$ is an exponent for $L$ if and only if $b$ is even or $L$ is vertically oriented.
\end{remark}
\begin{definition}\label{vertical relative homology definition}
Let $\underline {H_2}(X,L;\Z)$ be the sheaf over $\W$ given by sheafification of the presheaf with sections over an open set $U\subset \W$ given by $H_2\left({\left(\pi^X\right)}^{-1}(U),{\left(\pi^L\right)}^{-1}(U);\Z\right)$.
\end{definition}
The sheaf $\underline {H_2}(X,L;\Z)$ is the local system with fiber $H_2(X_t,L_t;\Z)$ for $t\in \W$, with the Gauss Manin connection. Let $\underline\mu:\underline{H_2}(X,L;\Z)\to \underline\Z$ be the bundle map given by the fiberwise Maslov index. Moreover, let $\underline\w:\underline{H_2}(X,L;\Z)\to \underline \R$ be the bundle map given over $t\in \W$ by 
\[
\underline\w|_t(\b_t)=\int_{\b_t}i_t^*\w,
\]
where $i_t:X_t\to X$ is the inclusion.
\begin{lemma}\label{constancy of w on H2 lemma}
The bundle maps $\underline\mu,\underline\w$ are constant on local sections of $\underline{H_2}(X,L;\Z)$.
\end{lemma}
\begin{proof}
For $\underline\mu$, the claim is clear since $\Z$ is discrete. For $\underline\w$,
it is enough to show that the map is constant on any path $\g:[0,1]\to \W$. Let $s\in [0,1]$ denote the parameter. Let $\b$ be a section of $\g^*\underline{H_2}(X,L;\Z),$ and let
\[
\s:(\S\times[0,1],\partial \S \times [0,1])\to (X,L)
\]
be such that the following diagram is commutative
\[
\begin{tikzcd}
\left(\S\times[0,1],\pa\S\times[0,1]\right)\ar[r,"\s"]\ar[d,"\pi_{[0,1]}"]&(X,L)\ar[d,"\pi^X"]\\
{[0,1]}\ar[r,"\g"]&\W
\end{tikzcd}
\]
and $\left[\s|_{\S\times\{s\}}\right]=\b|_{\g(s)}$. 
Let $\xi = \frac{\partial \sigma}{\partial s}.$
Recall that $L$ is an exact Lagrangian subfibration. By Lemma~\ref{naturality of exact submersion condition lemma} applied to the above diagram with $u=\pa_s$ and $\tilde u=\xi$, there exists a function $f:\pa \S\times[0,1]\to \R$ such that $i_{\xi}\w-df$ is horizontal with respect to $\pi_{[0,1]}$.
Since $[0,1]$ is one dimensional and $\pi_{[0,1]}^*ds$ is a horizontal $1$-form, it follows that
\begin{equation}\label{ds parallel to w minus df exauation}
\pi_{[0,1]}^*ds\we \left(i_\xi\w-df\right)=0.    
\end{equation}
Since $d\w=0,$
\begin{align*}
\int_{\b|_{\g(1)}}\w|_{\g(1)}-\int_{\b|_{\g(0)}}\w|_{\g(0)}\overset{\text{Stokes}}=&\int_{\pa \S\times [0,1]} \s^*\w= \int_{\pa \S\times [0,1]}\left(\pi_{[0,1]}\right)^*ds \wedge  i_{\xi}\w \\
\overset{\text{eq.~\eqref{ds parallel to w minus df exauation}}}=&\int_{\pa\S\times[0,1]}\left(\pi_{[0,1]}\right)^*ds\we df.
\end{align*}
However, by Stoke's theorem again,
\[
\int_{\pa\S\times[0,1]}\left(\pi_{[0,1]}\right)^*ds\we df=-\int_{\pa\S\times[0,1]}d(f\cdot \left(\pi_{[0,1]}\right)^*ds)=-\int_{\pa\S\times\{0,1\}}f\cdot \left(\pi_{[0,1]}\right)^*ds=0.
\]
Therefore, 
\[
\int_{\b|_{\g(1)}}\w|_{\g(1)}=\int_{\b|_{\g(0)}}\w|_{\g(0)}.
\] This completes the proof.
\end{proof}
\begin{definition}
A \textbf{target} is an octuple $\target:=(\W,X,\w,L,\pi^X,\fp,\underline\Upsilon,J)$ as follows. \begin{enumerate}
    \item 
$\W$ is manifold with corners.
\item $\pi^X:X\to \W$ is a symplectic fibration with respect to $\w$.
\item $L\subset X$ is an exact Lagrangian subfibration with a relative $Pin^\pm$ structure $\fp$.
\item $\underline\Upsilon \subset \ker(\underline\mu\oplus\underline\w)$ is a sub-bundle such that the quotient $\underline{H_2}(X,L;\Z)/\underline\Upsilon$ is a globally constant sheaf.
\item $J=\{J_t\}_{t\in\W}$ is a $\w$-tame almost complex structure on $T^vX$.
\end{enumerate}
The dimension of $\target$ is defined to be $\dim\target:=\dim \pi^X.$ 
\end{definition}
\begin{definition}
Let $\target:=(\W,X,\w,L,\pi^X,\fp,\underline\Upsilon,J)$ be a target. The \textbf{group of degrees of $\target$} which we denote by $\Pi:=\Pi(\target)$ is the fiber of $\underline{H_2}(X,L;\Z)/\underline\Upsilon.$
Lemma~\ref{constancy of w on H2 lemma} implies that the bundle-maps $\underline\mu,\underline\w$ descend to maps $\mu:\Pi\to \Z$ and $\w:\Pi\to \R$.
A degree $\b\in \Pi$ is called \textbf{admissible} if $\mu(\b)+1$ is an exponent for $L$. Denote by $\Pi^{ad}\subset\Pi$ the admissible degrees.
\end{definition}
\begin{example}
Consider $\R P^1$ as lines in the $yz$ plane and $S^2$ as the unit vectors in the $xyz$ space. For $t\in \R P^1$ and a vector $\vec v\in S^2$, we denote $\vec v\perp t$ if $\vec v$ is perpendicular to $t$. Set $\W=\R P^1$ and $X= \R P^1\times S^2$. Denote by $\pi:X\to \W$ and $p:X\to S^2$ the projections. Let $\w=p^*\w_0$ and $J=p^*J_0$ where $\w_0,J_0$ are the standard symplectic form and complex structure on $S^2$, respectively. Let
\[
L=\left\{(t,\vec v)\in X\mid \vec v\perp t\right\}.
\]
Namely, $L$ is a circle rotating on its diameter. Note that $\w|_L=0$. In particular,
$L\subset X$ is an exact Lagrangian subfibration. It is both relatively $Pin^+$ and relatively $Pin^-$. This may be seen as follows. $L$ is the Klein bottle and $T^vL\simeq\left(\pi^L\right)^*\mathcal O_{\R P^1}(-1)$. By the naturality of the characteristic classes $p^\pm$, it follows that $L$ is both $Pin^+$ and $Pin^-$ as a fibration. Let $\fp$ be any $Pin^\pm$ structure on $L$. 
The fibration $L$ is vertically orientable, yet the map $\pi^L$ is not relatively orientable. Moreover, we have 
\[
H_2(X_t,L_t;\Z)=\Z\oplus \Z
\]
and parallel transporting $(x,y)\in H_2(X_t,L_t;\Z)$ along the loop $\R P^1$ we get $(x,y)\mapsto (y,x)$. Let $\underline\Upsilon=\ker(\underline\mu\oplus \underline\w)$ which is the M\"obius $\Z$ bundle over $\R P^1$. 
Then
\[
\target_0:=\left(\W,X,\w,L,\pi^X,\fp,\underline\Upsilon,J\right)
\]
is a target. It holds that $\Pi(\target_0)=\Z$. Alternatively, we can take $\underline\Upsilon=2\cdot\ker(\underline\mu\oplus\underline\w)$ and then $\Pi=\Z\oplus \Z/2.$
\end{example}
\begin{definition}
Let $\target:=(\W,X,\w,L,\pi^X,\fp,\underline\Upsilon,J)$ be a target. Let $\W'$ be a manifold with corners and $\xi:\W'\to \W$ be a smooth map. By Corollary~\ref{pullback of exact manifold is exact naturality corollary}, $\xi^*L\subset\xi^*X$ is exact with respect to ${\xi^X}^*\w$. Thus, the octuple \[\xi^*\target=
\left(\W',\xi^*X, \xi^*\w,\xi^*L,\xi^*\pi^{X},\xi^*\fp, \xi^*\underline\Upsilon,\xi^*J\right)
\]
is a target.
Since pullback of sheaves is an exact functor, the canonical map
\[
\underline{H_2}(\xi^*X,\xi^*L;\Z)/\xi^*\underline\Upsilon\to \xi^*\left(\underline{H_2}(X,L;\Z)/\underline\Upsilon\right)
\] is an isomorphism, so $\xi^*\target$ is indeed a target. In particular, the canonical map \[\xi^*:\Pi(\target)\to \Pi(\xi^*\target)\] is an isomorphism.
\end{definition}
\begin{definition}\label{moduli spaces of families definition}
Fix a target $\target=(\W,X,\w,L,\pi^X,\fp,\underline\Upsilon, J)$.
For $k\geq-1,l\geq0$ and $\b\in \Pi$, denote by
\begin{equation*}
\mM_{k+1,l}(\b):=\mM_{k+1,l}(\target;\b):=\left\{(t,\fu)\mid t\in \W, \fu\in \mM_{k+1,l}(X_t,L_t,J_t;\b_t)\right\}.    
\end{equation*}
Denote by $\pi^{\mM}:\mM_{k+1,l}(\b)\to\W$ the map $(t,\mathfrak u)\mapsto t$.
Denote by
\[
\begin{split}
    evb_j^{(k,l,\b)}:\mdl{3}\to L,\qquad &j=0,...,k,\\
    evi_j^{(k,l,\b)}:\mdl{3}\to X,\qquad &j=1,...,l,
\end{split}
\]
the evaluation maps given by 
\begin{align*}
    evb_j^{(k,l,\b)}(t,(\S,u,\vec z,\vec w))&=(t,u(z_j)),\\evi_j^{(k,l,\b)}(t,(\S,u,\vec z,\vec w))&=(t,u(w_j)).
\end{align*}
We may omit part or all superscripts ${(k,l,\b)}$ when they are clear from the context.
\end{definition}

To streamline the exposition, we assume that $\mdl{3}$ are smooth orbifolds with corners and $evb_0^\b$ are submersions for $\b\in\Pi$. These assumptions hold in a range of important examples~\cite[Example 1.5]{Sara1}.

In general, the moduli spaces $\mdl{3}$ are only metrizable spaces. They can be highly singular and have varying dimension. Nonetheless, the theory of the virtual fundamental class being developed by several authors~\cite{Fu09a,FO19,FO20,HW10,HWZ21} allows one to perturb the $J$-holomorphic map equation to obtain moduli spaces that are weighted branched orbifolds with corners and evaluation maps that are smooth. Thus, we may consider \orientor{evb_0}s. Furthermore, by averaging over continuous families of perturbations, one can make $evb_0^\b$ behave like a submersion. So, fiber-products along $evb_0^\b$ exist. See~\cite{Fu09a,FO19,FO20}. When the unperturbed moduli spaces are smooth of expected dimension and $evb_0^\beta$ is a submersion, one can choose the perturbations to be trivial. Furthermore, as explained in~\cite{Fu09a,FO19}, one can make the perturbations compatible with forgetful maps of boundary marked points.
The compatibility of perturbations with forgetful maps of interior marked points has not yet been fully worked out in the Kuranishi structure formalism.

\begin{remark}
If $\xi:\W'\to \W$ is a map, $\target:=(\W,X,\w,L,\pi^X,\fp,\underline\Upsilon,J)$ is a target and $\b\in \Pi(\target)$, then 
\[
\mM_{k+1,l}(\xi^*\target;\xi^*(\b))=\xi^*\mM_{k+1,l}(\target;\b).
\]Moreover, for $i\leq k$ and $j\leq l$,
\[
evb_i^{\left(\xi^*\target\right)}=\xi^*\left(evb_i^{\target}\right),\qquad evi_j^{\left(\xi^*\target\right)}=\xi^*\left(evi_j^{\target}\right).
\]
\end{remark}
For a finite list of indices $I$, let $\mM_{k+1,I}(\b)$ denote the moduli space diffeomorphic to $\mM_{k+1,|I|}(\b)$ with interior marked points labeled by $I$. It carries evaluation maps $evb_i^\b$ for indices $i=0,...,k$ and $evi_j^\b$ for indices $j\in I$.

The orbifold structure of $\mM_{k+1,I}(\b)$ arises from the automorphisms of open stable maps. Vertical corners of codimension $r$ along the map $\pi^{\mM}$ consist of open stable maps $(\S,u,\vec z,\vec w)$ where $\S$ has $r$ boundary nodes. For $r=0,1,2,...$ denote by 
\[\mM_{k+1,I}(\b)^{(r)}\subset\mM_{k+1,I}(\b)\] the dense open subset consisting of stable maps with no more than $r$ boundary nodes and no interior nodes. Each of these subspaces is an essential subset of $\mM_{k+1,I}(\b)$.
A precise description of the vertical corners in the case $r=1$ is given in terms of gluing maps, as follows. 

Let $k\geq-1,l\geq 0,\b\in\Pi$. Fix partitions $k_1+k_2=k+1,\b_1+\b_2=\b$ and $I\dot\cup J=[l]$, where $k_1>0$ if $k+1> 0$. When $k+1>0$, let $0< i\leq k_1$. When $k=-1$ let $i=0$. Let
\[
B^{(1)}_{i,k_1,k_2,I,J}(\b_1,\b_2)\subset \pa \mM_{k+1,l}^{(1)}(\b)
\]
denote the locus of two component stable maps, described as follows. One component has degree $\beta_1$ and the other component has degree $\beta_2.$ The first component carries the boundary marked points labeled $0,\ldots,i-1,i+k_2,\ldots,k,$ and the interior marked points labeled by $I.$ The second component carries the boundary marked points labeled~$i,\ldots,i+~k_2-1$ and the interior marked points labeled by $J.$ The two components are joined at the $i$th boundary marked point on the first component and the $0$th boundary marked point on the second. Let 
\[
B_{i,k_1,k_2,I,J}(\b_1,\b_2):=\overline{B^{(1)}_{i,k_1,k_2,I,J}(\b_1,\b_2)}\subset \pa\mdl{3}
\] denote the closure. Denote by 
\[
\iota^{\b_1,\b_2}_{i,k_1,k_2,I,J}:B_{i,k_1,k_2,I,J}(\b_1,\b_2)\to \mdl{3}
\] the inclusion of the boundary.

There is a canonical gluing map 
\[
\vartheta_{i,k_1,k_2,\b_1,\b_2,I,J}:\mdl{1}_{evb_i^{\b_1}}\times_{evb_0^{\b_2}}\mdl{2}\to B_{i,k_1,k_2,I,J}(\b_1,\b_2).
\]
This map is a diffeomorphism, unless $k = -1, I = \emptyset = J$ and $\beta_1 = \beta_2.$ In the exceptional case, $\vartheta$ is a $2$ to $1$ local diffeomorphism in the orbifold sense.
The dense open subset 
\[
\mM_{k_1+1,I}^{(0)}(\b_1)\times_L\mM_{k_2+1,J}^{(0)}(\b_2)
\]
is carried by $\vartheta_{i,k_1,k_2,\b_1,\b_2,I,J}$ onto $B^{(1)}_{i,k_1,k_2,I,J}(\b_1,\b_2)$.
We abbreviate
\begin{align*}
\vartheta^\b&=\vartheta_{i,k_1,k_2,\b_1,\b_2,I,J}
\end{align*} 
when it creates no ambiguity.
The images of all such $\vartheta^\b$ intersect only in codimension $2$, and cover the vertical boundary of $\mdl{3}$, unless $k=-1$ and $\b\in \text{Im}(\G\left(\underline {H_2}(X;\Z)\right)\to \Pi)$.
In the exceptional case, there might occur another phenomenon of bubbling, which will not be relevant to this paper.



\section{Orientors and the gluing map}\label{Moduli orientation}
Fix a target $\target=\left(\W,X,\w,L,\pi^X,\fp,\underline\Upsilon, J\right)$. In the rest of this section we generalize the result of \cite[Proposition 8.3.3]{Fukaya}, which discusses the orientation of the gluing map $\vartheta^\b$ of Section~\ref{Moduli spaces}. In \cite[Proposition 8.3.3]{Fukaya} the moduli spaces are canonically oriented, and it is determined whether $\vartheta^\b$ preserves orientation. In our case, the moduli spaces may be non-orientable. There is only an \orientor{evb_0} of a bundle given in terms of the bundle $\zort{\pi^L}$, see Definition~\ref{Jakethesis} for the construction of the orientors. As a diffeomorphism, the map $\vartheta^\b$ induces a map \[\pull{\vartheta^\b}:{\vartheta^\b}^*\zcort{evb_0^{\b}\circ\iota^\b}\rightarrow\zcort{evb_0^{\b_1}\circ p_1}\]
where $p_1:\mM_{k_1+1,I}(\b_1)\times_L\mM_{k_2+1,J}(\b_2)\to \mM_{k_1+1,I}(\b_1)$ is the projection.
In this section we write a formula for $\pull{\vartheta^\b}$ in terms of the bundle $\zort{L}$. See Theorem~\ref{expinv of boundary of J by thetabeta theorem}.
\subsection{Conventions and notations of one moduli space}
\label{Moduli orientation - Conventions}
We identify $D^2$ with the unit disk in $\C$ and identify $S^1$ with $\pa D^2$. In particular, $D^2$ carries the induced orientation of $\C$ and $S^1$ carries the induced orientation of the boundary of $D^2$, which is counterclockwise. In particular, for points $z_0,z_1\in \pa D^2$ there is a notion of oriented segment $[z_0,z_1]$.

Following \cite[Section 8.3]{Fukaya} we fix an orientation for $\text{Aut}(D^2)\simeq PSL_2(\R)$ as follows. Recall that $\text{Aut}(D^2)$ acts from the left on $D^2$ and $S^1$ by M\"obius transformations 
\[
\begin{pmatrix}
a&b\\c&d
\end{pmatrix}\cdot z=\frac{az+b}{cz+d}.
\]
We let it act from the right on $D^2$ and $S^1$ by $(z,g)\mapsto g^{-1}\cdot z$.
\begin{definition}\label{orientation of PSL}
The \textbf{canonical orientation} $\mO_c^{\text{Aut}(D^2)}$ of $\text{Aut}(D^2)$ is given as follows. We pick three points~$z_0,z_1,z_2\in~S^1$ ordered counterclockwise. We embed $\text{Aut}(D^2)$ as an open subset of $S^1\times S^1\times S^1$ by \[{g
\mapsto (g^{-1}\cdot z_0,g^{-1}\cdot z_1,g^{-1}\cdot z_2)}\] and we equip $\text{Aut}(D^2)$ with the induced orientation.
\end{definition}

For a finite set $I$, set
\[
\mW_{k+1,I}=\mM_{k+1,I}(\b=0;X=\W,L=\W).
\]
Set
\[
\widehat{\mW_{k+1,I}}:=\W\times\left\{\begin{matrix}
(z_0,...,z_k)\in {(S^1)}^{k+1},\\
\{w_j\}_{j\in I}\in {(\text{int}D^2)}^I
\end{matrix}\quad\left|\quad\begin{matrix}
\forall j_1,j_2\leq k,\quad j_1\neq j_2\implies z_{j_1}\neq z_{j_2},\\
(z_0,...,z_k)\text{ is cyclically ordered},\\
\forall j_1,j_2\in I,\quad j_1\neq j_2\implies z_{j_1}\neq z_{j_2}
\end{matrix}\right.\right\}.
\]
The moduli space $\mW_{k+1,I}$ is nonempty only if $k+2|I|\geq 2$, since constant maps are unstable unless the condition on the number of marked points holds. If it is nonempty, then
\begin{equation}
\mW^{(0)}_{k+1,I}=\left.\widehat{\mW_{k+1,I}}\right/\text{Aut}(D^2).
\label{description of disk space}
\end{equation}

\begin{definition}
\label{interior forgetful map definition}
Let
\[
U_{k+1,I}:\mM_{k+1,I}(\b)\to
\mW_{k+1,I}
\]
denote the map that forgets the holomorphic map and contracts unstable components.
For $k\geq 0$, define
\[
Ev_{k+1,I}^\b:\mM_{k+1,I}(\b)\to L\times_\W\mW_{k+1,I},\quad Ev_{k+1,I}^\b=(evb_0^\b,U_{k+1,I}).
\]
Assume that $I\subset[l]$. Denote $\hat I=I\dot\cup \{l+1\}$. Let
\[
Fi^\b_{k+1,I}:\mM_{k+1,\hat I}(\b)\to \mM_{k+1,I}(\b)
\]
denote the map that forgets the $|I|+1$st interior point and contracts unstable components. Moreover, denote by 
\[
Fi_{k+1,I}:\mW_{k+1,\hat I}\to \mW_{k+1,I}
\]
the forgetful map when $X=L=\W$.

Similarly, let
\[
Fb^\b_{k+1,I}:\mM_{k+2,I}(\b)\to \mM_{k+1,I}(\b)
\]
denote the map that forgets the $k+1$st boundary point and contracts unstable components.
Denote by
\[
Fb_{k+1,I}:\mW_{k+2,I}\to \mW_{k+1,I}
\]
the forgetful map when $X=L=\W$.

Moreover, for $k\geq0$, let $f^\b:\mM_{k+1,l}(\b)\to \mM_{k+1,l}(\b)$ be given by
\[
f^\b(\S,u,(z_0,...,z_k),\vec w)=(\S,u,(z_1,...z_k,z_0),\vec w),
\]
the cyclic map. Denote by $f$ the map $f^\b$ when $X=L=\W$.
\end{definition}
It holds that \[
evi_j\circ f^\b =evi_j,\quad evb_k\circ  f^\b =evb_0,\quad evb_j\circ f^\b =evb_{j+1},\,\,j=0,...,k-1.
\]
Assume $k+2|I|\geq 2$. The following diagrams are essential pullbacks.
\begin{equation}
\begin{tikzcd}
\mM_{k+1,\hat I}(\b)\ar[r,"Fi^\b_{k+1,I}"]\ar[d,swap,"U_{k+1,\hat I}"]&\mM_{k+1,I}(\b)\ar[d,"U_{k+1,I}"]\\
\mW_{k+1,\hat I}\ar[r,swap,"Fi_{k+1,I}"]&\mW_{k+1,I}
\end{tikzcd}\qquad
\begin{tikzcd}
\mM_{k+2, I}(\b)\ar[r,"Fb_{k+1,I}^\b"]\ar[d,swap,"U_{k+2, I}"]&\mM_{k+1,I}(\b)\ar[d,"U_{k+1,I}"]\\
\mW_{k+2, I}\ar[r,swap,"Fb_{k+1,I}"]&\mW_{k+1,I}
\end{tikzcd}
\label{forgetting boundary point is functorial diagram}    
\end{equation}

Moreover, for $k\geq0$ the following diagram is a pullback diagram.
\begin{equation}\label{cyclic map pullback diagram}
\begin{tikzcd}
\mM_{k+1,l}(\b)\ar[r,"f^\b"]\ar[d,"U_{k+1,l}^\b"]&\mM_{k+1,l}(\b)\ar[d,"U_{k+1,l}^\b"]\\
\mW_{k+1,l}\ar[r,"f"]&\mW_{k+1,l}
\end{tikzcd}    
\end{equation}

For $k\geq 0$, we extend the boundary marked points from the set of indices $\{0,...,k\}$ to the set $\Z$ cyclicly. For example, $z_{k+1}=z_{0},$ and $z_{-1}=z_k$.
\begin{definition}\label{boundary transport}
Let $k\geq 0$. For $i,j\in\Z$, the \textbf{boundary transport} map is the degree $0$ isomorphism
\[
c_{ij}:(evb_j^\b)^*\zort{L}{\overset\sim\longrightarrow} (evb_i^\b)^*\zort{L},
\]
given as follows.
Define $c_{ii}$ to be the identity.
Assume $i< j$.
We construct $c_{ij}$ at a point $(t,\fu)=(t,(\S,u,\vec z,\vec w))\in \mM_{k+1,I}(\b)$. We then define $c_{ji}=c_{ij}^{-1}$. The complex structure of $\S$ provides
$\pa \S = \sum_r S^1/\sim$ with the boundary orientation. Consider the unique up to homotopy oriented arc $\g:[0,1]\to \pa \S$ from ${z_i}$ to $z_j$, that passes through $z_{i+1},...,z_{j-1}$ in order, and switches components whenever it reaches a nodal point. Then $u\circ\g$ defines a path in $L_t$ from $u(z_i)$ to $u(z_j)$. Trivializing $(u\circ \g)^*\zort{L}|_t$, as $[0,1]$ is contractible, we get an isomorphism
\[
c^u_{ij}:\zort{L}|_{t,u(z_j)}\to \zort{L}|_{t,u(z_i)}.
\]
If~$\phi:\fu\to \fu'$ is an isomorphism of stable maps,  then since $\phi$ preserves orientation and marked points,
\[
c_{ij}^{\fu}=c^{\fu'}_{ij}.
\]
Moreover, 
since $c^\fu_{ij}$ depends continuously on $(t,\fu)$ in the Gromov topology, we get a continuous map
\[
c_{ij}:{(evb_j^\b)}^*\zort{L}\to {(evb_0^\b)}^*\zort{L}
\]
on $\mM_{k+1,I}(\b)$.
\end{definition}

\begin{remark}
We treat the maps $c_{ij}$ as \orientor{\Id_{\mM_{k+1,I}(\b)}}s of ${(evb_j^\b)}^*\zort{L}$ to ${(evb_0^\b)}^*\zort{L}$.
\end{remark}
The following remarks imply that $c_{ij}$ for $i,j\in\Z$ is determined by $c_{0l}$ for $l\in \{0,...,k\}$.
\begin{remark}[Full circle]
\label{full round}
The isomorphism $c_{0,k+1}$ gives the involution $(-1)^{\mu(\b)}$ of~${evb^\b_0}^*\zort{L}$.
\end{remark}
\begin{remark}[Group-like]
\label{group-like transport}
Let $i,j,k\in \Z$ be indices. Then
\[
c_{ik}=c_{ij}\circ c_{jk}.
\]
\end{remark}
The proof of the following lemma is left to the reader.
\begin{lemma}[Properties of boundary transport]
Let $k,l\geq 0$. Then the following equations hold.
\begin{align*}
     \left(Fb_{k+1,I}^\b\right) ^*c_{0j}=&c_{0j},\qquad 0\leq j\leq k,\\
     \left(Fi_{k+1,I}^\b\right) ^*c_{0j}=&c_{0j},\qquad 0\leq j\leq k,\\
     \left(f^\b\right) ^*c_{0j}=&c_{1,j+1}.
\end{align*}
\end{lemma}
\subsection{Cauchy-Riemann   \texorpdfstring{$Pin$}{TEXT} boundary problems}
We recall here a number of definitions and results from~\cite{JakePhD}. See also~\cite[Chapter 8]{Fukaya} and~\cite{Jake-involutions-mirror-symmetry}.
 In the following, given a vector bundle over a manifold $V\to M$, we denote by $\G(V)$ an appropriate Banach space completion of the smooth sections of $V$.
\begin{definition}
A \textbf{Cauchy-Riemann $Pin$ boundary value problem} is a quintuple $\underline D=(\S,E,F,\fp,D)$ where
\begin{itemize}
    \item $\S$ is a Riemann surface.
    \item $E\to \S$ is a complex vector bundle.
    \item $F\to \pa \S$ is a totally real sub-bundle of $E|_{\pa \S}$ with an orientation over every component of $\pa \S$ where it is orientable.
    \item $\fp$ is a $Pin$ structure on $F$.
    \item $D:\G\left((\S,\pa \S),(E,F)\right)\to \G\left(\S,\W^{0,1}(E)\right)$ is a linear partial differential operator satisfying, for $\xi\in \G\left((\S,\pa \S),(E,F)\right)$ and $f\in C^\infty(\S,\R)$,
    \[
    D(f\xi)=fD\xi+(\overline\pa f)\xi.
    \]
    Such a $D$ is called a \textbf{real linear Cauchy-Riemann operator}.
\end{itemize}
\end{definition}
\begin{definition}
The \textbf{determinant line} of a Fredholm operator $D$ is the one-dimensional vector space
\[
\det(D)=\La^{\max}(\ker D)\otimes \La^{\max}(\coker D).
\]
\end{definition}
The determinant lines of a continuously varying family of Fredholm operators fit together to form a determinant line bundle~\cite[Appendix A]{mcduff-salamon}. It is well-known that real linear Cauchy-Riemann operators are Fredholm~\cite[Appendix C]{mcduff-salamon}. The following restates Proposition 2.8 and Lemma 2.9 of~\cite{JakePhD}. 
\begin{lemma}\label{Cauchy Riemann boundary value problem properties lemma}
The determinant of the real linear Cauchy-Riemann operator of a Cauchy-Riemann $Pin$ boundary value problem carries a canonical orientation, which has the following properties:
\begin{enumerate}
    \item The orientation varies continuously in families and thus defines an orientation of the associated determinant line bundle.
    \item Reversing the orientation of the boundary condition over one component of the boundary reverses the canonical orientation of the determinant line.
\end{enumerate}
\end{lemma}
\subsection{Orientors on moduli spaces}
\begin{definition}\label{canonical orientations for cyclic and forgetful maps definition}
The map $f^\b$ is a diffeomorphism, and thus has a canonical orientation $\mO^f_c$. Moreover, the maps $Fb^\b_{k+1,I}$ and $Fi^\b_{k+1,I}$ have canonical orientations, denoted $\mO^{Fb}, \mO^{Fi}$, originating in the complex orientation of the domains of the $J$-holomorphic maps and the induced orientation on the boundary.
\end{definition}
\begin{remark}\label{canonical orientations for cyclic and forgetful maps remark}
Consider the following essential pullback diagram.
\begin{equation}
    \label{commutativity of boundary and interior forgetful maps diagram}
    \begin{tikzcd}
        \mM_{k+2,\hat I}(\b)\ar[r,"Fb^\b_{k+1,\hat I}"]\ar[d,swap,"Fi^\b_{k+2,I}"]&\mM_{k+1,\hat I}(\b)\ar[d,"Fi^\b_{k+1,I}"]\\
        \mM_{k+2, I}(\b)\ar[r,swap,"Fb^\b_{k+1,I}"]&\mM_{k+1, I}(\b)
    \end{tikzcd}
\end{equation}
Then 
\begin{equation}
\label{commutativity of boundary and interior forgetful maps boundary equation}
    \mO^{Fi_{k+1}}\circ \mO^{Fb_{\hat I}}=\mO^{Fb_{I}}\circ \mO^{Fi_{k+2}}.
\end{equation}
Moreover, 
\[
\mO^{Fi^\b}:=\pull{\left(U_{k+1,\hat I}/U_{k+1,I}\right)}\mO^{Fi},\qquad
\mO^{Fb^\b}:=\pull{\left(U_{k+2,I}/U_{k+1,I}\right)}\mO^{Fb},
\]
where the pullbacks are taken according to diagrams~\eqref{forgetting boundary point is functorial diagram}.
\end{remark}

\begin{notation}
As this causes no confusion, we write $\expinv{\left(Fb^\b_{k+1,I}\right)}$ for $\expinv{\left(Fb^\b_{k+1,I},\mO^{Fb}\right)}$, and similarly for $Fi$.
\end{notation}
\begin{lemma}\label{disk space orientors with properties existence lemma}
There exists a collection of \eorientor{\pi^{\mW_{k+1,I}}}s $\phi_{k+1,I}$ of $\underline{\Z/2}$ labeled by $k\geq-1,$ $l\geq 0,I\subset [l]$ such that $k+2|I|\geq 2$, with the following properties:
\begin{enumerate}
    \item\label{disk space orientors:base case} In the zero dimensional cases $(k,|I|)\in \{(0,1),(2,0)\}$, in which the projection maps $\pi^{\mW_{k+1,I}}:\mW_{k+1,I}\to \W$ are diffeomorphisms, the orientors are the canonical orientors from Definition~\ref{orientation as orientor} \[\phi_{k+1,I}=\phi_{\pi^{\mW_{k+1,I}}}.\]
    \item\label{disk space orientors:boundary permutation} The orientors are well behaved under cyclic permutations of the boundary points, that is \[\expinv{f}\phi_{k+1,I}=(-1)^k\phi_{k+1,I}.\]
    \item\label{disk space orientors:interior permutation} The orientors are well behaved under arbitrary relabeling of interior points, that is, if $I,J\subset [l]$ and $a:I\to J$ is a bijection, the induced map $\tilde a:\mW_{k+1,I}\to \mW_{k+1,J}$ satisfies
    \[
    \expinv{\tilde a}\phi_{k+1,J}=\phi_{k+1,I}.
    \]
    \item\label{disk space orientors:forget boundary} The orientors are well behaved under forgetting boundary points, that is \[\expinv{\left(Fb_{k+1,I}\right)}\phi_{k+1,I}=\phi_{k+2,I}.\]
    \item \label{disk space orientors:forget interior} 
    The orientors are well behaved under forgetting interior points, that is \[\expinv{\left(Fi_{k+1,I}\right)}\phi_{k+1,I}=\phi_{k+1,\hat I}.\]
\end{enumerate}
Properties~\ref{disk space orientors:base case},~\ref{disk space orientors:forget boundary} and~\ref{disk space orientors:forget interior} determine the orientors $\phi_{k+1,I}$ uniquely.
\end{lemma}
\begin{proof}
We follow the proof of~\cite[Lemma 2.24]{jake-rahul-intersection-kdv}. A direct calculation shows that
\[
\expinv{\left(Fb_{2,1}\right)}\expinv{\left(Fb_{1,1}\right)}\phi_{1,1}=\expinv{\left(Fi_{3,0}\right)}\phi_{3,0}.
\]
See~\cite[Figure 2]{jake-rahul-intersection-kdv}.
Equation~\eqref{commutativity of boundary and interior forgetful maps boundary equation} and Lemma~\ref{pullback by composition relatively oriented functoriality lemma} imply that for any $k,I$ and any \orientor{\pi^{\mW_{k+1,I}}} $\phi$ we have
\[
\expinv{\left(Fi_{k+2,I}^\b\right)}\expinv{\left(Fb_{k+1,I}^\b\right)}\phi=\expinv{\left(Fb_{k+1,\hat I}^\b\right)}\expinv{\left(Fi_{k+1,I}^\b\right)}\phi.
\]
Therefore, we can define $\phi_{k+1,I}$ recursively, according to properties~\ref{disk space orientors:base case},~\ref{disk space orientors:forget boundary} and~\ref{disk space orientors:forget interior}.

For property~\ref{disk space orientors:boundary permutation} to hold, it suffices to check that $f$ changes orientation by $(-1)^k.$ This follows by Lemma~\ref{orientor restoration group action lemma} since $\mW^{(0)}_{k+1,I}=\widehat{\mW_{k+1,I}}/\text{Aut}(D^2)$ and the diffeomorphism of $\widehat{\mW_{k+1,I}}$ which cyclicly permutes the $S^1$ components changes orientation by $(-1)^k$.

Similarly, for property~\ref{disk space orientors:interior permutation} to hold, it suffices to check that $\tilde a$ does not change orientation. This is parallel to the proof of property~\ref{disk space orientors:boundary permutation}, only that the diffeomorphism does not change orientation because $\dim (D^2)=2$.
\end{proof}

\begin{lemma}
\label{Jake Thesis lemma}
Let $k+2|I|\geq 0$. For $k\geq 0$, the relative $Pin^\pm$ structure $\fp$ determines a canonical \orientor{U_{k+1,I}^\b} of $(evb^\b_{z_0})^*\left(\zort{L}^{\otimes \left(\mu(\b)+1\right)}\right)$ to $\underline{\Z/2}$ of degree $1-n$
\[
\tilde I^\b_{k+1,I}:(evb^\b_{z_0})^*\left(\zort{L}^{\otimes \left(\mu(\b)+1\right)}\right)\to \zcort{U_{k+1,l}^\b}.
\]

For $k=-1,\b\in \Pi^{ad}$, the relative $Pin^\pm$ structure $\fp$ determines a canonical \orientor{U_{0,I}^\b} of $\left(\pi^\mM\right)^*\mathcal X_L^{\mu(\b)+1}$ to $\Z/2$ of degree $1-n$
\[
\tilde I^\b_{0,I}:\left(\pi^\mM\right)^*\mathcal X_L^{\mu(\b)+1}\to \zcort{U_{0,l}^\b}.
\]
\end{lemma}
\begin{proof}
By Lemma~\ref{unique extension of bundle maps}, it is enough to construct $\tilde I^\b_{k+1,I}$ on $\mM^{(0)}_{k+1,I}(\b)$. Recall Definition~\ref{moduli spaces of families definition}. To an arbitrary point $(t,\fu)=(t,D^2,u,\vec z,\vec w)\in \mM^{(0)}_{k+1,I}(\b)$, associate the Cauchy-Riemann $Pin$ boundary value problem $\underline D_\fu$ over the disk, with \[E_\fu=u^* TX, \,\,F_\fu=u^*TL,\] with the $Pin$ structure $\fp_\fu$ on $F_\fu$ pulled-back from the corresponding structure on $L$, and $D_\fu$ the linearization of the $J_t$-holomorphic map operator.

We show that there exists a canonical isomorphism \[\ker(D_\fu)\simeq\ker dU_{k+1,I}|_{(t,\fu)},\]
which will imply, since $D_\fu$ and $dU_{k+1,I}|_{(t,\fu)}$ are surjective, a canonical isomorphism
\begin{equation}
    \label{det Du}
\zcort{U_{k+1,I}}|_{(t,\fu)}\simeq \det (\underline D_\fu).
\end{equation}
Denote by 
\[
\widehat{\mM_{k+1,I}}(\b):=\left\{(u,t)\mid\begin{smallmatrix}t\in\W,\\u:(D^2,\pa D^2)\to (X,L) \,\,J_t-\text{holomorphic,}\\
u_*\left([D^2,\pa D^2]\right)=\b\end{smallmatrix}\right\}\times_\W \widehat{\mW_{k+1,I}} 
\]
Then 
\[
\mM^{(0)}_{k+1,I}(\b)=\widehat{\mM_{k+1,I}}(\b)/\text{Aut}(D^2).
\]
To see the existence of the canonical isomorphism~\eqref{det Du}, we examine the following commutative diagram, of which the columns and bottom rows are exact.
\begin{equation}
    \begin{tikzcd}
        &&0\ar[d]&0\ar[d]&\\
        &0\ar[r]\ar[d]&\ker D_\fu\ar[d]\ar[r]&\ker dU_{k+1,I}|_{(t,\fu)}\ar[d]\ar[r]&0\\
        0\ar[r]&\mathfrak g\ar[d]\ar[r]&T_\fu\widehat{ \mM_{k+1,I}}(\b)\ar[d]\ar[r]&T_\fu\mM_{k+1,I}(\b)\ar[d]\ar[r]&0\\
        0\ar[r]&\mathfrak g\ar[d]\ar[r]&T_\fu\widehat{\mW_{k+1,I}}\ar[d]\ar[r]&T_\fu\mW_{k+1,I}\ar[d]\ar[r]&0\\&0&0&0&
    \end{tikzcd}
\end{equation}
By the nine lemma, the top row is exact as well, and thus the canonical isomorphism~\eqref{det Du} exists.

Consider the case $k\geq 0$. In \cite[Section 3]{JakePhD} a canonical isomorphism \[\det(D_\fu)\simeq\left(\zort{L}|_{u(z_0)}\right)^{\otimes (\mu(\b)+1)}\]
is constructed, which depends continuously on $\mathfrak u$.
Globally, under the identification~\eqref{det Du}, this isomorphism implies an isomorphism of bundles
\[
\tilde I_{k+1,l}^\b:{(evb_0^\b)}^*\zort{L}^{\otimes (\mu(\b)+1)}\to \zcort{U_{k+1,I}}.
\]
This isomorphism is continuous due to Lemma~\ref{Cauchy Riemann boundary value problem properties lemma}.
Note that $\text{dim}U_{k+1,I}^\b=n+\mu(\b)$, and $\deg \zort{L}=-1$, which implies $\deg \tilde I_{k+1,l}^\b=1-n$.

Consider the case $k=-1$. $F_\fu$ is orientable exactly if $\mu(\b)$ is even. Thus, recalling Remark~\ref{exponent for L cases remark}, the admissibility of $\b$ implies that either $F_\fu$ is non-orientable, if $\mu(\b)\equiv_21$, or $L$ is vertically oriented and thus $F_\fu$ is equipped with the pullback orientation. By Lemma~\ref{Cauchy Riemann boundary value problem properties lemma}, $\det(D_\fu)$ is canonically oriented. In case $L$ is vertically orientable, reversing the orientation of $L$ results in reversing the orientation of $F_{\fu}$. By Lemma~\ref{Cauchy Riemann boundary value problem properties lemma}, the latter reverses the canonical orientation of $\det(D_\fu)$ Thus, we get an isomorphism
\[
\mathcal X_L^{\mu(\b)+1}|_t\to \zcort{U_{0,I}^\b}|_{(t,\mathfrak u)}.
\]
Moreover, the canonical orientation is continuous in $(t,\fu)$. Globally, this implies an isomorphism of bundles
\[
\tilde I_{0,I}^\b:\left(\pi^\mM\right)^*\mathcal X_L^{\mu(\b)+1}\to \zcort{U_{0,I}^\b}.
\]
Its degree is $1-n.$
\end{proof}
\begin{remark}\label{tilde I forgetting marked point equation compatibility remark}
Let $k\geq 0, I\subset [l]$ be such that $k+2|I|\geq 2$. Let $\hat I=I\dot\cup \{l+1\}$.
Consider the essential pullback diagrams~\eqref{forgetting boundary point is functorial diagram}.
The proof of Lemma~\ref{Jake Thesis lemma} implies the following equations.
\begin{align*}
\expinv{\left(Fi_{k+1,I}^\b/Fi_{k+1,I}\right)}\tilde I_{k+1,I}^\b&=\tilde I_{k+1,\hat I}^\b,\\
\expinv{\left(Fb_{k+1,I}^\b/Fb_{k+1,I}\right)}\tilde I_{k+1,I}^\b&=\tilde I_{k+2,I}^\b
\end{align*}
Moreover, let $k=-1, I\subset[l]$ be such that $|I|\geq 2$. Let $\hat I=I\dot\cup\{l+1\}$. Consider the essential pullback diagram~\eqref{forgetting boundary point is functorial diagram}.
We have
\begin{equation}
\expinv{\left(Fi_{0,I}^\b/ Fi_{0,I}\right)}\tilde I^\b_{0,I} = \tilde I^\b_{0,\hat I}.
\label{tilde J forgetting boundary point equation compatibility k=-1}    
\end{equation}
\end{remark}

\begin{remark}
\label{rotation from 0 to 1 remark}
Recall Definition~\ref{boundary transport}. The map $\tilde I_{k+1,I}^\b$ is independent on the point $z_0$. This is incorporated in the following equation of \orientor{U_{k+1,I}^\b} of ${evb_1^\b}^*\left(\zort{L}^{\otimes (\mu(\b)+1)}\right)$ to $\underline{\Z/2}$,
\[
\expinv{\left(f^\b/f\right)}\tilde I_{k+1,I}^\b=\tilde I_{k+1,I}^\b\bu c_{01}^{\otimes (\mu(\b)+1)},
\]
where the pullback is taken with respect to the pullback diagram~\eqref{cyclic map pullback diagram}.
\end{remark}
\begin{definition}
\label{1-n translation L orientor definition}
Let
\[
e:=e_L:\zort{L}\to \zcort{L}
\]
be the \orientor{ \pi^L} of $\zort{L}$ to $\underline{\Z/2}$, of degree $1-n$, given by level-translation.
\end{definition}

Recall the map $Ev_{k+1,I}^\b$ from Definition~\ref{interior forgetful map definition}.
\begin{definition}\label{Jakethesis}
\leavevmode
For $k\geq0$, let \begin{equation}
\tilde{J}_{k+1,I}^\b:({evb}^\b_{z_0})^*\Big(\zort{L}^{\otimes {\mu(\b)}}\Big)\overset{\sim}\longrightarrow \zcort{Ev^\b_{k+1,I}}
\label{definition tilde J}    
\end{equation} be the unique \orientor{Ev_{k+1,I}^\b} of $({evb}^\b_{z_0})^*\Big(\zort{L}^{\otimes {\mu(\b)}}\Big)$ to $\underline{\Z/2}$,
satisfying
\[
\tilde I_{k+1,I}^\b={}^{\mW_{k+1,I}}e\bu {}^{\zort{}} \tilde J_{k+1,I}^\b.
\]
Existence and uniqueness is guaranteed by Lemma~\ref{orientor invertion lemma}.

Consider $k=-1,\b\in \Pi^{ad}$. Let \begin{equation}
\tilde{J}_{0,I}^\b:\left(\pi^\mM\right)^*\mathcal X_L^{\mu(\b)+1}\overset{\sim}\longrightarrow \zcort{U^\b_{0,I}}
\label{definition tilde J k=-1}    
\end{equation}be the \orientor{U_{0,I}^\b} of $\left(\pi^\mM\right)^*\mathcal X_L^{\mu(\b)+1}$ to $\underline{\Z/2}$ given by
\begin{equation*}
\tilde{J}_{0,I}^\b:=\tilde{I}_{0,I}^\b.
\end{equation*}
\end{definition}


\begin{remark}\label{tilde J for bottom cases b0 remark}
For $(k,l)\in \{(2,0),(0,1)\}$, the map $Ev_{k+1,l}^{\b_0}$ is a diffeomorphism.
It follows from the construction of $\tilde I_{k,l}^\b$ that,
\[
\tilde J_{k+1,l}^{\b_0} = \phi_{Ev_{k+1,l}^{\b_0}}.
\]
\end{remark}
\begin{lemma}\label{tilde J forgetting marked point compatibility lemma}
Let $k\geq 0, I\subset [l]$ be such that $k+2|I|\geq 2$. Let $\hat I=I\dot\cup \{l+1\}$. Consider the following diagrams.
\begin{equation}
\label{forgetting boundary point is functorial Ev diagram}
\begin{tikzcd}
\mM_{k+2, I}(\b)\ar[rr,"Fb_{k+1,I}^\b"]\ar[d,swap,"Ev_{k+2, I}"]&&\mM_{k+1,I}(\b)\ar[d,"Ev_{k+1,I}"]\\
L\times_\W\mW_{k+2,I}\ar[rr,swap,"\Id\times Fb_{k+1,I}"]&&L\times_\W\mW_{k+1,I}
\end{tikzcd}\qquad
\begin{tikzcd}
\mM_{k+1,\hat I}(\b)\ar[rr,"Fi_{k+1,I}^\b"]\ar[d,swap,"Ev_{k+1,\hat I}"]&&\mM_{k+1,I}(\b)\ar[d,"Ev_{k+1,I}"]\\
L\times_\W\mW_{k+1,\hat I}\ar[rr,swap,"\Id\times Fi_{k+1,I}"]&&L\times_\W\mW_{k+1,I}
\end{tikzcd}
\end{equation}The fact that diagrams~\eqref{forgetting boundary point is functorial diagram} are essential pullbacks implies that diagrams~\eqref{forgetting boundary point is functorial Ev diagram} are essential pullbacks as well. We have,
\begin{align*}
\expinv{\left(Fi_{k+1,I}^\b/\left(\Id\times Fi_{k+1,I}\right)\right)}\tilde J^\b_{k+1,I} &= \tilde J^\b_{k+1,\hat I},\\
\expinv{\left(Fb_{k+1,I}^\b/\left(\Id\times Fb_{k+1,I}\right)\right)}\tilde J^\b_{k+1,I} &= \tilde J^\b_{k+2,I}.
\end{align*}
Moreover, setting $\tilde{Ev}_{k+1,I}=\left(evb_1^\b,U_{k+1,I}^\b\right)$, the following diagram is a pullback diagram,
\[
\begin{tikzcd}
\mM_{k+1,I}(\b)\ar[r,"f^\b"]\ar[d,"\tilde{Ev}_{k+1,I}"]&\mM_{k+1,I}(\b)\ar[d," Ev_{k+1,I}"]\\L\times_\W \mW_{k+1,I}\ar[r,"\Id_L\times f"]&L\times_\W \mW_{k+1,I}
\end{tikzcd}
\]
and \[
\expinv{\left(f^\b/(\Id\times f)\right)}\tilde J^\b_{k+1,I}=\tilde J^\b_{k+1,I}\bu c_{01}^{\otimes(\mu(\b)+1)}.
\]
\end{lemma}
\begin{proof}
We prove for the case $Fb$. The proof for $Fi$ is the same. Consider the following diagram of pullback squares.
\[
\begin{tikzcd}
\mM_{k+2,I}(\b)\ar[r,"Fb^\b"]\ar[d,"Ev_{k+2,I}"]&\mM_{k+1,I}(\b)\ar[d,"Ev_{k+1,I}"]\\
L\times_\W\mW_{k+2,I}\ar[r,"\Id\times Fb"]\ar[d,"\pi_{\mW}"]&L\times_\W\mW_{k+1,I}\ar[d,"\pi_{\mW}"]\\
\mW_{k+2,I}\ar[r,"Fb"]&\mW_{k+1,I}
\end{tikzcd}
\]
We have
\[
\expinv{\left(Fb^\b/Fb\right)}\tilde I_{k+1,I}^\b\overset{\text{Lem.~\ref{vertical pullback expinv composition}}}=\expinv{\left(\Id\times Fb/Fb\right)}{}^{\mW_{k+1}}e\bu\expinv{\left(Fb^\b/\Id\times Fb\right)}\tilde J_{k+1,I}^\b.
\]
It is immediate that
\[\expinv{\left(\Id\times Fb/Fb\right)}{}^{\mW_{k+1}}e={}^{\mW_{k+2}}e.\]
Thus,
\[
{}^{\mW_{k+2}}e\bu\expinv{\left(Fb^\b/\Id\times Fb\right)}\tilde J_{k+1,I}^\b= \expinv{\left(Fb^\b/Fb\right)}\tilde I_{k+1,I}^\b\overset{\text{Rmk.~\ref{tilde I forgetting marked point equation compatibility remark}}}=\tilde I_{k+2,I}^\b={}^{\mW_{k+2}}e\bu \tilde J_{k+2,I}^\b.
\]
By Lemma~\ref{orientor invertion lemma} we get the result for $Fb$.
The case of $f$ is a similar argument based on Lemmas~\ref{vertical pullback expinv composition} and~\ref{orientor invertion lemma}.
\end{proof}
Use $\mW$ to denote $\mW_{k+1,I}$, and consider the following direct product diagram.
\[
\begin{tikzcd}
L\times_\W \mW\ar[r," \pi^{L\times \mW}_\mW"]\ar[d,swap," \pi^{L\times \mW}_L"]&\mW\ar[d," \pi^\mW"]\\L\ar[r,swap," \pi^L"]&\W
\end{tikzcd}
\]
Recall Definition~\ref{pulled orientation orientor}, and consider the orientor
\[\phi_{k+1,I}^L:=\expinv{\left( \pi^{L\times \mW}_\mW/ \pi^L\right)}\phi_{k+1,I}.\]
$\tilde J^\b_{k+1,I}$ and $\phi^L_{k+1,I}$ are orientors that can be composed.
Note that
\[
evb_0^\b= \pi^{L\times \mW}_L\circ Ev_0^\b.
\]
\begin{definition}
\label{fundamental orientor definition}
For $k\geq 0$, denote by $J^\b_{k+1,I}$ the \orientor{evb_0^\b} of degree $2-k-2|I|$ 
\[
J^\b_{k+1,I}:{(evb_0^\b)}^*\zort{L}^{\otimes \b}\to \zcort{evb_0^\b},
\] 
given as follows.
For $k+2|I|\geq 2$, it is given by the equation
\[
J^\b_{k+1,I}:=\phi^L_{k+1,I}\bu\tilde J^\b_{k+1,I}.
\]
For $k+2|I|<2$, it is defined by the equation
\[
\expinv{\left(Fi_{k+1,I}^\b\right)} J^\b_{k+1,I} = J^\b_{k+1,\hat I}.
\]
Using Lemma~\ref{orientor on pullback is pulled back} and the fact that the fibers of $Fi_{k+1,I}^\b$ are connected, such $J^\b_{k+1,I}$ exists and is unique.

Denote by $ \pi^{\mM_{k+1,I}(\b)}:\mM_{k+1,I}(\b)\to \W$ the projection.
For $k=-1,\b\in \Pi^{ad}$, denote by $J^\b_{0,I}$ the \orientor{\pi^{\mM_{0,I}(\b)}} of degree $4-~n-~2|I|$
\[
J^\b_{0,I}:\left(\pi^{\mM_{0,I}(\b)}\right)^*\mathcal X_L^{\mu(\b)+1}\to \zcort{ \pi^{\mM_{0,I}(\b)}},
\]
given as follows. For $|I|\geq 2$, it is given by the equation
\[
J^\b_{0,I}:=\phi_{0,I}\bu \tilde J_{0,I}^{\b}.
\]
For $|I|<2$, it is defined by the equation
\[
\expinv{\left(Fi_{0,I}^\b\right)} J^\b_{0,I} = J^\b_{0,\hat I}.
\]
\end{definition}
\begin{remark}\label{J for bottom cases b0 remark}
It follows from Lemma~\ref{disk space orientors with properties existence lemma} and Remark~\ref{tilde J for bottom cases b0 remark} that for $(k,l)\in \{(2,0),(0,1)\}$ we have
\[
J_{k+1,l}^{\b_0} = \left(\phi_{\mW_{k+1,l}}\right)^L\bu \phi_{Ev_0}\overset{\text{Lem.~\ref{pulled orientation orientor is pullback orientor of orientation local diffeomorphism}}}=\phi_{\Id_L\times\mW_{k+1,l}}\bu \phi_{Ev_0}\overset{\text{Rmk. \ref{composition of local diffeomorphisms orientors lemma}}}=\phi_{evb_0}.
\]
\end{remark}
\begin{lemma}\label{properties of J lemma}
The above definitions are compatible, in the sense that
\begin{align*}
\expinv{\left(Fi_{k+1,I}^\b\right)} J^\b_{k+1,I} &= J^\b_{k+1,\hat I}    \\\expinv{\left(Fb_{k+1,I}^\b\right)}J_{k+1,I}^\b&=J^\b_{k+2,I}.
\end{align*}
For all $k,I$.
Moreover, for $k\geq 0,$
\[
\expinv{f^\b}J^\b_{k+1,I}=(-1)^kJ^\b_{k=1,I}\bu c_{0i}^{\otimes (\mu(\b)+1)}.
\]
\end{lemma}
\begin{proof}
We prove for $Fb$. The case $Fi$ is the same, including the case $k=-1$. We have,
\[
\expinv{\left({Fb}^\b,\mO^{{Fb}^\b}\right)} \left(\phi_{k+1,I}^L\bu\tilde J^\b_{k+1,I}\right)=\expinv{\left(\Id\times {Fb},1\times \mO^{Fb}\right)}\left(\phi_{k+1,I}^L\right)\bu \expinv{\left({Fb}^\b/\Id\times {Fb}\right)}\tilde J^\b_{k+1,I}.
\]
However,
\[
\expinv{\left(\Id\times {Fb},1\times \mO^{Fb}\right)}\left(\phi_{k+1,I}^L\right) \overset{\text{Lem.~\ref{base extension of pullback is extended pullback of base extension lemma}}}= \left(\expinv{\left({Fb}\right)}\phi_{k+1,I}\right)^L\overset{\text{Lem.~\ref{disk space orientors with properties existence lemma}}}=\left(\phi_{k+2, I}\right)^L.
\]
Combining the above results with Lemma~\ref{tilde J forgetting marked point compatibility lemma}, we obtain the claim for $Fb$. The proof for $f$ is left to the reader. One might want to use Lemma~\ref{pullback of composition of orientors by relatively oriented} and recall Lemma~\ref{disk space orientors with properties existence lemma}.
\end{proof}

\subsection{Gluing map signs}\label{Gluing map signs section}
Denote by \[\d:=\begin{cases}
0,& \fp\text{ is a } Pin^+ \text{ structure},\\
1,& \fp\text{ is a } Pin^- \text{ structure}.
\end{cases}\] 
For $k\geq 0$, consider the following fiber-product diagram.
\[
\begin{tikzcd}
\mdl{1}\times_L\mdl{2}\ar[r,"p_2"]\ar[d,"p_1"]&\mdl{2}\ar[d,"evb_0^{\b_2}"]\\\mdl{1}\ar[r,"evb_i^{\b_1}"]&L
\end{tikzcd}
\]
Recall the boundary of orientors from Definition~\ref{Boundary of orientor}.
The main results of this section are the following two theorems.
\begin{theorem}
\label{expinv of boundary of J by thetabeta theorem}
Let $k\geq 0$ and let $\b_1+\b_2=\b,$ $k_1+k_2=k+1,I\dot\cup J=[l]$. 
The following equation of \orientor{\left(evb_0^\b\circ\iota^\b\circ\vartheta^\b\right)}s of
${\left(evb_0^\b\circ \iota^\b\circ\vartheta^\b \right)}^*\left(\zort{L}^{\otimes \b_2}\otimes \zort{L}^{\otimes \b_1}\right)$ to $\underline{\Z/2}$ holds,
\[
\expinv{\left(\vartheta^\b\right)}\pa J^\b_{k+1,l}\bu p_1^*\left(evb_0^{\b_1}\right)^*m_2=(-1)^sJ^{\b_1}_{k_1+1,I}\bu \left(\expinv{\left(p_2/evb_i^{\b_1}\right)}J^{\b_2}_{k_2+1,J}\right)^{\left(evb_0^{\b_1}\right)^*\zort{L}^{\otimes \b_1}}\!\!\!\bu(p_1^*c_{i0})^{\zort{L}^{\otimes \b_1}}
\]
with
\[
s=i+ik_2+k+\d\mu(\b_1)\mu(\b_2).
\]
\end{theorem}
The proof is given in Section~\ref{Gluing proof section}.

For $b\in \Z$ an exponent for $L$, denote by ${i}_{b}:\pi_L^*\mathcal X^{b}\to \zort{L}^{b}$ the canonical map, which is an isomorphism since $b$ is an exponent for $L$ and the fiber of $\pi^L$ is connected.
\begin{theorem}
\label{expinv of boundary of J by thetabeta theorem k=-1}
Let $\b\in \Pi^{ad}$. Let $\b_1+\b_2=\b$, and $I\dot\cup J=[l]$. 
As \orientor{ \pi^{\mM_1\times_L\mM_2}}s of
$\underline{\mathcal X_L^{\mu(\b)+1}}$ to $\underline{\Z/2}$,
it holds that
\[
\expinv{\left(\vartheta^\b\right)}\pa J^\b_{0,l}
=-e\bu \left(J_{1,I}^{\b_1}\bu\left(\expinv{\left(p_2/evb_0^{\b_1}\right)}J_{1,J}^{\b_2}\right)^{\zort{L}^{\b_1}}\right)^{\zort{L}}\bu \left(evb_0^{\b_1}\circ p_1\right)^*{i}_{\mu(\b)+1}.
\]
\end{theorem}
The proof is given in Section~\ref{Gluing proof section k=-1}.

\subsection{Conventions and notations for boundary components}\label{conventions for boundary section}
The following notations will be used in Sections~\ref{Gluing proof section} and~\ref{Gluing proof section k=-1} below.
Let $k\geq-1,l\geq 0,\b\in\Pi$. Fix partitions $k_1+k_2=k+1,\b_1+\b_2=\b$ and $I\dot\cup J=[l]$, such that $k_1>0$ if $k\geq 0$. In all lemmas leading to the proofs of Theorems~\ref{expinv of boundary of J by thetabeta theorem} and~\ref{expinv of boundary of J by thetabeta theorem k=-1}, but not in the proofs themselves, we assume
\begin{equation}\label{Stability condition for mW1 and mW2 equations} 
k_1+2|I|\geq 2, \qquad k_2+2|J|\geq 2.
\end{equation}
When $k\geq 0$, let $0< i\leq k_1$. When $k=-1$ let $i=0$. Set
\begin{align*}
\begin{matrix}
\mW=\mW_{k+1,[l]},&\mW_1=\mW_{k_1+1,I},&\mW_2=\mW_{k_2+1,J}\\
\phi=\phi_{k+1,l},&\phi_1=\phi_{k_1+1,I},&\phi_2=\phi_{k_2+1,J}\\
\mM=\mdl{3},&\mM_1=\mdl{1},&\mM_2=\mdl{2},\\
\tilde J=\tilde J_{k+1,l}^{\b}&\tilde J^1=\tilde J_{k_1+1,I}^{\b_1}&\tilde J^2=\tilde J_{k_2+1,J}^{\b_2},\\
J=J_{k+1,l}^{\b},&J^1:=J_{k_1+1,I}^{\b_1}&J^2:=J_{k_2+1,J}^{\b_2},
\end{matrix}
\end{align*}
where $\mM,\mW$ possibly carry a superscript $(r)$, for $r=0,1,2,...$.
Moreover, set
\begin{align*}
    \begin{matrix}
    B=B_{i,k_1,k_2,I,J}(X=pt),&
    B(\b)=B_{i,k_1,k_2,I,J}(\b_1,\b_2),&\\
    B^{(1)}=B^{(1)}_{i,k_1,k_2,I,J},&B^{(1)}(\b)=B^{(1)}_{i,k_1,k_2,I,J}(\b_1,\b_2),&\\
    \vartheta=\vartheta_{i,k_1,k_2,I,J}^{X=L=pt},&
    \vartheta^\b=\vartheta_{i,k_1,k_2,\b_1,\b_2,I,J},&\\
    \iota=\iota_{i,k_1,k_2,I,J}^{X=L=pt},&\iota^\b=\iota^{\b_1,\b_2}_{i,k_1,k_2,I,J},&
    \end{matrix}
\end{align*}
Denote by
\begin{align*}
q_j:&\mW_1\times_\W\mW_2\to\mW_j,
\end{align*}
 the projections to the $j$th coordinates, where $j=1,2$.
\begin{remark}[Boundary marked points on $B(\b)$]
\label{boundary marked points on B}
The following equations hold.
\[
evb^\b_j\circ \iota^\b\circ \vartheta^\b=
\begin{cases}
evb^{\b_1}_j\circ p_1&,\quad 0\leq j\leq i-1,\\
evb^{\b_2}_{j-i+1}\circ p_2&,\quad i\leq j\leq i+k_2-1,\\
evb^{\b_1}_{j-k_2+1}\circ p_1&,\quad i+k_2\leq j\leq k.
\end{cases}
\]
\[
evb_i^{\b_1}\circ p_1=evb_0^{\b_2}\circ p_2
\]
\end{remark}

Abbreviate
\begin{align*}
\begin{matrix}
U=U_{k+1,[l]},& Ev=Ev_{k+1,[l]}^{\b},\\
U_1=U_{k_1+1,I},& Ev_1=Ev_{k_1+1,I}^{\b_1},\\
U_2=U_{k_2+1,J},& Ev_2=Ev_{k_2+1,J}^{\b_2}.
\end{matrix}
\end{align*}

Write $\iota:B\to \mW$ for the inclusion of the boundary.
The maps $U,Ev^\b$ preserve boundaries, up to codimension 2, in the sense of the following lemma.

Recall the Definition~\ref{essential boundary factorization} of essential boundary factorization.
\begin{lemma}\label{essential factorization of Ev lemma}
\label{stability on B}
For $k\geq -1$ the map $U$ essentially factorizes through the boundary of $ \pi^\mW$. In particular, for $k\geq 0$ the map $Ev^\b$ essentially factorizes through the boundary of $ \pi^{L\times \mW}_L$. Namely, there exist maps 
\begin{align*}
\iota^*U:B(\b)&\to B,\\
\iota^*Ev^\b:B(\b)&\to L\times_\W B,
\end{align*}
such that the following diagrams are fiber-products.
\[
\begin{tikzcd}
B^{(1)}(\b)\ar[d,dashed,"\iota^*U"]\ar[r,"\iota^\b"]&\mM^{(1)}\ar[d,"U"]\\
B^{(1)}\ar[r,"\iota"]&\mW^{(1)}
\end{tikzcd}
\qquad\text{and}\qquad
\begin{tikzcd}
B^{(1)}(\b)\ar[d,dashed,"\iota^*Ev^\b"]\ar[r,"\iota^\b"]&\mM^{(1)}\ar[d,"Ev^\b"]\\
L\times_\W B^{(1)}\ar[r,"\Id\times\iota"]&L\times_\W \mW^{(1)}
\end{tikzcd}
\]
Moreover, 
\begin{equation}
    \label{iota Ev}
    \iota^*Ev^\b=(evb_0^\b\circ\iota^\b,\iota^*U).
\end{equation}
\end{lemma}
\begin{remark}
The maps $\Id\times\iota$ and $Ev^\b$ are transversal.
\end{remark}
\begin{proof}
By assumption~\eqref{Stability condition for mW1 and mW2 equations}, there are enough marked points on each component of $\S$ at points of $B^{(1)}(\b)$ that forgetting the holomorphic map never results in unstable components. Thus, the domain of the holomorphic map does not change. 
\end{proof}
\subsection{Proof for \texorpdfstring{$k\geq 0$}{k>-1}}
\label{Gluing proof section}
In this section, we prove Theorem~\ref{expinv of boundary of J by thetabeta theorem}. We assume $k\geq 0$, unless stated otherwise.

Recall the boundary map from Definition~\ref{boundary-operator for relative orientation}.
The following lemma holds for $k=-1$ as well.
\begin{lemma}[Orientation of gluing map of disks]
\label{moduli of disks orientation}  
The following equation of \orientor{\pi^{\mW_1\times\mW_2}} holds.
\begin{equation}
\expinv{\vartheta}(\pa\phi)=(-1)^{i+ik_2+k}\,\,\phi_{1}\bu \left(\phi_{2}\right)^{\mW_1},
\label{moduli of disks orientation lemma orientors equation}
\end{equation}
where $\pa\phi$ is the boundary of $\phi$ from Definition~\ref{Boundary of orientor}
\end{lemma}
\begin{proof}
First, we reduce to the case $\W$ is a point. Write $\tilde\phi,\tilde \vartheta$ for $\phi,\vartheta$ in the case $\W$ is a point. Then
\[
\phi=\tilde\phi^\W,\qquad \vartheta=\Id_\W\times \tilde\vartheta.
\]
Assuming we proved the lemma for that case, we get
\begin{align*}
\expinv{\vartheta}\pa\phi=\expinv{\left(\Id_\W\times \tilde\vartheta\right)}\pa\tilde\phi^\W\overset{\begin{smallmatrix}\text{Lem.~\ref{base extension of pullback is extended pullback of base extension lemma}}\\\text{Lem.~\ref{boundary of pulled orientation orientor}}\end{smallmatrix}}=\left(\expinv{\tilde\vartheta}\pa\tilde\phi\right)^\W
\overset{\text{rdct.}}=(-1)^{i+ik_2+k}\left(\tilde\phi_1\bu(\tilde \phi_2)^{\tilde \mW_1}\right)^\W
\\\overset{\text{Lem~\ref{base extension is distributive expinv}}}=(-1)^{i+ik_2+k}\phi_1\bu\phi_2^{\tilde \mW_1}.
\end{align*}
Therefore, we continue with the assumption $\W$ is a point.
Let $\mO^{\mW_i}$, possibly with no subscript, denote the orientation of $\mW_i$ defining $\phi_{i}$. Our orientation convention of the absolute boundary agrees with that of \cite{Sara1}. See Remark~\ref{boundary orientation comparison}. Rewriting~\cite[Proposition 2.8]{Sara1} in the case $X,L,$ are both points in the notation of the present work, we obtain
\begin{equation}
\label{orientations calculation in boundary phi proof equation}
\mO^{\mW}\circ\mO^{\iota}\circ\mO^\vartheta_c=(-1)^{i+ik_2+k_1k_2}\mO^{\mW_1}\times \mO^{\mW_2}.
\end{equation}
See also~\cite[Proposition 8.3.3]{Fukaya} in the case $i=1$.

The bundle-map $(-1)^k\pa$ is given by composition with $\mO^{\iota}$ on the right, so equation~\eqref{orientations calculation in boundary phi proof equation} is equivalent to the commutativity of the following diagram up to the sign $(-1)^{i+ik_2+k_1k_2}$.
\begin{equation}
\begin{tikzcd}
\underline{\Z/2}\ar[rr,"\vartheta^*\iota^*\phi"]\ar[d,equal]&&\vartheta^*\iota^*\zcort{\mW}\ar[r,"(-1)^{k}\vartheta^*\pa"]&\vartheta^*\zcort{B}\ar[d,"\pull{\vartheta}"]
\\\underline{\Z/2}\otimes \underline{\Z/2}\ar[rr,swap,"q_2^*\phi_2\otimes q_1^*\phi_1"]&&q_2^*\zcort{\mW_2}\otimes q_1^*\zcort{\mW_1}\ar[r,swap,"\sim"]&\zcort{\mW_1\times\mW_2}
\end{tikzcd}
\label{moduli of disks orientation lemma diagram}
\end{equation}

Equation~\eqref{moduli of disks orientation lemma orientors equation} follows from diagram~\eqref{moduli of disks orientation lemma diagram}, using the Kozsul signs~\ref{Koszul signs} and Example~\ref{boundary of orientation as orientor} as follows.
\begin{multline*}
\phi_1\bu (\phi_2)^{\mW_1}=(\Id\otimes q_1^*\phi_1)\circ q_2^*\phi_2=(-1)^{k_1k_2}q_2^*\phi_2\otimes q_1^*\phi_1\\=(-1)^{i+ik_2+k}\pull{\vartheta}\circ \vartheta^*\pa \circ \vartheta^*\iota^*\phi=(-1)^{i+ik_2+k}\expinv{\vartheta}\left(\pa \phi\right).
\end{multline*}
\end{proof}
Denote by $c_{0l},c^1_{0l}$ and $c^2_{0l}$ the boundary transport maps on $\mM,\mM_1$ and $\mM_2$, respectively. Recall Remark~\ref{boundary marked points on B}.
\begin{lemma}[Rotation around $\b_2$]
\label{rotate around b2}
The boundary transport map $c_{0l}$ may be reconstructed on the boundary component $B(\b)$ as follows.
\[
{(\vartheta^\b)}^* c_{0j}=
\begin{cases}
p_1^*c_{0j}^{1},& 0\leq j\leq i-1,\\
p_1^*c^{1}_{0i}\circ p_2^*c^{2}_{0,j-i+1},& i\leq j\leq i+k_2-1,\\
(-1)^{w_1(\pa\b_2)}\cdot\,\, p_1^*c^{1}_{0,j-k_2+1},& i+k_2\leq j\leq k.
\end{cases}
\]
\end{lemma}
\begin{proof}
The lemma follows from Definition~\ref{boundary transport} and Remark~\ref{full round}.
\end{proof}
Set $E:=Ev_1\times U_2$. That is,
\[
E=(Ev_1\times\Id)\circ (\Id\times U_2).
\]
Consider the following commutative diagram, in which the square is a fiber-product.
\begin{equation}
\begin{tikzcd}
\mM_1\times_L\mM_2\ar[r,"p_2"]\ar[d,swap,"\Id\times U_2"]\ar[dd,bend right=80,swap,"E"]&\mM_2\ar[d,"Ev_2"]\\
\mM_1\times_\W \mW_2\ar[r,swap,"evb_i^{\b_1}\times\Id"]\ar[d,swap,"Ev_1\times\Id"]&L\times_\W\mW_2\\
L\times_\W \mW_1\times_\W\mW_2
\end{tikzcd}
\label{E as fibered composition}
\end{equation}
Keeping in mind Lemma~\ref{essential factorization of Ev lemma}, all squares in the following diagram are fiber products.
\begin{equation*}
\label{theta beta over theta essential}
  \begin{tikzcd}
\mM^{(0)}_1\times_L\mM^{(0)}_2\ar[d,swap,"E"]\ar[r,"\vartheta^\b"]&B^1(\b)\ar[d,"\iota^*Ev"]\ar[r,"\iota^\b"]&\mM^{(1)}\ar[d,"Ev"]
\\
L\times_\W\mW^{(0)}_1\times_\W \mW^{(0)}_2\ar[r,swap,"\Id\times \vartheta"]&L\times_\W B^{(1)}\ar[r,swap,"\Id\times\iota"]&L\times_\W \mW
\end{tikzcd}
\end{equation*}
In particular, all three squares in the following diagram are essential fiber products.
\begin{equation}
\label{theta beta over theta}
  \begin{tikzcd}
\mM_1\times_L\mM_2\ar[d,swap,"E"]\ar[r,"\vartheta^\b"]&B(\b)\ar[d,"\iota^*Ev"]\ar[r,"\iota^\b"]&\mM\ar[d,"Ev"]
\\
L\times_\W\mW_1\times_\W \mW_2\ar[r,swap,"\Id\times \vartheta"]&L\times_\W B\ar[r,swap,"\Id\times\iota"]&L\times_\W \mW
\end{tikzcd}
\end{equation}

The proof of the following lemma is similar to proofs that appear in \cite{Fukaya}. It relies on the fact that $e_1e_1=(-1)^\d$ in $\fp$. 
\begin{lemma}\label{axiomb1b2}
The following diagram is commutative, up to the sign $(-1)^{\d\mu(\b_1)\mu(\b_2)}$,
\begin{equation}
\begin{tikzcd}
{\vartheta^\b}^*\iota^*{(evb_0^\b)}^*\zort{L}^{\otimes \b}\ar[from=d,"p_1^*{evb_0^{\b_1}}^*m_2"]\ar[rrrrr,"{\vartheta^\b}^*\iota^*\tilde J"]&&&&&{\vartheta^\b}^*\zcort{Ev}\ar[dd,swap,"\pull{\left(\left(\iota^\b\circ\vartheta^\b\right)/\left((\Id\times\iota)\circ (\Id\times\vartheta)\right)\right)}"]
\\
p_1^*{evb_0^{\b_1}}^*\left(\zort{L}^{\otimes\b_2}\otimes\zort{L}^{\otimes\b_1}\right)\ar[d,swap,"p_1^*c_{i0}\otimes \Id"]\\
\begin{matrix}
p_2^*{evb_0^{\b_2}}^*\zort{L}^{\otimes \b_2}\bigotimes\\ p_1^*{evb_0^{\b_1}}^*\zort{L}^{\otimes \b_1}
\end{matrix}
\ar[rrrr,swap,"\left(\expinv{\left(p_2/\left(evb_i^{\b_1}\times\Id\right)\right)}\tilde J^2\right)^{{evb_0^{\b_1}}^*\mathbb L_L^{\otimes \b_1}}"]&&&&\zcort{\Id\times U_2}\otimes p_1^*{evb_0^{\b_1}}^*\zort{L}^{\otimes \b_1}\ar[r,swap,"{}^{\mW_2}\tilde J^1"]
&\zcort{E}
\end{tikzcd}
\end{equation}
where $m_2$ is the canonical multiplication $m_2:\zort{L}^{\otimes \b_2}\otimes \zort{L}^{\otimes \b_1}\to \zort{L}^{\otimes \b}$.
That is, we have an equality of \orientor{E}s of ${\vartheta^\b}^*\iota^*{(evb_0^\b)}^*\zort{L}^{\otimes\b_1+\b_2}$ to $\underline{\Z/2}$, as follows.
\begin{align*}
(-1)^{\d\mu(\b_1)\mu(\b_2)}\cdot\,\,&\expinv{\left(\left(\iota^\b\circ\vartheta^\b\right)/\left((\Id\times\iota)\circ (\Id\times\vartheta)\right)\right)}\tilde J\,\,\bu\,\, p_1^*{evb_0^{\b_1}}^*m_2\\
=&{}^{\mW_2}\tilde J^1\bu \left(\expinv{\left(p_2/\left(evb_i^{\b_1}\times\Id\right)\right)}\tilde J^2\right)^{{evb_0^{\b_1}}^*\zort{L}^{\otimes \b_1}}\bu \left(p_1^*c_{i0}\right)^{\zort{L}^{\otimes \b_1}}
\end{align*}

The pullback of $\tilde J$ on the left hand side is based on diagram~\eqref{theta beta over theta} and the pullbacks of $\tilde J^2,\tilde J^1$ are based on diagram~\eqref{E as fibered composition}.
\end{lemma}
\begin{lemma}
\label{boundary is fibered}
It holds that
\begin{equation}
\label{boundary is fibered equation}
\pa J^{\b}_{k+1,l}=\pa\phi^L_{k+1,l}\bu \expinv{\left(\iota^\b/(\Id\times\iota)\right)}\tilde J^\b_{k+1,l}.
\end{equation}
\end{lemma}
\begin{proof}
By Definition~\ref{fundamental orientor definition},
\[
J^{\b}_{k+1,l}=\phi^L_{k+1,I}\bu \tilde J^\b_{k+1,l}.
\]By Lemma~\ref{stability on B} we may apply Lemma~\ref{boundary of composition orientor} to get equation~\eqref{boundary is fibered equation}.
\end{proof}
\begin{lemma}
\label{phi squared times J squared is phi times J squared lemma}
The following equation holds.
\[
\left(\phi_{1}\bu \left(\phi_{2}\right)^{\mW_1}\right)^L\bu\,\,{}^{\mW_2}\tilde J^1\bu \left(\expinv{\left(p_2/(evb_i^{\b_1}\times\Id)\right)}\tilde J^2\right)^{\zort{L}^{\b_1}}=J^1\bu \left(\expinv{\left(p_2/evb_i^{\b_1}\right)}J^2\right)^{\zort{L}^{ \b_1}}
\]
\end{lemma}
\begin{proof}
First, note that
\begin{equation}
\label{disks extensions distributivity equation}   \left(\phi_{1}\bu \left(\phi_{2}\right)^{\mW_1}\right)^L\overset{\text{Lem.~\ref{base extension is distributive expinv}}}=\left(\phi_{1}\right)^L\bu \left(\left(\phi_{2}\right)^{\mW_1}\right)^L\overset{
\text{Lem.~\ref{base extension is multiplicative expinv}}}=\left(\phi_{1}\right)^L\bu \left(\phi_{2}\right)^{L\times\mW_1}.
\end{equation}
Consider the following diagram in which all three squares are Cartesian.
\[
\begin{tikzcd}
\mM_1\times_\W\mW_2\ar[r,"Ev_1\times\Id"]\ar[d," \pi_{\mM_1}"]&L\times_\W\mW_1\times_\W\mW_2\ar[d," \pi_{L\times \mW_1}"]\ar[r," \pi_{\mW_2}"]&\mW_2\ar[d]\\
\mM_1\ar[r,"Ev_1"]&L\times_\W\mW_1\ar[r]&*
\end{tikzcd}
\]
Lemma~\ref{horizontal expinv composition} implies
\[
\left(\phi_2\right)^{\mM_1}=\expinv{\left(Ev_1\times\Id/Ev_1\right)}\left(\left(\phi_2\right)^{L\times_\W\mW_1}\right).
\]
Applying Lemma~\ref{commutativity vertical pullback vs horizontal pullback orientors} to the left Cartesian square gives
\begin{equation}
\left(\phi_{2}\right)^{L\times_\W\mW_1}\bu\,\, {}^{\mW_2}\tilde J^1=\tilde J^1\bu \left(\left(\phi_2\right)^{\mM_1}\right)^{{evb_1^\b}^*\zort{L}^{\otimes \b_1}}.
\label{phi over tilde J1}
\end{equation}
Moreover, Lemmas~\ref{horizontal expinv composition} and~\ref{vertical pullback expinv composition} applied to the commutative diagram
\[
\begin{tikzcd}
\mM_1\times_L\mM_2\ar[r,"p_2"]\ar[d,swap,"\Id\times U_2"]&\mM_2\ar[d,"Ev_2"]\\
\mM_1\times_\W \mW_2\ar[r,swap,"evb_i^{\b_1}\times\Id"]\ar[d,swap," \pi_{\mM_1}"]&L\times_\W\mW_2\ar[d," \pi_L"]\ar[r,swap," \pi_{\mW_2}"]&\mW_2\ar[d]\\
\mM_1\ar[r,swap,"evb_i^{\b_1}"]&L\ar[r]&*
\end{tikzcd}
\]
give
\begin{equation}
\left(\phi_2\right)^{\mM_1}\bu \expinv{\left(p_2/\left(evb_i^{\b_1}\times\Id\right)\right)}\tilde J^2=\expinv{\left(p_2/evb_i^{\b_1}\right)}\left(\phi_2^L\bu \tilde J^2\right)=\expinv{\left(p_2/evb_i^{\b_1}\right)}J^2.
\label{phi over tilde J2}    
\end{equation}
Therefore, 
\begin{align*}
\left(\phi_{1}\bu \left(\phi_{2}\right)^{\mW_1}\right)^L\bu\,\,{}^{\mW_2}\tilde J^1\bu& \left(\expinv{\left(p_2/(evb_i^{\b_1}\times\Id)\right)}\tilde J^2\right)^{\zort{L}^{\b_1}}=\\
\overset{\text{eq.~\eqref{disks extensions distributivity equation}}}=&\left(\phi_{1}\right)^L\bu \left(\phi_{2}\right)^{L\times_\W\mW_1}\bu\,\,{}^{\mW_2}\tilde J^1\bu \left(\expinv{\left(p_2/(evb_i^{\b_1}\times\Id)\right)}\tilde J^2\right)^{\zort{L}^{\b_1}}\\
\overset{\text{eq.~\eqref{phi over tilde J1}+\eqref{phi over tilde J2}}}=&\left(\phi_1\right)^L\bu\tilde J^1\bu\left(\expinv{\left(p_2/evb_i^{\b_1}\right)}J^2\right)^{\zort{L}^{\b_1}}\\
\overset{\text{Def.~\ref{fundamental orientor definition}}}=&J^1\bu \left(\expinv{\left(p_2/evb_i^{\b_1}\right)}J^2\right)^{\zort{L}^{ \b_1}}
\end{align*}
\end{proof}
\begin{proof}[Proof of Theorem~\ref{expinv of boundary of J by thetabeta theorem}]
First, we show that it is enough to show the lemma for $|I|,|J|$ large enough. We show the reduction for $|I|$. The reduction for $|J|$ is similar. For this part only, for ease of notation, we denote spaces, maps, and orientors by their respective interior points index-set, for example, $\vartheta^\b_l, \mM_{I},\mM_{J},$ and $\vartheta^\b_{l+1}, \mM_{\hat I},\mM_{J}$. Moreover, denote $Fi^\b_{k+1,l}$ by $F_l$ and $Fi^{\b_1}_{k_1+1,I}$ by $F_I$.
Consider the following diagram of essential fiber products.
\[
\begin{tikzcd}
\mM_{\hat I}\times_L\mM_{J}\ar[r,"\vartheta_{l+1}"]\ar[d,"F_I\times\Id"]&B_{l+1}\ar[r,"\iota^\b_{l+1}"]\ar[d,"\iota^*F_l"]&\mM_{l+1}\ar[d,"F_l"]\\
\mM_{I}\times_L\mM_{J}\ar[r,"\vartheta_{l}"]&B_{l}\ar[r,"\iota_{l}^\b"]&\mM_{l}
\end{tikzcd}
\]
Since the fiber of $F_l$ is two-dimensional, it is easy to see that 
\begin{equation}
\pull{\left(\left(\vartheta^\b_{l+1}\circ \iota^\b_{l+1}\right)/\left(\vartheta^\b_{l}\circ \iota^\b_{l}\right)\right)}\left(\mO^{F_l}\right)=\mO^{F_I}\times1.
\label{pullback of orientation of forgetful map is orientation of forgetful map equation}    
\end{equation}
Observe,
\begin{align}
\label{pullback of boundary of J with extra point is remembering a point}
\expinv{\vartheta_{l+1}}\pa J_{l+1}\overset{\text{Lem.~\ref{properties of J lemma}}}=&\quad\expinv{\vartheta_{l+1}}\pa\left(\expinv{\left(F_l,\mO^{F_l}\right)}J_{l}\right)\\\notag
\overset{\text{Def.~\ref{pullback of orientor}}}=&\quad\expinv{\vartheta_{l+1}}\pa\left(J_{l}\bu \phi^{\mO^{F_l}}\right)\\\notag
\overset{\text{Lem.~\ref{boundary of composition orientor}}}=&\quad\expinv{\vartheta_{l+1}}\left(\pa J_{l}\bu \expinv{\left(\iota^\b_{l+1}/\iota^\b_l\right)}\left(\phi^{\mO^{F_l}}\right)\right)\\\notag
\overset{\text{Lem.~\ref{pullback of composition of orientors by relatively oriented}}}=&\quad\expinv{\vartheta_{l}}\pa J_{l}\bu \expinv{\left(\vartheta^\b_{l+1}/\vartheta^\b_l\right)}\expinv{\left(\iota^\b_{l+1}/\iota^\b_l\right)}\left(\phi^{\mO^{F_l}}\right)
\\\notag
\overset{\text{eq.~\eqref{pullback of orientation of forgetful map is orientation of forgetful map equation}}}=&\quad\expinv{\vartheta_{l}}\pa J_{l}\bu \phi^{\mO^{F_I}\times 1}\\\notag
\overset{\text{Def.~\ref{pullback of orientor}}}=&\quad
\expinv{\left(F_I\times\Id,\mO^{F_I}\times 1\right)}\expinv{\left(\vartheta^\b_l\right)}\pa J_{l}.
\end{align}

On the other hand, consider the following essential fiber product diagram.
\[
\begin{tikzcd}
\mM_{\hat I}\times_L\mM_{J}\ar[r,"F_I\times\Id"]\ar[d,"p_1"]&\mM_{I}\times_L\mM_{J}\ar[d,"p_1"]\ar[r,"p_2"]&\mM_J\ar[d,"evb^{\b_2}_0"]\\
\mM_{\hat I}\ar[r,"F_I"]&\mM_{I}\ar[r,"evb_i^{\b_1}"]&L
\end{tikzcd}
\]
Applying Lemma~\ref{pullback of composition of orientors by relatively oriented} to the above diagram, we obtain
\begin{align*}
\expinv{\left(F_I\times\Id,\mO^{F_I}\times1\right)}&\left(J_{I}\bu \left(\expinv{\left(p_2/evb_i^{\b_1}\right)}J_{J}\right)^{\zort{L}^{\b_1}}\bu c_{i0}^{\zort{L}^{\b_1}}\right)=\\  
\overset{\text{Lem.~\ref{pullback of composition of orientors by relatively oriented}}}=&\left(\expinv{\left(F_I,\mO^{F_I}\right)}J_{I}\right)\bu \left(\expinv{\left(F_I\times\Id/F_I\right)}\expinv{\left(p_2/evb_i^{\b_1}\right)}J_{J}\right)^{\zort{L}^{\b_1}}\bu c_{i0}^{\zort{L}^{\b_1}}\\
\overset{\text{Lem.~\ref{properties of J lemma}}}=&J_{\hat I}\bu \left(\expinv{\left(p_2/evb_i^{\b_1}\right)}J_{J}\right)^{\zort{L}^{\b_1}}\bu c_{i0}^{\zort{L}^{\b_1}}.
\end{align*}

Combining the preceding calculation with equation~\eqref{pullback of boundary of J with extra point is remembering a point}, the theorem for $l+1$ implies
\begin{multline*}
\expinv{\left(F_I\times\Id,\mO^{F_I}\times1\right)}\expinv{\vartheta_{l}}\pa J_{l}=\\=(-1)^s\expinv{\left(F_I\times\Id,\mO^{F_I}\times1\right)}\left(J_I\bu \left(\expinv{\left(p_2/evb_i^{\b_1}\right)}J_J\right)^{\zort{L}^{\b_1}}\bu c_{i0}^{\zort{L}^{\b_1}}\right).
\end{multline*}
Since $F_I$ is surjective, we obtain the desired reduction.

We proceed to prove the theorem in the case $k_1+2|I|\geq 2$ and $k_2+2|J|\geq 2$ with the standard abbreviations. Observe that
\begin{align}
\expinv{(\Id\times \vartheta)}\left(\pa\phi^L\right)&\overset{\text{Lem.~\ref{boundary of pulled orientation orientor}}}=\expinv{(\Id\times \vartheta)}\left(\left(\pa\phi\right)^L\right)\\&\overset{\text{Lem.~\ref{pullback of pullback by diffeomorphism of orientor}}}=\left(\expinv{\vartheta}\left(\pa\phi\right)\right)^L
\\&\overset{\text{Lem.~\ref{moduli of disks orientation}}}=(-1)^{s_1}\left(\phi_{1}\bu \left(\phi_{2}\right)^{\mW_1}\right)^L,
\label{moduli of disks orientation times L}
\end{align}
where $s_1=i+ik_2+k$.

The following calculation suffices. Due to lack of horizontal space, we write $\zort{L}^{\b_1}$ for
${evb_0^{\b_1}}^*\zort{L}^{\otimes \b_1}$ in the exponential notation, and $m_2$ for $p_1^*{evb_0^{\b_1}}^*m_2$.
\begin{align*}
\expinv{\left(\vartheta^\b\right)}\pa J\bu m_2\overset{\text{Lem.~\ref{boundary is fibered}}}=&\expinv{\left(\vartheta^\b\right)}\left(\pa\phi^L\bu \expinv{\left(\iota^\b/(\Id\times\iota)\right)}\tilde J\right)\bu m_2\\
\overset{\begin{smallmatrix}\text{diag.~\eqref{theta beta over theta}} \\\text{ Rmk.~\ref{pullback of composition of orientors by local diffeomorphism}}\end{smallmatrix}}=&
\expinv{(\Id\times \vartheta)}\left(\pa\phi^L\right)\bu\left(\expinv{\left(\vartheta^\b/(\Id\times\vartheta)\right)}\expinv{\left(\iota^\b/(\Id\times\iota)\right)}\tilde J\right)\bu m_2\\
\overset{\text{Lem.~\ref{horizontal expinv composition}}}=&
\expinv{(\Id\times \vartheta)}\left(\pa\phi^L\right)\bu\left(\expinv{\left(\left(\vartheta^\b\circ \iota^\b\right)/\left(\Id\times(\vartheta\circ\iota)\right)\right)}\tilde J\right)\bu m_2
\\
\overset{\begin{smallmatrix}\text{eq.~\eqref{moduli of disks orientation times L}}\\\text{Lem. ~\ref{axiomb1b2}}\end{smallmatrix}}=&(-1)^{s}\left(\phi_{1}\bu \left(\phi_{2}\right)^{\mW_1}\right)^L\bu\,\,{}^{\mW_2}\tilde J^1\bu\\&\bu \left(\expinv{\left(p_2/(evb_i^{\b_1}\times\Id)\right)}\tilde J^2\right)^{\zort{L}^{\b_1}}\bu \left(p_1^*c_{i0}\right)^{\zort{L}^{ \b_1}}\\
\overset{\text{Lem.~\ref{phi squared times J squared is phi times J squared lemma}}}=&(-1)^sJ^1\bu \left(\expinv{\left(p_2/evb_i^{\b_1}\right)}J^2\right)^{\zort{L}^{ \b_1}}\bu(p_1^*c_{i0})^{\zort{L}^{\b_1}}.
\end{align*}
\end{proof}
\subsection{Proof for \texorpdfstring{$k=-1$}{k=-1}}
\label{Gluing proof section k=-1}
In this section, we prove Theorem~\ref{expinv of boundary of J by thetabeta theorem k=-1}. We assume $k=-1$ and $\b\in\Pi^{ad}$. In particular, it follows that $k_1=i=k_2=0$.
As stated in the previous section, Lemma~\ref{moduli of disks orientation} holds in this case as well, and equation~\eqref{moduli of disks orientation lemma orientors equation} reads
\begin{equation}
\expinv{\vartheta}(\pa \phi)=-\phi_1\bu \left(\phi_2\right)^{\mW_1}.
\label{moduli of disks orientation lemma orientors equation k=-1}
\end{equation}
Let $U_0:=U_{0,l}^\b$. Set $E:=U_1\times U_2$. That is,
\[
E= \pi^{L\times \mW_1\times\mW_2}_{\mW_1\times\mW_2}\circ(Ev_1\times \Id_{\mW_2})\circ (\Id_{\mM_1}\times U_2).
\]
In particular, 
\begin{equation}
\label{orientation bundle of E k=-1 equation}\zcort{E}\simeq \zcort{Ev_1\times U_2}\otimes \left(Ev_1\times U_2\right)^*\zcort{ \pi^{L\times \mW_1\times \mW_2}_{\mW_1\times \mW_2}}\simeq\zcort{Ev_1\times U_2}\otimes p_1^*{evb_0^{\b_1}}^*\zcort{L}.
\end{equation} 
Consider the following diagram, in which the square is a fiber-product.
\begin{equation*}
\begin{tikzcd}
\mM_1\times_L\mM_2\ar[r,"p_2"]\ar[d,swap,"\Id_{\mM_1}\times U_2"]\ar[ddd,bend right=80,swap,"E"]&\mM_2\ar[d,"Ev_2"]\\
\mM_1\times_\W \mW_2\ar[r,swap,"evb_0^{\b_1}\times\Id"]\ar[d,swap,"Ev_1\times \Id_{\mW_2}"]&L\times_\W\mW_2\\
L\times_\W \mW_1\times_\W\mW_2\ar[d,swap," \pi^{L\times \mW_1\times\mW_2}_{\mW_1\times\mW_2}"]\\
\mW_1\times_\W\mW_2
\end{tikzcd}
\label{E as fibered composition k=-1}    
\end{equation*}
Keeping in mind Lemma~\ref{essential factorization of Ev lemma}, all squares in the following diagram are fiber-products.
\begin{equation*}
\label{theta beta over theta essential k=-1}
\begin{tikzcd}
\mM^{(0)}_1\times_L\mM^{(0)}_2\ar[d,swap,"E"]\ar[r,"\vartheta^\b"]&B^{(1)}(\b)\ar[d,"\iota^*U_0"]\ar[r,"\iota^\b"]&\mM^{(1)}\ar[d,"U_0"]
\\
\mW^{(0)}_1\times_\W \mW^{(0)}_2\ar[r,swap,"\vartheta"]& B^{(1)}\ar[r,swap,"\iota"]&\mW
\end{tikzcd}
\end{equation*}
In particular, all three squares in the following diagram are essential fiber products.
\begin{equation}
\label{theta beta over theta k=-1}
  \begin{tikzcd}
\mM_1\times_L\mM_2\ar[d,swap,"E"]\ar[r,"\vartheta^\b"]&B(\b)\ar[d,"\iota^*U_0"]\ar[r,"\iota^\b"]&\mM\ar[d,"U_0"]
\\
\mW_1\times_\W \mW_2\ar[r,swap,"\vartheta"]& B\ar[r,swap,"\iota"]& \mW
\end{tikzcd}
\end{equation}

\begin{lemma}
\label{axiomb1b2 k=-1}
The following diagram is commutative.
\begin{equation}\label{axiomb1b2 k=-1 diagram}
\begin{tikzcd}
\left(\pi^{\mM_1\times_L\mM_2}\right)^*\mathcal X_L^{\mu(\b)+1}\ar[rrrr,"{\vartheta^\b}^*{\iota^\b}^*\tilde J_0"]\ar[d,swap,"\left(evb_0^{\b_1}\circ p_1\right){i}_{\mu(\b)+1}"]&&&&{\vartheta^\b}^*{\iota^\b}^*\zcort{U_{0,l}^\b}\ar[ddd,swap,"\expinv{\left(\left(\vartheta^\b\circ\iota^\b\right)/\left(\vartheta\circ\iota\right)\right)}"]\\
p_2^*{evb_0^{\b_2}}^*\left(\zort{L}^{\otimes \b_2}\otimes \zort{L}^{\otimes \b_1}\otimes \zort{L}\right)\ar[d,"\left(\expinv{\left(p_2/\left(evb_0^{\b_1}\times\Id\right)\right)}\tilde J^2\right)^{\zort{L}^{\otimes \b_1+1}}"]
\\
\zcort{\Id\times U_2}\otimes p_1^*{evb_0^{\b_1}}^*\left(\zort{L}^{\otimes \b_1}\otimes \zort{L}\right)\ar[d,"\Id\otimes p_1^*\left({}^{\mW_2}\tilde J^1\right)\otimes \Id"]&&&&
\\
\zcort{Ev_1\times U_2}\otimes p_1^*{evb_0^{\b_1}}^*\zort{L}\ar[rr,swap,"\begin{matrix}
\Id\otimes \\p_1^*{evb_0^{\b_1}}^*\left({}^{\mW_1\times_\W\mW_2}(e)\right)
\end{matrix}"]&&\begin{matrix}
\zcort{Ev_1\times U_2}\otimes\\ p_1^*{evb_0^{\b_1}}^*\zcort{L}\end{matrix}\ar[rr,swap,"\text{eq.~\eqref{orientation bundle of E k=-1 equation}}"]&&\zcort{E}
\end{tikzcd}
\end{equation}
That is, we have an equality of \orientor{E}s of $\left(\pi^{\mM_1\times_L\mM_2}\right)^*\mathcal X_L^{\mu(\b)+1}$ to $\underline{\Z/2}$, as follows.
\begin{align*}
\expinv{\left(\left(\iota^\b\circ\vartheta^\b\right)/\left(\iota\circ\vartheta\right)\right)}\tilde J_0={}^{\mW_1\times_\W\mW_2}(e)\bu \left({}^{\mW_2}\tilde J^1\bu \left(\expinv{\left(p_2/evb_0^{\b_1}\times\Id\right)}\tilde J^2\right)^{\zort{L}^{\otimes \b_1}}\right)^{\zort{L}}\bu{i}_{\mu(\b)+1}.
\end{align*}
\end{lemma}
\begin{remark}
The above lemma lacks the sign $S=(-1)^{\d\mu(\b_1)\mu(\b_2)}$ in the commutativity of the diagram in comparison with Lemma~\ref{axiomb1b2}. The reason for this is that since $\b\in \Pi^{ad}$, by assumption, either $\mu(\b)=_21$ or $L$ is vertically orientable. In both cases, one of $\mu(\b_1)$ and $\mu(\b_2)$ is even.
\end{remark}
\begin{lemma}
\label{boundary is fibered k=-1}
It holds that
\begin{equation}
\label{boundary is fibered equation k=-1}
\pa J^{\b}_{0,l}=\pa\phi_{0,l}\bu \expinv{\left(\iota^\b/\iota\right)}\tilde J_0.
\end{equation}
\end{lemma}
\begin{proof}
By Definition~\ref{fundamental orientor definition},
\[
J^{\b}_{0,l}=\phi_{0,I}\bu \tilde J_0.
\]By Lemma~\ref{stability on B} we may apply Lemma~\ref{boundary of composition orientor} to get equation~\eqref{boundary is fibered equation k=-1}.
\end{proof}
\begin{proof}[Proof of Theorem~\ref{expinv of boundary of J by thetabeta theorem k=-1}]
As in the proof of Theorem~\ref{expinv of boundary of J by thetabeta theorem}, the proof can be reduced to the case $|I|,|J|$ large enough. Therefore, we assume $|I|,|J|\geq 2$. Consider the following Cartesian square.
\[
\begin{tikzcd}
L\times_\W\mW_1\times_\W\mW_2\ar[r]\ar[d]&\mW_1\times_\W\mW_2\ar[d]\\L\ar[r]&\W
\end{tikzcd}
\]
Applying Lemma~\ref{commutativity vertical pullback vs horizontal pullback orientors}, keeping in mind that $\deg\phi_1=\deg\phi_2=_20$, we learn that
\begin{align}
\left(\phi_1\bu\phi_2^{\mW_1}\right)\bu\,{}^{\mW_1\times_\W\mW_2}(e)=e\bu \left(\left(\phi_1\bu\phi_2^{\mW_1}\right)^L\right)^{\zort{L}}.
\label{phi1 phi2 commutativity with e}    
\end{align}
Continuing as in the proof of Theorem~\ref{expinv of boundary of J by thetabeta theorem}, we calculate
\begin{align*}
\expinv{\left(\vartheta^\b\right)}\pa J_{0,l}^\b\overset{\text{Lem.~\ref{boundary is fibered k=-1}}}=&\expinv{\left(\vartheta^\b\right)}\left(\pa\phi\bu\expinv{\left(\iota^\b/\iota\right)}\tilde J_0\right)\\
\overset{\begin{smallmatrix}\text{diag.~\eqref{theta beta over theta k=-1}}\\\text{Rmk.~\ref{pullback of composition of orientors by local diffeomorphism}}\end{smallmatrix}}=&\expinv{\vartheta}\left(\pa \phi\right)\bu \left(\expinv{\left(\vartheta^\b/\vartheta\right)}\expinv{\left(\iota^\b/\iota\right)}\tilde J_0\right)\\
\overset{\text{Lem.~\ref{horizontal expinv composition}}}=&\expinv{\vartheta}\left(\pa \phi\right)\bu \left(\expinv{\left(\vartheta^\b\circ\iota^\b/\vartheta\circ\iota\right)}\tilde J_0\right)\\
\overset{\begin{smallmatrix}\text{eq.~\eqref{moduli of disks orientation lemma orientors equation k=-1}}\\\text{Lem.~\ref{axiomb1b2 k=-1}}\end{smallmatrix}}=&-\phi_1\bu\phi_2^{\mW_1}\bu\,{}^{\mW_1\times_\W\mW_2}(e)\bu \\&\bu\left({}^{\mW_2}\tilde J^1\bu \left(\expinv{\left(p_2/evb_0^{\b_1}\times\Id\right)}\tilde J^2\right)^{\zort{L}^{\otimes \b_1}}\right)^{\zort{L}}\bu{i}_{\mu(\b)+1}\\
\overset{\text{eq.~\eqref{phi1 phi2 commutativity with e}}}=&
-e\bu \left(\left(\phi_1\bu\phi_2^{\mW_1}\right)^L\bu{}^{\mW_2}\tilde J^1\bu \left(\expinv{\left(p_2/evb_0^{\b_1}\times\Id\right)}\tilde J^2\right)^{\zort{L}^{\otimes \b_1}}\right)^{\zort{L}}\bu{i}_{\mu(\b)+1}\\
\overset{\text{Lem.~\ref{phi squared times J squared is phi times J squared lemma}}}=&-e\bu \left(J^1\bu\left(\expinv{\left(p_2/evb_0^{\b_1}\right)}J^2\right)^{\zort{L}^{\b_1}}\right)^{\zort{L}}\bu{i}_{\mu(\b)+1}.
\end{align*}
\end{proof}
\subsection{Base change}\label{naturality of J families section}
Let $\xi:\W'\to \W$ be a smooth map of manifolds with corners. Recall diagram~\ref{pullback along families basic diagram}.
For a map of orbifolds $f:M\to N$ over $\W$ and an \orientor{f} $F$ we write
\[
\expinv{\xi}F:=\expinv{\left(\xi^M/\xi^N\right)}F.
\]
Let $\target=\left(\W,X,\w,L,\pi^X,\fp,\underline\Upsilon,J\right)$ be a target over $\W$ as in Section~\ref{families target definition section}.
Denote by $\pull\xi_{\zort{}}$ the \orientor{\Id_{\xi^*L}} given by \[\pull\xi_{\zort{}}:=\pull{\left(\xi^L/\xi\right)}:{\xi^L}^*\zort{L}\to \zort{\xi^*L}.\]

Let $k\geq -1,l\geq 0,I\subset[l]$ and $\b\in \Pi(\target)$. Then
\begin{align*}
   Fi^{(\xi^*\target;\xi^*\b)}_{k+1,I}=\xi^*Fi^{(\target;\b)}_{k+1,I},&\qquad \pull\xi(\mO^{Fi^\mathcal{D}})=\mO^{Fi^{\xi^*\target}},\\
 Fb^{(\xi^*\target;\xi^*\b)}_{k+1,I}=\xi^*Fb^{(\target;\b)}_{k+1,I},&\qquad \pull\xi(\mO^{Fb^\mathcal{D}})=\mO^{Fb^{\xi^*\target}},\\
 f^{(\xi^*\target;\xi^*\b)}=\xi^*f^{(\target;\b)}.
\end{align*}
The proofs of the following lemmas are immediate from the definitions.
\begin{lemma}\label{naturality of cij families lemma}
Let $i,j\in \Z$ and $\b\in \Pi(\target)$.
The following diagram is commutative.
\[
\begin{tikzcd}
\left(evb_j^{(\xi^*\target;\xi^*\b)}\right)^*{\xi^L}^*\zort{L}\ar[rr,"{\xi^{\mM(\target;\b)}}^*c_{ij}^{\target}"]\ar[d,swap,"\left(evb_j^{(\xi^*\target;\xi^*\b)}\right)^*\left(\pull\xi_{\zort{}}\right)"]&&\left(evb_i^{(\xi^*\target;\xi^*\b)}\right)^*{\xi^L}^*\zort{L}\ar[d,"\left(evb_i^{(\xi^*\target;\xi^*\b)}\right)^*\left(\pull\xi_{\zort{}}\right)"]\\
\left(evb_j^{(\xi^*\target;\xi^*\b)}\right)^*\zort{\xi^*L}\ar[rr,swap,"c_{ij}^{\left(\xi^*\target\right)}"]&&\left(evb_i^{(\xi^*\target;\xi^*\b)}\right)^*\zort{\xi^*L}
\end{tikzcd}
\]
That is, the following equation of \orientor{\Id_{\mM_{k+1,l}(\xi^*\target;\xi^*\b)}}s holds.
\[
\left(evb_i^{(\xi^*\target;\xi^*\b)}\right)^*\left(\pull\xi_{\zort{}}\right)\bu{\xi^{\mM(\target;\b)}}^*c_{ij}^{\target}=c_{ij}^{\left(\xi^*\target\right)}\bu \left(evb_j^{(\xi^*\target;\xi^*\b)}\right)^*\left(\pull\xi_{\zort{}}\right)
\]
\end{lemma}
\begin{lemma}\label{naturality of phi families}
Let $k\geq-1,l\geq0$. Then the following equation of orientors holds.\[
\expinv{\xi}\phi_{k+1,l}^{\target}=\phi_{k+1,l}^{\left(\xi^*\target\right)}.
\]
\end{lemma}

\begin{lemma}
\label{naturality of tilde I families}
Let $\target$ be a target. Let $k\geq0,l\geq 0$ and $\b\in \Pi(\target)$.
The following diagram is commutative.
\[
\begin{tikzcd}
{\xi^{\left(\mM_{k+1,l}(\target;\b)\right)}}^*\left(evb_0^\target\right)^*\zort{L}^{\otimes (\mu(\b)+1)}\ar[d,swap,"(evb_0)^*\left(\pull\xi_{\zort{}}\right)^{\otimes (\mu(\b)+1)}"]\ar[r,"\expinv{\xi}\tilde I^{\left(\target,\b\right)}_{k+1,l}"]&\zcort{U_{k+1,l}^{(\xi^*\target;\xi^*\b)}}\\\left(evb_0^{\xi^*\target}\right)^*\zort{\xi^*L}^{\otimes( \mu(\b)+1)}\ar[ur,swap,"\tilde I^{\left(\xi^*\target,\xi^*\b\right)}_{k+1,l}"]
\end{tikzcd}
\]That is, the following equation of orientors holds.
\[
\expinv\xi \tilde I^{\left(\target;\b\right)}_{k+1,l}=\tilde I^{\left(\xi^*\target,\xi^*\b\right)}_{k+1,l}\bu (evb_0)^*\left(\pull\xi_{\zort{}}\right)^{\otimes (\mu(\b)+1)}
\]
Similarly, for $\b\in \Pi^{ad}(\target)$, the following diagram is commutative.
\[
\begin{tikzcd}
{\xi^{\left(\mM_{k+1,l}(\target;\b)\right)}}^*{\pi^{\left(\mM_{k+1,l}(\target;\b)\right)}}^*\mathcal X_L^{\mu(\b)+1}\ar[d,swap,"\left({\pi^{\left(\mM_{k+1,l}(\xi^*\target;\xi^*\b)\right)}}\right)^*\pi^{\xi^*L}_*\left(\pull\xi_{\zort{}}\right)^{\otimes (\mu(\b)+1)}"]\ar[r,"\expinv{\xi}\tilde I^{\left(\target,\b\right)}_{0,l}"]&\zcort{U_{0,l}^{\left(\xi^*\target;\xi^*\target\right)}}
\\\left({\pi^{\left(\mM_{k+1,l}(\xi^*\target;\xi^*\b)\right)}}\right)^*\mathcal X_{\xi^*L}^{\mu(\b)+1}\ar[ur,swap,"\tilde I^{\left(\xi^*\target,\xi^*\b\right)}_{0,l}"]
\end{tikzcd}
\]That is, the following equation of orientors holds.
\[
\expinv\xi \tilde I^{\left(\target;\b\right)}_{0,l}=\tilde I^{\left(\xi^*\target,\xi^*\b\right)}_{k+1,l}\left({\pi^{\left(\mM_{k+1,l}(\xi^*\target;\xi^*\b)\right)}}\right)^*\pi^{\xi^*L}_*\left(\pull\xi_{\zort{}}\right)^{\otimes (\mu(\b)+1)}
\]
\end{lemma}

\begin{lemma}
\label{naturality of J families}
Let $\target$ be a target. Let $k\geq0,l\geq 0$ and $\b\in \Pi(\target)$.
The following diagram is commutative.
\[
\begin{tikzcd}
{\xi^{\left(\mM_{k+1,l}(\target;\b)\right)}}^*\left(evb_0^\target\right)^*\zort{L}^{\otimes \mu(\b)}\ar[d,swap,"(evb_0)^*\left(\pull\xi_{\zort{}}\right)^{\otimes \mu(\b)}"]\ar[r,"\expinv{\xi}J^{\left(\target,\b\right)}_{k+1,l}"]&\zcort{evb_0^{\xi^*\target}}\\\left(evb_0^{\xi^*\target}\right)^*\zort{\xi^*L}^{\otimes \mu(\b)}\ar[ur,swap,"J^{\left(\xi^*\target,\xi^*\b\right)}_{k+1,l}"]
\end{tikzcd}
\]That is, the following equation of orientors holds.
\[
\expinv\xi J^{\left(\target;\b\right)}_{k+1,l}=J^{\left(\xi^*\target,\xi^*\b\right)}_{k+1,l}\bu (evb_0)^*\left(\pull\xi_{\zort{}}\right)^{\otimes \mu(\b)}
\]
Similarly, for $\b\in \Pi^{ad}(\target)$, the following diagram is commutative.
\[
\begin{tikzcd}
{\xi^{\left(\mM_{k+1,l}(\target;\b)\right)}}^*{\pi^{\left(\mM_{k+1,l}(\target;\b)\right)}}^*\mathcal X_L^{\mu(\b)+1}\ar[d,swap,"\left({\pi^{\left(\mM_{k+1,l}(\xi^*\target;\xi^*\b)\right)}}\right)^*\pi^{\xi^*L}_*\left(\pull\xi_{\zort{}}\right)^{\otimes (\mu(\b)+1)}"]\ar[r,"\expinv{\xi}J^{\left(\target,\b\right)}_{0,l}"]&\zcort{\pi^{\mM_{0,l}\left(\xi^*\target;\xi^*\target\right)}}
\\\left({\pi^{\left(\mM_{k+1,l}(\xi^*\target;\xi^*\b)\right)}}\right)^*\mathcal X_{\xi^*L}^{\mu(\b)+1}\ar[ur,swap,"J^{\left(\xi^*\target,\xi^*\b\right)}_{0,l}"]
\end{tikzcd}
\]That is, the following equation of orientors holds.
\[
\expinv\xi J^{\left(\target;\b\right)}_{0,l}=J^{\left(\xi^*\target,\xi^*\b\right)}_{0,l}\bu \left({\pi^{\left(\mM_{k+1,l}(\xi^*\target;\xi^*\b)\right)}}\right)^*\pi^{\xi^*L}_*\left(\pull\xi_{\zort{}}\right)^{\otimes (\mu(\b)+1)}
\]
\end{lemma}
\begin{proof} We prove for $k\geq 0$. The proof for $k=-1$ is similar and requires fewer steps.
By definition of $\expinv \xi$, it is clear that the following diagram is commutative.
\[\begin{tikzcd}
\xi^*\zort{L}\ar[r,"\expinv\xi e_L"]\ar[d,swap,"\pull\xi_{\zort{}}"]&\zcort{\xi^*L}\\
\zort{\xi^*L}\ar[ur,swap,"e_{\xi^*L}"]&
\end{tikzcd}
\]
That is,
\begin{equation}\label{naturality of e families equation}
    \expinv\xi e_L=e_{\xi^*L}\bu \pull{\xi}_{\zort{}}.
\end{equation}
Consider the following diagram.
\begin{equation}\label{trivial pi circ Ev diagram}
\begin{tikzcd}
\mM_{k+1,l}(\xi^*\target;\xi^*\b)\ar[r,"\Id_{\mM}"]\ar[d,swap,"Ev_{k+1,l}^\b"]&\mM_{k+1,l}(\xi^*\target;\xi^*\b)\ar[d,"Ev_{k+1,l}^\b"]\\
L\times \mW_{k+1,l}^{\xi^*\target}\ar[r,swap,"\Id_{L\times\mW}"]\ar[d,swap,"\pi_L"]&L\times \mW_{k+1,l}^{\xi^*\target}\ar[d,"\pi_L"]\\
L\ar[r,swap,"\Id_L"]&L
\end{tikzcd}    
\end{equation}
We calculate
\begin{equation}
\label{expinv of extended pull xi is pullback by evb0 equation}
\expinv{\left(Ev_{k+1,l}^\b/Ev_{k+1,l}^\b\right)}\left({}^{\mW_{k+1,l}^{\xi^*\target}}\left(\pull{\xi}_{\zort{}}\right)\right)\overset{\text{Lem. \ref{vertical pullback expinv composition}}}=\expinv{\left(evb_0^\b/evb_0^\b\right)}\left(\pull{\xi}_{\zort{}}\right)\overset{\text{Exp.~\ref{trivial expinv is pullback}}}=(evb_0^\b)^*\left(\pull{\xi}_{\zort{}}\right).    
\end{equation}
Applying Lemma~\ref{commutativity vertical pullback vs horizontal pullback orientors} to the upper pullback square of diagram~\eqref{trivial pi circ Ev diagram}
with $F={}^{\mW_{k+1,l}^{\xi^*\target}}\left(\pull{\xi}_{\zort{}}\right)$ and $G=\tilde J_{k+1,l}^{(\xi^*\target;\xi^*\b)},$ keeping in mind equation~\eqref{expinv of extended pull xi is pullback by evb0 equation},
one gets
\begin{equation}
    \label{commuting tilde J with pull xi equation}
    {}^{\zort{\xi^*L}} \tilde J^{\left(\xi^*\target,\xi^*\b\right)}_{k+1,l}\bu evb_0^*\left(\pull\xi_{\zort{}}\right)^{\zort{\xi^*L}^{\otimes \mu(\b)}}={}^{\mW_{k+1,l}^{\xi^*\target}}\left( \pull\xi_{\zort{}}\right)\bu{}^{\xi^*\zort{L}} \left(\tilde J^{\left(\xi^*\target,\xi^*\b\right)}_{k+1,l}\right).
\end{equation}
On the one hand,
\begin{align*}
\expinv\xi\tilde I^{(\target;\b)}_{k+1,l}\overset{\text{Lem. \ref{naturality of tilde I families}}}=&\tilde I^{\left(\xi^*\target,\xi^*\b\right)}_{k+1,l}\bu evb_0^*\left(\pull\xi_{\zort{}}\right)^{\otimes (\mu(\b)+1)}
\\\overset{\text{Def.~\ref{Jakethesis}}}=&
{}^{\mW_{k+1,l}^{\xi^*\target}}\left(e_{\xi^*L}\right)\bu {}^{\zort{\xi^*L}} \tilde J^{\left(\xi^*\target,\xi^*\b\right)}_{k+1,l}\bu evb_0^*\left(\pull\xi_{\zort{}}\right)^{\otimes (\mu(\b)+1)}\\
=\quad&{}^{\mW_{k+1,l}^{\xi^*\target}}\left(e_{\xi^*L}\right)\bu {}^{\zort{\xi^*L}} \tilde J^{\left(\xi^*\target,\xi^*\b\right)}_{k+1,l}\bu evb_0^*\left(\pull\xi_{\zort{}}\right)^{\zort{\xi^*L}^{\otimes \mu(\b)}} \bu evb_0^*\left({}^{\xi^*\zort{L}}\left(\left(\pull\xi_{\zort{}}\right)^{\otimes \mu(\b)}\right)\right)
\\\overset{\begin{smallmatrix}\text{eq. ~\eqref{commuting tilde J with pull xi equation}}\\\text{Lem. \ref{base extension is distributive expinv}}\end{smallmatrix}}=&{}^{\mW_{k+1,l}^{\xi^*\target}}\left(e_{\xi^*L}\bu \pull\xi_{\zort{}}\right)\bu{}^{\xi^*\zort{L}} \left(\tilde J^{\left(\xi^*\target,\xi^*\b\right)}_{k+1,l}\bu\left(\pull\xi_{\zort{}}\right)^{\otimes \mu(\b)}\right).
\end{align*}
On the other hand,
\begin{align*}
\expinv\xi\tilde I^{(\target;\b)}_{k+1,l}\overset{\text{Def. \ref{Jakethesis}}}=&\expinv\xi\left({}^{\mW_{k+1,l}^{\target}}e_{L}\bu\tilde J^{(\target;\b)}_{k+1,l}\right)\\\overset{\begin{smallmatrix}\text{Lem. \ref{vertical pullback expinv composition}}\\\text{Lem. \ref{base extension of pullback is pullback by base extension orientors lemma}}\end{smallmatrix}}=&{}^{\mW_{k+1,l}^{\xi^*\target}}\left(\expinv\xi e_{L}\right)\bu\expinv\xi\tilde J^{(\target;\b)}_{k+1,l}\\\overset{\begin{smallmatrix}\text{eq.~\ref{naturality of e families equation}}\\\text{Lem. \ref{base extension is distributive expinv}}\end{smallmatrix}}=&{}^{\mW_{k+1,l}^{\xi^*\target}}\left(e_{
\xi^*L}\bu\pull\xi_{\zort{}} \right)\bu\expinv\xi\tilde J^{(\target;\b)}_{k+1,l}.
\end{align*}
Thus, the uniqueness part of Lemma~\ref{orientor invertion lemma} implies
\begin{equation}\label{naturality of tilde J families}
    \expinv\xi\tilde J^{(\target;\b)}_{k+1,l}=\tilde J^{(\xi^*\target;\xi^*\b)}_{k+1,l}\bu evb_0^*\left(\pull{\xi}_{\zort{}}\right)^{\otimes \mu(\b)}
\end{equation}
We calculate,
\begin{align*}
    \expinv\xi J_{k+1,l}^{(\target;\b)}\overset{\text{Def.\ref{fundamental orientor definition}}}=&\expinv\xi\left(\phi^L_{k+1,l}\bu \tilde J^{(\target;\b)}_{k+1,l}\right)\overset{\text{Lem. \ref{vertical pullback expinv composition}}}=\expinv\xi\phi^L_{k+1,l}\bu \expinv\xi \tilde J^{(\target;\b)}_{k+1,l}\\
    \overset{\begin{smallmatrix}\text{Lem. \ref{naturality of phi families}}\\\text{eq.~\eqref{naturality of tilde J families}}\end{smallmatrix}}=&\phi^{\xi^*L}_{k+1,l}\bu \tilde J^{(\xi^*\target;\xi^*\b)}_{k+1,l}\bu evb_0^*\left(\pull{\xi}_{\zort{}}\right)^{\otimes \mu(\b)}\overset{\text{Def.\ref{fundamental orientor definition}}}=J^{\left(\xi^*\target;\xi^*\b\right)}_{k+1,l}\bu evb_0^*\left(\pull{\xi}_{\zort{}}\right)^{\otimes \mu(\b)}.
\end{align*}
This is the desired result.
\end{proof} 
\section{Orientors of local systems of algebras}\label{orientors of local systems of algebras section}
\subsection{On the Lagrangian}
\label{rort_L section}
\begin{notation}[The algebra of relative orientations]
\label{relative orientation ring notation}
Recall the following notations from Section~\ref{orientors for moduli introduction section}.
For a Lagrangian $L\subset X$ set
\[
{\rort_L}=\bigoplus_{k\in\Z}{\ort{L}}^{\otimes k},
\]
with multiplication operator
\[
m:{\rort_L}^{\otimes 2}\to {\rort_L}
\]
defined naturally by the tensor product.

Let $\efield_L = \pi^L_*{\rort_L}$ be the pushforward of sheaves along $\pi^L:L\to \W$. The multiplication $m$ of $\rort_L$ descends to $\efield_L$. Let $\eort_L = \left(\pi^L\right)^*\efield_L.$ The adjunction of $\pi^L_*$ and $\left(\pi^L\right)^*$ provides a canonical morphism $i:\eort_L\hookrightarrow \rort_L$. Moreover, there exists a morphism \[w:\rort_L\to \eort_L\] which assigns to a point $(t,x)\in{\lort_L}^{\otimes j}$ the locally-constant section over $t\in \W$ that passes through $x$, if it exists, and otherwise $0$. It holds that $w\circ i=\Id_{\eort_L}$. 
The maps $i,w$ commute with the multiplication. The map $i$ is an isomorphism exactly when $L$ is vertically orientable. Otherwise, $\rort_L$ is $2$-dimensional over $\eort_L$ with respect to the algebra structure induced by $i$.

In this section we abbreviate $\rort:={\rort_L}$, $\lort:=\ort{L}, \efield := \efield_L, \eort :=\eort_L$ and for $a\in \Z$ we abbreviate $\lort^a:=\lort^{\otimes a}$. Unless otherwise specified, we use $\mO$ for a local section of $\lort$.
\end{notation}
\begin{remark}
There exists a global section denoted $\chi^2$ of $\efield$ given locally by $\mO^2$ for any local section of $\lort$. $\chi^2$ plays the same role as the variable $e$ in the Novikov ring in \cite[Section 1.2]{Fukaya}.
\end{remark}

\begin{remark}[non-commutativity]
Note that on local sections $x,y$ of ${\rort}$, 
\[
m(x,y)=m(y,x),
\]i.e.
\begin{equation}
\label{noncommutativity}
m\circ\tau_{{\rort},{\rort}}(x\otimes y)=(-1)^{xy}m(x\otimes y).    
\end{equation}

In particular, restricted to odd degrees,
\[
m\circ\tau_{{\rort},{\rort}}=-m.
\]
\end{remark}
\begin{definition}
The \textbf{twistor} $\mS:\rort\to \rort$ is the \eorientor{\Id_L} of $\rort$ given locally by \[\mS\left(\mO_L^{r}\right)=(-1)^{r}\mO_L^{r}.\]
For $\b\in \Pi$ we abbreviate $\mS^\b:=\mS^{\mu(\b)}=\mS^{w_1(\pa\b)}$.
\end{definition}
\begin{lemma}\label{m commutes with S lemma}
As \orientor{\Id_L}s of $\rort\otimes \rort$ to $\rort$, the following equation holds.
\[
\mS\bu m=m\bu \mS^\rort\bu \,{}^\rort\mS.
\]
\end{lemma}
\begin{proof}
Let $\mO$ be a local section of $\lort$, and let $r,t\in \Z$. Then
\[
\left(\mS\bu m\right)\left(\mO^r\otimes \mO^t\right)=\mS(\mO^{r+t})=(-1)^{r+t}\mO^{r+t}=(-1)^{r+t}m\left(\mO^{r}\otimes \mO^t\right)=m\left(\mS\mO^r\otimes \mS\mO^t\right).
\]
\end{proof}
\begin{definition}
Let $a\in \Z$. 
The \textbf{$a$-splitor}
\[
\Psi_a:\rort\to \lort^{\otimes a}\otimes \rort
\]
is the \orientor{\Id_L} given locally by
\[
\Psi_a(\mO_L^r)=\mO_L^a\otimes \mO_L^{r-a}.
\]
If $a$ is an exponent for $L$, then $\Psi_a$ descends to a map 
\[\Psi^\efield_a:\efield\to \mathcal X^{b}\otimes \efield.\]
\end{definition}
\begin{remark}\label{exchange Psi on rort with eort remark}
Let $b\in \Z$. The following equation holds.
\[
i_b^\eort\bu \left(\pi^L\right)^*\Psi^\efield_b={}^{\lort^b}w\bu \Psi_b.
\]
\end{remark}
\begin{notation}
We write $m$ for the restrictions of $m$ to sub-local systems of $\rort\otimes \rort$ in the source and sub-local systems of $\rort$ in the image, when the restriction makes sense. For example, we write
\[
m:\lort^a\otimes \lort^b\to \lort^{a+b}.
\]
\end{notation}
\begin{lemma}
\label{splitor commutativity}Let $a\in \Z$. The following equation of \orientor{\Id_L} of $\rort\otimes \rort$ to $\lort^a\otimes \rort$ holds.
\[
{}^{\lort^a}m\bu \left(\tau_{\rort,\lort^{a}}\right)^{\rort}\bu\,{}^{\rort}\Psi_a=\Psi_a\bu m\bu \left(\mS^a\right)^{\rort}.
\]
\end{lemma}
\begin{proof}
Let $\mO$ be a local section of $\lort$ and let $r_1,r_2\in \Z$. Then
\begin{align*}
{}^{\lort^a}m\bu \left(\tau_{\rort,\lort^{a}}\right)^{\rort}\bu\,{}^{\rort}\Psi_a\left(\mO^{r_1}\otimes \mO^{r_2}\right)&={}^{\lort^a}m\bu \left(\tau_{\rort,\lort^{a}}\right)^{\rort}\bu\,{}^{\rort}\left(\mO^{r_1}\otimes \mO^{a}\otimes \mO^{r_2-a}\right)\\
&=(-1)^{ar_1}\,\,{}^{\lort^a}\!\!m\bu\left(\mO^{a}\otimes\mO^{r_1}\otimes  \mO^{r_2-a}\right)\\
&=(-1)^{ar_1}\,\,\mO^a\otimes \mO^{r_1+r_2-a}\\
&=\Psi_a\bu m\bu \left(\mS^a\right)^{\rort}\left(\mO^{r_1}\otimes \mO^{r_2}\right).
\end{align*}
\end{proof}
\begin{lemma}\label{iterated splitor}
Let $a,b\in \Z$. The following equation of \orientor{\Id_L} of $\rort$ to $\lort^{a+b}\otimes \rort$ holds.
\[
\Psi_{a+b}=m^\rort\bu \,\,{}^{\lort^{a}}\Psi_{b}\bu \Psi_{a}.
\]
\end{lemma}
\begin{corollary}\label{trice iterated splitor corollary}
Let $a,b,c\in\Z$. Then
\[
\Psi_{a+b+c}=m_3^\rort\bu {}^{\lort^a\otimes \lort^b}\Psi_c\bu {}^{\lort^a}\Psi_b\bu \Psi_a.
\]
\end{corollary}
\begin{proof}
Let $\mO$ be a local section of $\lort$ and let $r\in\Z$. Then
\begin{align*}
m^\rort\bu \,\,{}^{\lort^{a}}\Psi_{b}\bu \Psi_{a}\left(\mO^r\right)&=m^\rort\bu \,\,{}^{\lort^{a}}\Psi_{b}\left(\mO^a\otimes \mO^{r-a}\right)\\ 
&\overset{(*)}=m^\rort\left(\mO^a\otimes \mO^b\otimes \mO^{r-a-b}\right)\\
&=\mO^{a+b}\otimes \mO^{r-a-b}\\&=\Psi_{a+b}(\mO^r),
\end{align*}
where $(*)$ follows since $\deg \Psi_b=0$.
\end{proof}
\begin{notation}
For $k\geq 2$, define $m_k:\rort^{\otimes k}\to \rort$ to be the (associative) $k$-fold multiplication 
\[
m_k(x_1,...,x_k)=m(x_1,m(x_2,...,m(x_{k-1},x_{k}))).
\]
Locally it is given by
\[
m_k\left(\mO_L^{r_1}\otimes\cdots \otimes \mO_L^{r_k}\right)={\mO_L}^{r_1+\ldots +r_k}.
\]
Let $m_1:=\Id_{\rort}$ and $m_0:\underline\A\to \rort$ be the embedding of $\underline\A$ as the $0$th degree of $\rort.$
\end{notation}
Recall the \orientor{\pi^L} $e$ of $\lort$ to $\A$ from Definition~\ref{1-n translation L orientor definition} and the map $w$ from Notation~\ref{relative orientation ring notation}.
\begin{definition}
\label{Otm definition}
Denote by $\Otm=\Otm^L$ the \orientor{\pi^L} of $\rort$ to $\efield$ of degree $1-n$ given by the composition
\[
\Otm:=\Otm^L:=e^\efield\bu {}^\lort w\bu \Psi_1.
\]
Moreover, denote by $\Otm_{odd}$ the \orientor{\pi^L} which agrees with $\Otm$ on inputs of odd degree and vanishes on inputs of even degree.
\end{definition}
\subsection{On disk moduli spaces}
Recall the boundary transport orientors $c_{ij}$ from Definition~\ref{boundary transport}. We extend $c_{ij}$ to \orientor{\Id_{\mM}}s of $evb_j^*\rort$ to $evb_i^*\rort$ by $c_{ij}^{\otimes a}$ on $\lort^{\otimes a}$.

Consider the local system \[\underbrace{\rort\boxtimes ...\boxtimes \rort}_{k \,\,\text{times}}\to L^{\times_\W k}.\] 
Let
\begin{align*}
ev:\mdl{3}&\to L^{\times_\W k}
\end{align*}
be given by \begin{align*}
\left(t,\S, u, \vec z,\vec w\right)&\mapsto \left(\left(t,u(z_1)\right),...,\left(t,u(z_k)\right)\right).
\end{align*}
\begin{definition}[The operator $\mI_k$]\label{mI definition}
Let 
\[
\mI_k:ev^*\left(\rort\boxtimes ...\boxtimes \rort\right)\longrightarrow evb_0^*\left(\rort^{\otimes k}\right) 
\]
be the \orientor{\Id_{\mdl{3}}}
given by
\[
ev^* \left( \rort^{\boxtimes k}\right) =\bigotimes_{j=1}^k evb_j^*\rort\xrightarrow{\bigotimes\limits_{j=1}^k c_{0j}}evb_0^*\rort^{\otimes k}.
\]
\end{definition}

Recall the \orientor{evb_0^\b} of $\lort^{\mu(\b)}$ to $\underline{\A}$ of degree $2-k-2|I|$,
\[
J_{k+1,I}^\b:evb_0^*\big(\lort\big)^{\otimes \mu(\b)}\overset{\sim}\longrightarrow\cort{evb_0}, 
\]
and, for $\b\in \Pi^{ad}$, the \orientor{\pi^{\mM_{0,I}(\b)}} of $\left(\pi^{\mM_{0,I}(\b)}\right)^*\mathcal X^{\mu(\b)+1}$ to ${\A}$ of degree $4-n-2|I|$,
\[
J_{0,I}^{\b}:\left(\pi^{\mM_{0,I}(\b)}\right)^*\mathcal X^{\mu(\b)+1}\to \cort{\pi^{\mM_{0,I}(\b)}}
\]
from Definition~\ref{fundamental orientor definition}. 
\begin{definition}[The operator $\mJ^\b_{k+1,I}$]
\label{The operator mJ}
For $k\geq 0$, let
\begin{align*}
    \mJ_{k+1,I}^\b:evb_0^*\rort\to  \cort{evb_0}\otimes evb_0^*\rort
\end{align*}
be the \eorientor{\rort} of  degree $2-k-2|I|$,
given by
\[
\mJ_{k+1,I}^\b:=\left( J^{\b}_{k+1,I}\right)^{\rort}\bu evb_0^*\Psi_\b.
\]
\end{definition}
\begin{definition}
\label{q orientor definition}
For $k\geq 0$, let $\qor_{k,I}^\b$ be
the \orientor{evb_0^\b} of $ev^*(\rort\boxtimes\cdots \boxtimes \rort)$ to $\rort$ of degree $2-k-2|I|$, given by
\begin{align*}
\qor_{k,I}^\b:=\mJ_{k+1,I}^\b\bu evb_0^*\left(\mS^\b\bu m_k\right)\bu \mI_{k}.
\end{align*}
In particular, $\qor_{0,I}^\b$ is an \orientor{evb_0^{\b}} of $\underline \A$ to $\rort$.

For $k=-1,$ let $\qor^\b_{-1,I}$ be the \orientor{\pi^{\mM_{0,l}(\b)}} of $\underline{\A}$ to $\efield$, of degree $4-n-2|I|$, given by
\begin{align*}
\qor_{-1,I}^\b:=\left(J_{0,I}^\b\right)^{\efield}\bu \left(\pi^{\mM}\right)^*\left(\Psi^{\efield}_{\mu(\b)+1}\bu m_0\right)
\end{align*}
if $\b\in \Pi^{ad}$ and $\qor_{-1,I}^\b=0$ otherwise.
\end{definition}
\begin{lemma}\label{q for bottom cases b0 lemma}
For $(k,l,\b)\in\{(2,0,\b_0),(0,1,\b_0)\}$ we have
\[
\qor_{k,l}^{\b_0}=\expinv{evb_0}m_k.
\]
\end{lemma}
\begin{proof}
We have
\[
\mS^\b=\Id,\qquad \mI_{k}=\Id.
\]
The proof is immediate by Remark~\ref{J for bottom cases b0 remark}, followed by the calculation
\[
\qor_{k,l}^{\b_0}=\left(\phi_{evb_0}\right)^{\rort}\bu evb_0^*m_k\overset{\text{Lem. \ref{commutativity vertical pullback vs horizontal pullback orientors}}}=\expinv{evb_0}m_k.
\]
\end{proof}
\begin{notation}[Abuse of notation]\label{evb pullback drop notation}
When there is no ambiguity, we write $\Psi_\b,\mS^\b,m_k$ and $\rort,\lort$ for $evb^*\Psi_\b,evb^*\mS^\b,evb^*m_k$ and $evb^*\rort,evb^*\lort.$ If $G$ is one of $m_k,\Psi_\b,\mS^\b,\mI_k,\mJ_k^\b$ or $\qor_{k.l}^\b$, we also let $G$ denote the extension of scalars by ${\tilde\La[[t_1,...,t_N]]}$ on the right of $G.$
\end{notation}
\begin{remark}\label{forgetful maps on mI implies factorization of q remark}
Theorem~\ref{factorization q through forgetful theorem introduction} follows immediately from Lemma~\ref{properties of J lemma} and the following equation.
\[
\mI_{k+1}=(Fb)^*\mI_{k}\otimes c_{0,k+1}=\left((Fb)^*\mI_{k}\right)^{evb_0^*\rort}\bu {}^{(Fb)^*E_k}c_{0,k+1}.
\]
\end{remark}



\subsection{Proof for \texorpdfstring{$k\geq 0$}{k>-1}}
\label{proof of q-relations}
This section contains the proof of Theorem~\ref{boundary of q theorem introduction}
Let $P$ be an ordered 3-partition of $(1,...,k)$, i.e.
\begin{equation}
\label{3-partition example}
P=(1,...,i-1)\circ(i,...,i+k_2-1)\circ (i+k_2,...,k)=(1:3)\circ(2:3)\circ(3:3),
\end{equation}
and $I\dot\cup J$ be a partition of $[l]$.
Recall the following notation from Section~\ref{orientors for moduli introduction section}
\[E:=E^k:=E^k_L:=\underset{j=1}{\overset{k}\boxtimes} {\left(evb_j^{\b}\right)}^*\rort,\]
and
\begin{align*}
E_1^k:= \underset{j=1}{\overset{i-1}\boxtimes} {\left(evb_j^{\b_1}\right)}^*\rort,\qquad
E_2^k:=\underset{j=i}{\overset{k_2+i-1}\boxtimes} {\left(evb_{j+i-k_2}^{\b_2}\right)}^*\rort,\qquad
E_3^k:=\underset{j=k_2+i}{\overset{k_1}\boxtimes} {\left(evb_{j-k_2+1}^{\b_1}\right)}^*\rort.
\end{align*}
When this creates no ambiguity we write $E_i$ for $E^k_i$, for $i=1,2,3.$
Recall from Section~\ref{Moduli orientation} the gluing map $\vartheta^\b$,
\[
\mdl{1}\times_L\mdl{2}\overset{\vartheta^\b}\rightarrow B(\b),
\]
where $B(\b)$ is a vertical boundary component of $\pi^{\mM_{k+1,l}(\b)}$ from Section~\ref{Moduli spaces}.
For simplicity, we abbreviate
$\vartheta=\vartheta^\b$ and denote the inclusion $B(\b)\rightarrow \pa\mdl{3}$ by $\iota$. Moreover, we abbreviate $B:=B(\b)$. This notation replaces the notation from Section~\ref{Moduli orientation} since the latter is irrelevant in the current section.
In order to prove Theorem~\ref{boundary of q theorem introduction}. We need to consider how $\mI$ and $\mJ$ behave under partitions. Thus, we fix the following.

Let $P\in S_3[k],\quad I\dot\cup J=[l]$ be partitions, and $\b_1,\b_2\in\Pi$ such that $\b_1+\b_2=\b$. Let $\a=\a_1\otimes\cdots\otimes \a_k\in C^{\otimes k}$ and $\g=\g_1\otimes \cdots\otimes \g_l\in D^{\otimes l}$. Let $k_1=|(1:3)|+|(3:3)|+1$, $k_2=(2:3)$ and $i=|(1:3)|+1$.

Throughout this section, we abbreviate
\begin{align*}
    \begin{matrix}
    \mI=\mI_{k},&
    \mJ=\mJ_{k+1,l}^\b,& \qor=\qor_{k.l}^\b,
    \\
    \mI_1=\mI_{k_1},&
    \mJ^1=\mJ_{k_1+1,I}^{\b_1},& \qor_1=\qor_{k_1,I}^{\b_1},
    \\
    \mI_2=\mI_{k_2},&
    \mJ^2=\mJ_{k_2+1,J}^{\b_2},&
    \qor_2=\qor_{k_2,J}^{\b_2}.
    \end{matrix}
\end{align*}

\begin{remark}[Rotation around $\b_2$]
\label{rotate around b2 rort}
Recall Remark~\ref{boundary marked points on B} and Lemma~\ref{rotate around b2}:
\[
evb^\b_j|_B\circ \vartheta=
\begin{cases}
evb^{\b_1}_j\circ p_1,& 0\leq j\leq i-1,\\
evb^{\b_2}_{j-i+1}\circ p_2,& i\leq j\leq i+k_2-1,\\
evb^{\b_1}_{j-k_2+1}\circ p_1,& i+k_2\leq j\leq k.
\end{cases}
\]
\[
evb_i^{\b_1}\circ p_1=evb_0^{\b_2}\circ p_2
\]
\[
{\vartheta}^* c_{0j}=
\begin{cases}
p_1^*c_{0j}^{1},& 0\leq j\leq i-1,\\
p_1^*c^{1}_{0i}\circ p_2^*c^{2}_{0,j-i+1},& i\leq j\leq i+k_2-1,\\
\mS^{\b_2}\circ\,\, p_1^*c^{1}_{0,j-k_2+1},& i+k_2\leq j\leq k.
\end{cases}
\]
Note that we replaced $(-1)^{w_1(\pa\b_2)}$ by $\mS^{\b_2}$ to account for the degrees in $\rort$.
\end{remark}

\begin{lemma}[Boundary of $\mI_k$]
\label{Boundary of mI lemma} Let $P\in S_k[3],$ $I\dot\cup J=[l],$ $\b_1,\b_2\in \Pi$ be as in the previous notation.
Then, as \orientor{\Id_{\mM_1\times_L\mM_2}}s of $\vartheta^*\iota^*{evb^{\b}}^*\left(\underset{j=1}{\overset{k}\boxtimes}\rort\right)$ to $p_1^*evb_0^{\b_1}\left(\rort^{\otimes 3}\right)$, the following equation holds.
\[
\Big(m_{i-1}\otimes m_{k_2}\otimes m_{k_1-i}\Big)\bu \vartheta^*\iota^*\mI_{k}\!=\!\left(m_{i-1}\!\otimes \!\Id_{\rort}\!\otimes\! \left(\mS^{\b_2}\bu m_{k_1-i}\right)\right)\bu p_1^*\mI_{k_1}\bu \,\,{}^{p_1^*E_1\!\!}\left(m_{k_2}\bu p_2^*\mI_{k_2}\right)^{p_1^*E_3}
\]
That is, the following diagram of local systems is commutative without a sign.
\[
\begin{tikzcd}
p_1^* E_1\otimes p_2^*E_2\otimes p_1^* E_3
\ar[rrr,"\Id\otimes  \left(m_{k_2}\circ p_2^*\mI_{k_2}\right)\otimes \Id"]\ar[d,equal]&&&p_1^* E_1\otimes p_2^*{evb_0^{\b_2}}^*\rort\otimes p_1^* E_3\ar[d,equal]\\
\vartheta^*\iota^*{evb^\b}^*\left(\underset{j=1}{\overset{k}\boxtimes} \rort\right)\ar[d,"\vartheta^*\iota^*\mI_{k}"]&&&p_1^*{evb^{\b_1}}^*\Big(\underset{j=1}{\overset{k_1}\boxtimes}\rort\Big)\ar[d," p_1^*\mI_{k_1}"]\\
\vartheta^*\iota^*\bigotimes\limits_{j=1}^k{evb_0^{\b}}^*\rort\ar[rrd,"m_{i-1}\otimes m_{k_2}\otimes m_{k_1-i}",swap]&&&p_1^*\bigotimes\limits_{j=1}^{k_1}{evb_0^{\b_1}}^*\rort\ar[ld,"m_{i-1}\otimes \Id_\rort\otimes \left(\mS^{\b_2}\bu m_{k_1-i}\right)"]\\
&&p_1^*{evb_0^{\b_1}}^*\left(\rort^{\otimes 3}\right)
\end{tikzcd}
\]
\end{lemma}
\begin{proof}
The lemma follows from Remark~\ref{rotate around b2}.
\end{proof}
\begin{lemma}
\label{boundary transport commutes with b2 splitor lemma}Let $a\in \Z$.
As \orientor{\Id_{\mM_1}}s of ${evb_0^{\b_1}}^*\rort$ to ${evb_0^{\b_1}}^*\left(\lort^{\b_2}\otimes \rort\right)$, the following equation holds.
\begin{equation}
\label{boundary transport commutes with b2 splitor equation}
\left(c_{0i}^{\rort}\bu \,\,{}^{\lort^{a}}c_{0i}\right)\bu {evb_i^{\b_1}}^*\Psi_{a}\bu c_{i0}= {evb_0^{\b_1}}^*\Psi_{a}
\end{equation}
\end{lemma}
\begin{proof}
This follows from an algebraic calculation on local sections $\mO^r$ of $\rort$.
\end{proof}
\begin{corollary}
\label{boundary transport commutes with b2 splitor corollary}
The following equation of \orientor{\Id_{\mM_1}}s from $\rort\otimes \rort$ to $\lort^{\b_2}\otimes \rort$ holds.
\begin{align*}
{}^{\lort^{\b_2}} c_{0i}\bu  {evb_i^{\b_1}}^*\Psi_{\b_2}\bu c_{i0}\bu m
={}^{\lort^{\b_2}} m\bu \left(\tau_{\rort,\lort^{\b_2}}\right)^{\rort}\bu \,\,{}^{\rort}\!\!\left({}^{\lort^{\b_2}} c_{0i}\bu {evb_i^{\b_1}}^*\Psi_{\b_2}\bu c_{i0}\right)\bu \left(\mS^{\b_2}\right)^{\rort}
\end{align*}
\end{corollary}
\begin{proof}
We calculate
\begin{align}
\notag c_{0i}^\rort\bu{}^{\lort^{\b_2}} c_{0i}\bu  {evb_i^{\b_1}}^*\Psi_{\b_2}\bu c_{i0}\bu m
\overset{\text{Lem.~\ref{boundary transport commutes with b2 splitor lemma}}}=&\,\, {evb_0^{\b_1}}^*\Psi_{\b_2}\bu m\\
\notag\overset{\text{Lem.~\ref{splitor commutativity}}}=&\,\,{}^{\lort^{\b_2}}m\bu \left(\tau_{\rort,\lort^{\b_2}}\right)^{\rort}\bu \,{}^\rort\!{evb_0^{\b_1}}^*\Psi_{\b_2}\bu \left(\mS^a\right)^\rort\\
\label{boundary transport commutes with b2 splitor corollary proof equation 1}\overset{\text{Lem.~\ref{boundary transport commutes with b2 splitor lemma}}}=&\,\,{}^{\lort^{\b_2}}m\bu \left(\tau_{\rort,\lort^{\b_2}}\right)^{\rort}\bu\\
\notag&\bu\,{}^\rort\!\left(\left(c_{0i}^{\rort}\bu \,\,{}^{\lort^{\b_2}}c_{0i}\right)\bu {evb_i^{\b_1}}^*\Psi_{\b_2}\bu c_{i0}\right)\bu \left(\mS^{\b_2}\right)^\rort.
\end{align}
However, by Lemma~\ref{orientors symmetry koszul} applied to the map $\Id_{\mM_1}$, we see that
\[
\tau_{\rort,\lort^{\b_2}}\bu \,\,{}^\rort\!c_{0i}=c_{0i}^\rort\bu \tau_{\rort,\lort^{\b_2}},
\]
and therefore,
\begin{equation}
\left(\tau_{\rort,\lort^{\b_2}}\right)^\rort\bu \,\,{}^\rort\!c_{0i}^\rort=c_{0i}^{\rort\otimes \rort}\bu \left(\tau_{\rort,\lort^{\b_2}}\right)^\rort.
\label{boundary transport commutes with b2 splitor corollary proof equation 2}    
\end{equation}
Moreover, by Lemma~\ref{commutativity vertical pullback vs horizontal pullback orientors} applied to the diagram all arrows of which are the identity on $\mM_1$, we see that
\begin{equation}
    \label{boundary transport commutes with b2 splitor corollary proof equation 3}
    {}^{\lort^{\b_2}}m\bu c_{0i}^{\rort\otimes \rort}=c_{0i}^\rort\bu {}^{\lort^{\b_2}}m.
\end{equation}
Combining equations~\eqref{boundary transport commutes with b2 splitor corollary proof equation 1},~\eqref{boundary transport commutes with b2 splitor corollary proof equation 2} and~\eqref{boundary transport commutes with b2 splitor corollary proof equation 3}, we obtain
\[
c_{0i}^\rort\bu{}^{\lort^{\b_2}} c_{0i}\bu  {evb_i^{\b_1}}^*\Psi_{\b_2}\bu c_{i0}\bu m=
c_{0i}^\rort\bu {}^{\lort^{\b_2}}m\bu \tau_{\rort,\lort^{\b_2}}\bu  \,{}^\rort\!\left({}^{\lort^{\b_2}}c_{0i}\bu {evb_i^{\b_1}}^*\Psi_{\b_2}\bu c_{i0}\right)\bu \left(\mS^{\b_2}\right)^\rort.
\]
Since $c_{0i}$ is invertible, we conclude the corollary.
\end{proof} 
\begin{lemma}
\label{extend m2 to m3 lemma}
Let $Q,K$ be local systems over $\mM$ and let $F,G$ be \orientor{\Id_{\mM}}s of $Q\otimes \rort$ to $K\otimes \rort$, such that
\[
{}^Km_2\bu F^{\rort}=G\bu {}^Qm_2.
\]
Then
\[
{}^Km_3\bu F^{\rort\otimes \rort}=G\bu {}^Qm_3.
\]
\end{lemma}
\begin{proof}
By definition $m_3:=m_2\bu \,{}^{\rort}m_2.$ Therefore, by Lemma~\ref{commutativity vertical pullback vs horizontal pullback orientors} applied to the square diagram all arrows of which are $\Id_{\mM}$, it holds that
\[
\,{}^{K\otimes \rort}m_2\bu F^{\rort\otimes \rort}= F^{\rort}\bu {}^{Q\otimes \rort}m_2.
\]
Therefore,
\[
{}^Km_3\bu F^{\rort\otimes \rort}={}^Km_2\bu \,{}^{K\otimes \rort}m_2\bu F^{\rort\otimes \rort}={}^Km_2\bu F^{\rort}\bu {}^{Q\otimes \rort}m_2=G\bu \,{}^Q m_2\bu \,{}^{Q\otimes \rort}m_2=G\bu \,{}^Qm_3.
\]
\end{proof}
\begin{corollary}
\label{boundary transport commutes with b2 splitor m3 corollary}
The following equation of \orientor{\Id_{\mM_1}}s from $\rort\otimes \rort\otimes \rort$ to $\lort^{\b_2}\otimes \rort\otimes \rort$ holds.
\begin{multline*}
{}^{\lort^{\b_2}} c_{0i}\bu  {evb_i^{\b_1}}^*\Psi_{\b_2}\bu c_{i0}\bu m_3
\\={}^{\lort^{\b_2}} m_3\bu \left(\left(\tau_{\rort,\lort^{\b_2}}\right)^{\rort}\bu \,\,{}^{\rort}\!\!\left({}^{\lort^{\b_2}} c_{0i}\bu {evb_i^{\b_1}}^*\Psi_{\b_2}\bu c_{i0}\right)\bu \left(\mS^{\b_2}\right)^{\rort}\right)^\rort
\end{multline*}
\end{corollary}
\begin{proof}
This is a consequence of Lemma~\ref{extend m2 to m3 lemma} applied to Corollary~\ref{boundary transport commutes with b2 splitor corollary}.
\end{proof}
\begin{lemma}[Boundary of $\mJ_{k+1,l}^\b$]\label{boundary of mJ lemma}
Let $\rort_1,\rort_2,\rort_3$ be copies of $\rort.$ As \orientor{evb_0^{\b}\circ\iota\circ \vartheta}s \[
\vartheta^*\iota^*\left(evb_0^\b\right)^*\left({\rort_1}\otimes {\rort_2}\otimes {\rort_3}\right)\to \cort{evb_0^{\b}\circ\iota\circ \vartheta}\otimes \vartheta^*\iota^*\left(evb_0^\b\right)^*\rort,
\]
the following equations hold.
\begin{align*}
\label{boundary of mJ equation}
\expinv{\vartheta}(\pa \mJ)\bu m_3=&(-1)^s\mJ^1\bu m_3\bu \,{}^{{\rort_1}}\!\!\left(c_{0i}\bu \expinv{\left(p_2/evb_i^{\b_1}\right)}\mJ^2\bu c_{i0}\right){}^{{\rort_3}}\bu \left(\mS^{\b_2}\right)^{{\rort_2}\otimes {\rort_3}},\\
\expinv{\vartheta}(\pa \mJ\bu \mS^{\b})\bu m_3=&(-1)^s\mJ^1\bu\mS^{\b_1}\bu m_3\bu \\
&\quad\bu\,{}^{{\rort_1}}\!\!\left(c_{0i}\bu \expinv{\left(p_2/evb_i^{\b_1}\right)}\mJ^2\bu\mS^{\b_2}\bu c_{i0}\right){}^{{\rort_3}}\bu {}^{{\rort_1}\otimes {\rort_2}}\!\!\left(\mS^{\b_2}\right),
\end{align*}
with
\[
s=i+ik_2+k+\d\mu(\b_1)\mu(\b_2).
\]
\end{lemma}
\begin{proof}[Proof of Lemma \ref{boundary of mJ lemma}]
Consider the first equation.
It follows from Lemma~\ref{boundary of composition orientor} and Example~\ref{trivial expinv is pullback} applied to the diagram
\[
\begin{tikzcd}
B\ar[r,"\iota"]\ar[d,equal]&\mM\ar[d,equal]\\
B\ar[r,"\iota"]&\mM\ar[r,"evb_0^\b"]&L
\end{tikzcd}
\]
that
\[
\pa\mJ=(\pa J)^{\rort}\bu \iota^*{evb_0^\b}^*\Psi_\b.
\]
Similarly, it follows from Remark~\ref{pullback of composition of orientors by local diffeomorphism} and Example~\ref{trivial expinv is pullback} applied to the diagram
\[
\begin{tikzcd}
\mM_1\times_L\mM_2\ar[r,"\vartheta"]\ar[d,equal]&B\ar[d,equal]\\
\mM_1\times_L\mM_2\ar[r,"\vartheta"]&B\ar[r,"evb_0^\b\circ \iota"]&L
\end{tikzcd}
\]
that
\begin{align}
\label{pullback of boundary of mJ initial equation}
\expinv{\vartheta}\left(\pa \mJ\right)=\left(\expinv{\vartheta}\left(\pa J\right)\right)^{\rort}\bu \vartheta^*\iota^*{evb_0^\b}^*\Psi_\b.
\end{align}

We investigate $\vartheta^*\iota^*{evb_0^\b}^*\Psi_\b.$
By Lemma~\ref{iterated splitor},
\[
\Psi_\b=m_2^\rort\bu \,\,{}^{\lort^{\b_2}}\Psi_{\b_1}\bu \Psi_{\b_2}.
\]
Keeping in mind Example~\ref{trivial expinv is pullback}, Lemma~\ref{commutativity vertical pullback vs horizontal pullback orientors} applied to the pullback diagram all arrows of which are the identity on $\mM_1$ gives
\begin{equation}
(c_{0i})^{\lort^{\b_1}\otimes\rort}\bu {evb_0^{\b_1}}^*\left({}^{\lort^{\b_2}}\Psi_{\b_1}\right)={evb_0^{\b_1}}^*\left({}^{\lort^{\b_2}}\Psi_{\b_1}\right)\bu (c_{0i})^{\rort}.
\label{boundary transport commutes with b1 splitor equation}    
\end{equation}

Therefore,
\begin{align*}
\vartheta^*\iota^* {evb_0^\b}^*\Psi_\b\overset{\text{Lem.~\ref{iterated splitor}}}=&\vartheta^*\iota^* {evb_0^\b}^*\left(m_2^\rort\bu \,\,{}^{\lort^{\b_2}}\Psi_{\b_1}\bu \Psi_{\b_2}\right)\\
\overset{\text{Rmk.~\ref{rotate around b2 rort}}}=&p_1^*{evb_0^{\b_1}}^*(m_2^{\rort})\bu p_1^*{evb_0^{\b_1}}^*\left({}^{\lort^{\b_2}}\Psi_{\b_1}\right)\bu p_1^*{evb_0^{\b_1}}^*\Psi_{\b_2}
\\
\overset{\text{Lem.~\ref{boundary transport commutes with b2 splitor lemma}}}=&
p_1^*{evb_0^{\b_1}}^*m_2^\rort\bu p_1 {evb_0^{\b_1}}^*\left({}^{\lort^{\b_2}}\Psi_{\b_1}\right)\bu p_1^*\left(c_{0i}^{\rort}\bu \,\,{}^{\lort^{\b_2}}c_{0i}\right)\bu \\&\qquad\bu p_2^*{evb_0^{\b_2}}^*\Psi_{\b_2}\bu p_1^*c_{i0}\\
\overset{\text{eq.~\eqref{boundary transport commutes with b1 splitor equation}}}=&
p_1^*{evb_0^{\b_1}}^*m_2^{\rort} \bu (p_1^*c_{0i})^{\lort^{\b_1}\otimes\rort}\bu {}^{\lort^{\b_2}}p_1^*\left({evb_0^{\b_1}}^*\left(\Psi_{\b_1}\right)\bu c_{0i}\right)\bu\\&\qquad\bu  p_2^* {evb_0^{\b_2}}^*\Psi_{\b_2}\bu p_1^*c_{i0}.
\end{align*}

Together with Theorem~\ref{expinv of boundary of J by thetabeta theorem} and equation~\eqref{pullback of boundary of mJ initial equation}, the above equation implies that
\begin{align*}
\expinv{\vartheta}\left(\pa \mJ\right)=&(-1)^{s}
\left(
J^1\bu \left(\expinv{\left(p_2/evb_i^{\b_1}\right)}J^2\right)^{\!\lort^{ \b_1}}\right)^{\!\!\rort}\bu\\&\qquad\bu  p_1^*\left({}^{\lort^{\b_2}}\left({evb_0^{\b_1}}^*\Psi_{\b_1}\bu c_{0i}\right)\right)\bu p_2^* {evb_0^{\b_2}}^*\Psi_{\b_2}\bu p_1^*c_{i0},
\end{align*}
with $s=i+ik_2+k+\d\mu(\b_1)\mu(\b_2)$. Keeping in mind Example~\ref{trivial expinv is pullback}, Lemma~\ref{commutativity vertical pullback vs horizontal pullback orientors} applied to the pullback diagram
\[
\begin{tikzcd}
\mM_1\times_L\mM_2\ar[d,"p_1"]\ar[r,equal]&\mM_1\times_L\mM_2\ar[d,"p_1"]\\
\mM_1\ar[r,equal]&\mM_1
\end{tikzcd}
\]
gives
\begin{align}
\left(\expinv{\left(p_2/evb_i^{\b_1}\right)}J^2\right)^{\lort^{ \b_1}\otimes \rort}\bu{}^{\lort^{\b_2}}\!\!\left(p_1^*{evb_0^{\b_1}}^*\Psi_{\b_1}\right)={evb_0^{\b_1}}^*\Psi_{\b_1}\bu\left(\expinv{\left(p_2/evb_i^{\b_1}\right)}J^2\right)^{ \rort},
\label{J2 commutes with Psi1 equation}    
\end{align}
and thus
\begin{align*}
\expinv{\vartheta}\left(\pa \mJ\right)=(-1)^{s}
\mJ^1\bu\left(\expinv{\left(p_2/evb_i^{\b_1}\right)}J^2\right)^{\rort}\bu\,\, p_1^*\left({}^{\lort^{\b_2}} c_{0i}\right)\bu p_2^* {evb_0^{\b_2}}^*\Psi_{\b_2}\bu p_1^*\left(c_{i0}\right).
\end{align*}
By Remark~\ref{rotate around b2 rort}, $evb_i^{\b_1}\circ p_1=evb_0^{\b_2}\circ p_2$. Therefore, Corollary~\ref{boundary transport commutes with b2 splitor m3 corollary} implies 
\begin{align}
\notag\expinv{\vartheta}(\pa\mJ)\bu m_3=(-1)^s&\mJ^1\bu\left(\left(\expinv{p_2/evb_i^{\b_1}}\right)J^2\right)^{\rort}\bu
\\&\bu{}^{\lort^a}\!m_3\bu p_1^*\left( \left(\tau_{\rort,\lort^{\b_2}}\right)^{\rort}\bu \,{}^{\rort}\!\!\left({}^{\lort^{\b_2}}c_{0i}\bu {evb_i^{\b_1}}^*\Psi_{\b_2}\bu c_{i0}\right)\bu \left(\mS^{\b_2}\right)^{\rort}\right)^{\rort}.
\label{boundary of mJ bu m3 prior commutativity with J2 equation}
\end{align}

We proceed by commuting $\left(\expinv{\left(p_2/evb_i^{\b_1}\right)}J^2\right)^{\rort}$ with orientors on the right hand side of equation~\eqref{boundary of mJ bu m3 prior commutativity with J2 equation} until it arrives at $\Psi_{\b_2}$, with which it combines to create $\mJ^2$.
By Lemma~\ref{commutativity vertical pullback vs horizontal pullback orientors} applied to the following diagram,
\begin{equation}
\begin{tikzcd}
\mM_1\times_L\mM_2\ar[r,equal]\ar[d,"p_1"]&\mM_1\times_L\mM_2\ar[d,"p_1"]\\
\mM_1\ar[r,equal]&\mM_1
\end{tikzcd}
\label{vertical p1 square diagram}
\end{equation}
we obtain an equality of \orientor{p_1}s of $p_2^*{evb_0^{\b_2}}^*\left(\lort^{\b_2}\otimes \rort\otimes \rort\right)$ to ${evb_i^{\b_1}}^*\rort$,
\begin{equation*}
\left(\expinv{\left(p_2/evb_i^{\b_1}\right)}J^2\right)^\rort\bu\,\,{}^{\lort^{\b_2}} m_2=m_2\bu \left(\expinv{\left(p_2/evb_i^{\b_1}\right)}J^2\right)^{\rort\otimes\rort}.
\end{equation*}
By Lemma~\ref{extend m2 to m3 lemma}, we obtain an equality of \orientor{p_1}s of $p_2^*{evb_0^{\b_2}}^*\left(\lort^{\b_2}\otimes \rort\otimes \rort\otimes \rort\right)$ to ${evb_i^{\b_1}}^*\rort$,
\begin{equation}
\label{J2 commutes with m3 equation}
\left(\expinv{\left(p_2/evb_i^{\b_1}\right)}J^2\right)^\rort\bu\,\,{}^{\lort^{\b_2}} m_3=m_3\bu \left(\expinv{\left(p_2/evb_i^{\b_1}\right)}J^2\right)^{\rort\otimes\rort\otimes\rort}.
\end{equation}
Applying Lemma~\ref{orientors symmetry koszul} with $B=\lort^{\b_2}$ and $C=\underline{\Z/2}$, we obtain an equality of \orientor{p_1}s of $p_1^*evb_i^*\left(\rort\otimes \lort^{\b_2}\right)$ to $evb_i^*\left(\rort\right)$,
\begin{equation}
\label{J2 commutes with tau equation} \left(\expinv{\left(p_2/evb_i^{\b_1}\right)}J^2\right)^\rort\bu {p_1^*\tau_{\rort,\lort^{\b_2}}}={}^\rort\left(\expinv{\left(p_2/evb_i^{\b_1}\right)}J^2\right).   
\end{equation}
By Lemma~\ref{commutativity vertical pullback vs horizontal pullback orientors} applied to diagram~\eqref{vertical p1 square diagram} we obtain the following equation of \orientor{p_1}s of $p_2^*{evb_0^{\b_2}}^*\left(\lort^{\b_2}\otimes \rort\right)$ to ${evb_0^{\b_1}}^*\rort,$
\begin{equation}
\label{J2 commutes with c0i equation}
\left(\expinv{\left(p_2/evb_i^{\b_1}\right)}J^2\right)^{\rort}\bu p_1^*\left({}^{\lort^{\b_2}}c_{0i}\right)=p_1^*c_{0i}\bu \left(\expinv{\left(p_2/evb_i^{\b_1}\right)}J^2\right)^{\rort}.    
\end{equation}
Combining equations~\eqref{boundary of mJ bu m3 prior commutativity with J2 equation} and~\eqref{J2 commutes with m3 equation}-\eqref{J2 commutes with c0i equation}, we conclude that
\[
\expinv{\vartheta}(\pa \mJ)\bu m_3=(-1)^s\mJ^1\bu m_3\bu \,{}^{\rort}\left(p_1^*c_{0i}\bu \expinv{\left(p_2/evb_i^{\b_1}\right)}\mJ^2\bu p_1^*c_{i0}\right)^{\rort}\bu \left(\mS^{\b_2}\right)^{\rort\otimes \rort}.
\]This proves the first equation.

The second equation follows by Lemma~\ref{m commutes with S lemma} and the commutativity of $\mS$ with $c_{0i}$.
\end{proof}

\begin{proof}[Proof of Theorem \ref{boundary of q theorem introduction}]
The associativity of $m$ implies
\begin{align}
\label{mk associativity}m_k&=m_3\bu\left(m_{i-1}\otimes m_{k_2}\otimes m_{k_1-i}\right),\\
\label{mk1 associativity} m_{k_1}&=m_3\bu\left(m_{i-1}\otimes \Id_{\rort}\otimes m_{k_1-i}\right).    
\end{align}
It follows from Lemma~\ref{boundary of composition orientor} and Example~\ref{trivial expinv is pullback} applied to the diagram
\[
\begin{tikzcd}
B\ar[r,"\iota"]\ar[d,equal]&\mM\ar[d,equal]\\
B\ar[r,"\iota"]&\mM\ar[r,"evb_0^\b"]&L
\end{tikzcd}
\]
that
\[
\pa\qor=(\pa \mJ)\bu \iota^*\left({evb_0^\b}^*\left(\mS^\b\bu m_k\right)\bu \mI_k\right).
\]
Similarly, it follows from Remark~\ref{pullback of composition of orientors by local diffeomorphism} and Example~\ref{trivial expinv is pullback} applied to the diagram
\[
\begin{tikzcd}
\mM_1\times_L\mM_2\ar[r,"\vartheta"]\ar[d,equal]&B\ar[d,equal]\\
\mM_1\times_L\mM_2\ar[r,"\vartheta"]&B\ar[r,"evb_0^\b\circ \iota"]&L
\end{tikzcd}
\]
that

\[
\expinv{\vartheta}(\pa \qor)=\expinv{\vartheta}(\pa \mJ)\bu \vartheta^*\iota^*{evb_0^\b}^*\left(\mS^\b\bu m_k\right)\bu\vartheta^*\iota^*\mI_k
\]
Keeping in mind Notation~\ref{evb pullback drop notation} and using equation~\eqref{mk associativity}, we see that
\[
\expinv{\vartheta}(\pa \qor)=\expinv{\vartheta}(\pa \mJ)\bu \mS^\b\bu m_3\bu \left(m_{i-1}\otimes m_{k_2}\otimes m_{k_1-i}\right)\bu \vartheta^*\iota^*\mI_k.
\]
Apply Lemmas~\ref{Boundary of mI lemma},~\ref{boundary of mJ lemma}, noting that ${}^{\rort\otimes \rort}\left(\mS^{\b_2}\right)$ cancels, to obtain
\begin{align*}
\expinv{\vartheta}(\pa \qor)=&(-1)^s\mJ^1\bu\mS^{\b_1}\bu m_3\bu \,{}^{{\rort_1}}\!\!\left(p_1^*c_{0i}\bu \expinv{\left(p_2/evb_i^{\b_1}\right)}\mJ^2\bu\mS^{\b_2}\bu p_1^*c_{i0}\right){}^{{\rort_3}}\bu\\
&\bu \left(m_{i-1}\otimes \Id_{E_2}\otimes  m_{k_1-i}\right)\bu p_1^*\mI_{k_1}\bu \,\,{}^{p_1^*E_1\!\!}\left(m_{k_2}\bu p_2^*\mI_{k_2}\right)^{p_1^*E_3}\\
\overset{\text{Lem.~\ref{commutativity vertical pullback vs horizontal pullback orientors}}}
=&(-1)^s\mJ^1\bu\mS^{\b_1}\bu m_3\bu \left(m_{i-1}\otimes \Id_{\rort}\otimes  m_{k_1-i}\right)\bu\\
&\bu \,{}^{\rort^{\otimes (i-1)}}\!\!\left(p_1^*c_{0i}\bu \expinv{\left(p_2/evb_i^{\b_1}\right)}\mJ^2\bu\mS^{\b_2}\bu p_1^*c_{i0}\right){}^{\rort^{\otimes (k_1-i)}}\bu \\&\bu p_1^*\mI_{k_1}\bu \,\,{}^{p_1^*E_1\!\!}\left(m_{k_2}\bu p_2^*\mI_{k_2}\right)^{p_1^*E_3}\\
\overset{\text{eq.~\eqref{mk1 associativity}}}=&(-1)^s\mJ^1\bu\mS^{\b_1}\bu m_{k_1}\bu\\
&\bu \,{}^{\rort^{\otimes (i-1)}}\!\!\left(p_1^*c_{0i}\bu \expinv{\left(p_2/evb_i^{\b_1}\right)}\mJ^2\bu\mS^{\b_2}\bu p_1^*c_{i0}\right){}^{\rort^{\otimes (k_1-i)}}\bu \\&\bu p_1^*\mI_{k_1}\bu \,\,{}^{p_1^*E_1\!\!}\left(m_{k_2}\bu p_2^*\mI_{k_2}\right)^{p_1^*E_3}.
\end{align*}
By the definition of $\mI_{k_1}$ it follows that
\begin{multline*}
\,{}^{\rort^{\otimes (i-1)}}\!\!\left(p_1^*c_{0i}\bu \expinv{\left(p_2/evb_i^{\b_1}\right)}\mJ^2\bu\mS^{\b_2}\bu p_1^*c_{i0}\right)\!{}^{\rort^{\otimes (k_1-i)}}\bu p_1^*\mI_{k_1}\\=\mI_{k_1}\bu {}^{E_{1}}\!\!\left( \expinv{\left(p_2/evb_i^{\b_1}\right)}\mJ^2\bu\mS^{\b_2}\right)\!{}^{E_{3}}.
\end{multline*}
By Example~\ref{trivial expinv is pullback}, Lemma~\ref{vertical pullback expinv composition} and Lemma~\ref{orientors tensor distributivity},
we see that
\[
{}^{E_1}\left(\expinv{\left(p_2/evb_i^{\b_1}\right)}\qor_2\right)^{E_3}={}^{E_{1}}\!\!\left( \expinv{\left(p_2/evb_i^{\b_1}\right)}\mJ^2\bu\mS^{\b_2}\right)\!{}^{E_{3}}\bu\,\,{}^{p_1^*E_1\!\!}\left(m_{k_2}\bu p_2^*\mI_{k_2}\right)^{p_1^*E_3}.
\]
We conclude that
\[
\expinv{\vartheta}(\pa \qor)=(-1)^s\qor_1\bu {}^{E_1}\left(\expinv{\left(p_2/evb_i^{\b_1}\right)}\qor_2\right)^{E_3}.
\]
\end{proof}

\subsection{Proof for \texorpdfstring{$k=-1$}{k=-1}}
\label{proof of q-relations k=-1}
This section consists of the proof of Theorem~\ref{boundary of q k=-1 introduction theorem}. We keep the notations from Section~\ref{proof of q-relations}, when applicable.
\begin{proof}[Proof of Theorem~\ref{boundary of q k=-1 introduction theorem}]
Recall Definition~\ref{q orientor definition}.
It follows from Lemma~\ref{boundary of composition orientor} and Example~\ref{trivial expinv is pullback} applied to the diagram
\[
\begin{tikzcd}
B\ar[r,"\iota"]\ar[d,equal]&\mM\ar[d,equal]\\
B\ar[r,"\iota"]&\mM\ar[r,"\pi^{\mM}"]&\W
\end{tikzcd}
\]
that
\[
\pa \qor_{-1,l}^{\b}=\left(\pa J_{0,l}^\b\right)^{\efield}\bu \left(\pi^{\pa \mM}\right)^*\left(\Psi^\efield_{\mu(\b)+1}\bu m_0\right).
\]
Similarly, it follows from Remark~\ref{pullback of composition of orientors by local diffeomorphism} and Example~\ref{trivial expinv is pullback} applied to the diagram
\[
\begin{tikzcd}
\mM_1\times_L\mM_2\ar[r,"\vartheta"]\ar[d,equal]&B\ar[d,equal]\\
\mM_1\times_L\mM_2\ar[r,"\vartheta"]&B\ar[r,"\pi^B"]&\W
\end{tikzcd}
\]
that
\begin{align}
\label{pullback of boundary of q initial equation k=-1}
\expinv{\vartheta}\pa \qor_{-1,l}^\b\overset{\text{Rmk. \ref{pullback of composition of orientors by local diffeomorphism}}}=&\left(\expinv{\vartheta}\left(\pa J_{0,l}^\b\right)\right)^{\efield}\bu \left(\pi^{\pa \mM_1\times_L\mM_2}\right)^*\left(\Psi^\efield_{\mu(\b)+1}\bu m_0\right)\\
\overset{\text{Thm. ~\ref{expinv of boundary of J by thetabeta theorem k=-1}}}=&-e^\efield\bu\left( \left(J_{1,I}^{\b_1}\bu\left(\expinv{\left(p_2/evb_0^{\b_1}\right)}J_{1,J}^{\b_2}\right)^{\lort^{\b_1}}\right)^{\lort}\bu \left(evb_0^{\b_1}\circ p_1\right)^*{i}_{\mu(\b)+1}\right)^{\eort}\\\notag &\bu \left(\pi^{\pa \mM_1\times_L\mM_2}\right)^*\left(\Psi^\efield_{\mu(\b)+1}\bu m_0\right)\\
\notag\overset{\text{Rmk. \ref{exchange Psi on rort with eort remark}}}=&-e^\efield\bu \left(J_{1,I}^{\b_1}\bu\left(\expinv{\left(p_2/evb_0^{\b_1}\right)}J_{1,J}^{\b_2}\right)^{\lort^{\b_1}}\right)^{\lort\otimes\eort}\bu\\&\bu \left(evb_0^{\b_1}\circ p_1\right)^*\left({}^{\lort^{\mu(\b)+1}} w\bu \Psi_{\mu(\b)+1}\bu m_0\right)\\
=\,\,\,\,&-e^\efield\bu \left(J_{1,I}^{\b_1}\bu\left(\expinv{\left(p_2/evb_0^{\b_1}\right)}J_{1,J}^{\b_2}\right)^{\lort^{\b_1}}\right)^{\lort\otimes\eort}\bu\\&\bu \left(evb_0^{\b_1}\circ p_1\right)^*\left({}^{\lort^{\mu(\b)+1}} w\bu {}^{\lort^{\b}}\Psi_1\bu {}^{\lort^{\b_2}}\Psi_{\b_1}\bu \Psi_{\b_2}\bu m_0\right)
\end{align}

Using Lemma~\ref{commutativity vertical pullback vs horizontal pullback orientors}, we get 
\begin{align}
    \label{boundary of q calculation}
    \expinv{\vartheta}\pa \qor_{-1,l}^\b=-e^\efield\bu {}^\lort w\bu \Psi_1\bu \left(J_{1,I}^{\b_1}\bu\left(\expinv{\left(p_2/evb_0^{\b_1}\right)}J_{1,J}^{\b_2}\right)^{\lort^{\b_1}}\right)^{\rort}\bu {}^{\lort^{\b_2}}\Psi_{\b_1}\bu\Psi_{\b_2}\bu m_0.
\end{align}
Equation~\eqref{J2 commutes with Psi1 equation} reads
\[
\left(\expinv{\left(p_2/evb_i^{\b_1}\right)}J_{1,J}^{\b_2}\right)^{\lort^{ \b_1}\otimes \rort}\bu{}^{\lort^{\b_2}}\!\!\left(p_1^*{evb_0^{\b_1}}^*\Psi_{\b_1}\right)={evb_0^{\b_1}}^*\Psi_{\b_1}\bu\left(\expinv{\left(p_2/evb_i^{\b_1}\right)}J_{1,J}^{\b_2}\right)^{ \rort}.
\]
Combining the preceding with equation~\eqref{boundary of q calculation}, and
recalling Definition~\ref{Otm definition} and Definition~\ref{The operator mJ}, we get
\[
\expinv{\vartheta}\left(\pa \qor_{-1,l}^\b\right)=-\Otm\bu \mJ_{1,I}^{\b_1}\bu \expinv{\left(p_2/evb_0^{\b_1}\right)}\mJ_{1,J}^{\b_2}\bu m_0.
\]
However,
observe that $\mS\bu m_0=m_0$, since $\A$ is of null-degree. Moreover, $\mI_0=\Id_\A$. Therefore, for all $\b\in \Pi, I\subset [l]$, 
\[
\qor_{0,I}^{\b}=\mJ^{\b}_{1,I}\bu m_0.
\]
It follows that
\begin{align*}
\mJ_{1,I}^{\b_1}\bu \expinv{\left(p_2/evb_0^{\b_1}\right)}\mJ_{1,J}^{\b_2}\bu m_0 &= m_2\bu \left(\mJ_{1,I}^{\b_1}\bu m_0\right)^{\rort}\bu \expinv{\left(p_2/evb_0^{\b_1}\right)}\mJ_{1,J}^{\b_2}\bu m_0 \\&= m_2\bu
\left(\qor_{0,I}^{\b_1}\right)^{\rort}\bu \expinv{\left(p_2/evb_0^{\b_1}\right)}\qor_{0,J}^{\b_2},
\end{align*}
which gives the result.
\end{proof}

\subsection{Cyclic symmetry}\label{proof of the cyclic equation section}
Recall the cyclic isomorphism \[f^\b:\mdl{3}\to\mdl{3},\] that is the map given by 
\[
f^\b(\S,u,(z_0,...,z_k),\vec w)=(\S,u,(z_1,...z_k,z_0),\vec w).
\]
Set \begin{align*}
    ev^{cyc}:=&(evb_1,...,evb_k,evb_0):\mM_{k+1,l}(\b)\to L^{\times_\W k+1},\\
    E^{cyc}:=&\left(ev^{cyc}\right)^*\rort_L^{\boxtimes k+1}.
\end{align*}
Let
\begin{align*}
\tau:{\rort_L}^{\boxtimes k+1}\to&{\rort_L}^{\boxtimes k+1},\\ a_0\otimes\cdots\otimes a_{k}\mapsto& (-1)^{|a_{0}|\cdot\sum_{j=1}^k|a_j|}a_1\otimes\cdots\otimes a_{k}\otimes a_0
\end{align*}
denote the graded symmetry isomorphism. Recall the following definition from Section~\ref{orientors for moduli introduction section}.
\begin{definition}
Let $T$ be the \orientor{f^\b} of $E^{cyc}$ to $E^{cyc}$ given by
\[
T:=\phi_{f^\b}^{E^{cyc}}\bu \left(ev^{cyc}\right)^*\t.
\]
\end{definition}

We now turn to prove Theorem~\ref{cyclic theorem introduction}.
For simplicity, we write $f:=f^\b$ as this creates no ambiguity.
The proof of Theorem~\ref{cyclic theorem introduction} depends on the following lemmas and is given below their statements.
Denote by $\rort',\rort''$ the sub-local systems of $\underset{j=1}{\overset{k+1}{\boxtimes}}{\rort_L}$ of total degree with parity  $\mu(\b)+1,\mu(\b)$, respectively.
In particular, 
\[
\underset{j=1}{\overset{k+1}\boxtimes}\rort_L=\rort'\oplus\rort''.
\]
\begin{lemma}
\label{rotation of mI from 0 to 1}
The following equation of \orientor{\Id_{\mM_{k+1,l}(\b)}}s of ${ev^{cyc}}^*\rort'$ to $evb_1^*\rort_L$ holds.
\[
 f ^*\left(m_2\bu\left( \left(\mS^\b\bu m_k\bu \mI\right)\otimes \Id\right)\right)\bu \left(ev^{cyc}\right)^*\tau=c_{10}\bu m_{2}\bu\left(\left(\mS^\b\bu m_k\bu\mI\right)\otimes \Id\right)
\]
\end{lemma}

\begin{lemma}
\label{rotation of mJ from 0 to 1}
The following equation of \orientor{\pi^{\mM_{k+1,l}(\b)}}s of $\rort_L$ to $\efield$ holds.
\[\expinv{f}\mJ_{k+1,l}^\b=(-1)^k
\mJ_{k+1,l}^\b\bu c_{01}
\]
\end{lemma}
The proof of Lemmas~\ref{rotation of mI from 0 to 1} and~\ref{rotation of mJ from 0 to 1} is given below, after the proof of Theorem~\ref{cyclic theorem introduction}. 

\begin{proof}[Proof of Theorem \ref{cyclic theorem introduction}]
Both sides vanish on $\rort''$ by Definition~\ref{Otm definition}. Therefore, we may assume that our input is in $\rort'$. In this case, Lemma \ref{rotation of mI from 0 to 1} applies.
We calculate
\begin{align*}
\notag\expinv{ f }&\left(\Otm_{odd}\bu m_2\bu\left(\qor_{k,l}^\b\otimes \Id\right)\right)=\\\overset{\text{Def. \ref{q orientor definition}}}=&\expinv{f}\left(\Otm_{odd}\bu m_2\bu \left(\left(\mJ_{k+1,I}^\b\bu evb_0^*\left(\mS^\b\bu m_k\right)\bu \mI_{k}\right)\otimes \Id\right)\right)
\\\notag
\overset{\text{Lem. \ref{commutativity vertical pullback vs horizontal pullback orientors}}}=&\expinv{f}\left(\Otm_{odd}\bu \mJ_{k+1,I}^\b\bu m_2\bu \left(\left( evb_0^*\left(\mS^\b\bu m_k\right)\bu \mI_{k}\right)\otimes \Id\right)\right)
\\\notag
\overset{\begin{smallmatrix}\text{Rmk. \ref{pullback of composition of orientors by local diffeomorphism}}\\\text{Ex. \ref{trivial expinv is pullback}}\\\text{Ex. \ref{expinv by diffeomorphism over Id}}\end{smallmatrix}}=&
\Otm_{odd}\bu\expinv{f}\mJ_{k+1,I}^\b\bu  f^* \left(m_2\bu\left(\left(evb_0^*\left(\mS^\b\bu m_k\right)\bu \mI_{k}\right)\otimes \Id\right)\right).
\end{align*}
Therefore, on the one hand
\begin{align*}\notag\expinv{ f }\left(\Otm_{odd}\bu m_2\bu\left(\qor_{k,l}^\b\otimes \Id\right)\right)&\bu \left(ev^{cyc}\right)^*\t=\\
\overset{\begin{smallmatrix}\text{Lem. \ref{rotation of mI from 0 to 1}}\\\text{Lem. \ref{rotation of mJ from 0 to 1}}\end{smallmatrix}}=&
(-1)^k\Otm_{odd}\bu \mJ_{k+1,l}^\b\bu m_2\bu\left(\left(\mS^\b\bu m_k\bu \mI\right)\otimes \Id\right)
\\\overset{\begin{smallmatrix}\text{Lem. \ref{commutativity vertical pullback vs horizontal pullback orientors}}\end{smallmatrix}}=&(-1)^k\Otm_{odd} \bu m_2\bu \left(\qor_{k,l}^\b\otimes \Id\right).
\end{align*}
On the other hand,
\begin{align*}
\expinv{f}\left(\Otm_{odd}\bu m_2\bu\left(\qor_{k,l}^\b\otimes \Id\right)\right)\bu &\left(ev^{cyc}\right)^*\t=\\\overset{\text{Def~\ref{pullback by diffeomorphism}}}=& \Otm_{odd}\bu m_2\bu\left(\qor_{k,l}^\b\otimes \Id\right)\bu \phi_f^{E^{cyc}}\bu \left(ev^{cyc}\right)^*\t \\
=&\Otm_{odd}\bu m_2\bu\left(\qor_{k,l}^\b\otimes \Id\right)\bu T.
\end{align*}

\end{proof}
\begin{proof}[Proof of Lemma \ref{rotation of mI from 0 to 1}]
We check this equation on a local section
\[
X=\bigotimes_{j=1}^{k+1}evb_j^*\mO^{r_j},
\]
where $\mO$ is a local section of $\lort$, and $r_j\in \Z$ are such that $\sum_{j=1}^{k+1}r_j=_2\mu(\b)+1$.
Observe that
\[
(m_{k+1}\bu\tau^{-1})(X)=(-1)^{r_{k+1}\left(\sum_{j=1}^kr_j\right)}m_{k+1}(X).
\]
However, 
\[
r_{k+1}\left(\sum_{j=1}^kr_j\right)=_2 r_{k+1}\left(r_{k+1}+\mu(\b)+1\right)=_2r_{k+1}\mu(\b).
\] Therefore, restricted to $\rort'$
\begin{equation}
    m_{k+1}\bu\tau^{-1}=m_{k+1}\bu\left(\Id^{\otimes k}\otimes \mS^\b\right).
\label{tau embarks mS^b}
\end{equation}
Recall that 
\begin{equation}
c_{10}\circ \mS^\b=c_{1,k+1}.\label{rotation from 0 to 1 shtifel whitney}
\end{equation}
We calculate
\begin{align*}
\notag f ^*&\left(m_2\bu\left( \left(\mS^\b\bu m_k\bu \mI\right)\otimes \Id\right)\right)\\&=\mS^\b\bu m_{k+1}\bu\left( f ^*\mI\otimes \mS^\b\right)\\\notag
&\overset{\text{Koszul \ref{Koszul signs}}}=\mS^\b\bu m_{k+1}\bu (\Id^{\otimes k}\otimes \mS^\b)\bu\left( f ^*\mI\otimes \Id\right)\\\notag
&\overset{\text{eq. \eqref{tau embarks mS^b}}}=\mS^\b\bu m_{k+1}\bu \tau^{-1}\bu \left( f ^*\mI\otimes \Id\right)
\\\notag
&\overset{\text{Koszul \ref{Koszul signs}}}=\mS^\b\bu m_{k+1}\bu\left(\Id\otimes  f ^*\mI\right)\bu \left(ev^{cyc}\right)^*\tau^{-1}\\\notag
&\overset{\text{Def. \ref{mI definition}}}=\mS^\b\bu m_{k+1}\bu\left(c_{11}\otimes c_{12}\otimes\cdots\otimes c_{1,k+1}\right)\bu\left(ev^{cyc}\right)^*\tau^{-1}\\\notag
&\overset{\text{eq. \eqref{rotation from 0 to 1 shtifel whitney}}}=\mS^\b\bu m_{k+1}\bu c_{10}^{\otimes k+1}\bu\left(c_{01}\otimes c_{02}\otimes\cdots\otimes c_{0k}\otimes \mS^\b\right)\bu \left(ev^{cyc}\right)^*\tau^{-1}\\\notag
&=c_{10}\bu\mS^\b\bu m_{k+1}\bu\left(\mI\otimes \mS^\b\right)\bu\left(ev^{cyc}\right)^*\tau^{-1}\\
&=c_{10}\bu m_{2}\bu\left(\left(\mS^\b\bu m_k\bu\mI\right)\otimes \Id\right)\bu\left(ev^{cyc}\right)^*\tau^{-1}.
\end{align*}
\end{proof}
\begin{proof}[Proof of Lemma \ref{rotation of mJ from 0 to 1}]
Recalling Definition~\ref{The operator mJ}, the lemma follows immediately from Lemma~\ref{properties of J lemma}.
\end{proof}

\subsection{Base change}
\label{proof of q-naturality families}

Let $\xi:\W'\to \W$ be a map.
Let $\target=\left(\W,X,\w,L,\pi^X,\fp,\underline\Upsilon,J\right)$ be a target over $\W$. Recall Section~\ref{families target definition section} regarding families and Section~\ref{naturality of J families section} regarding naturality. Let
\[\pull\xi_\rort:\left(\xi^{L}\right)^*\rort_L\to \rort_{\xi^*L}\]
be the map given by $\pull{\left(\xi^{L}/\xi\right)}:{\xi^{L}}^*\lort_L\to \lort_{\xi^*L}$ extended multiplicatively to ${\xi^{L}}^*\rort_L$. Set
\[\pull\xi_{\efield}:=\pi^L_*\pull\xi_{\rort}:\xi^*\efield_L\to \efield_{\xi^*L}.\]
We think of $\pull\xi_\rort$ and $\pull\xi_\efield$ as \orientor{\Id_{\xi^*L} \text{ and } \Id_{\W'}}s, respectively.
\begin{remark}\label{naturality of m families}
$\pull\xi_\rort,\pull\xi_\efield$ are algebra homomorphisms with respect to the corresponding tensor multiplication maps $m_L,m_{\xi^*L}.$ Similarly, they commute with $\Psi_b$ for $b\in \Z$.
\end{remark}
\begin{proposition}\label{naturality of Otm families}
With the above notations, the following diagram is commutative.
\[
\begin{tikzcd}
\xi^*\rort_L\ar[rr,"\pull\xi_{\rort}"]\ar[d,swap,"\expinv\xi \Otm^L"]&&\rort_{\xi^*L}\ar[d,"\Otm^{\xi^*L}"]\\
\cort{\pi^{\xi^*L}}\otimes \xi^*\eort_{L}\ar[rr,swap,"1\otimes (\pi^L)^*\pull\xi_{\efield}"]&&\cort{\xi^*L}\otimes \eort_{\xi^*L}
\end{tikzcd}
\]
That is, the following equation of \orientor{\pi^{\xi^*L}}s holds.
\[
\pull\xi_{\efield}\bu \expinv\xi \Otm^L=\Otm^{\xi^*L}\bu \pull\xi_{\rort}.
\]
\end{proposition}
\begin{proof}
This is immediate by the definitions and equation~\eqref{naturality of e families equation}.
\end{proof}

Let $k\geq0,l\geq 0$ and $\b\in \Pi(\target)$. Abbreviate \[\mM:=\mM_{k+1,l}(\target;\b),\qquad \mM':=\mM_{0,l}(\xi^*\target;\xi^*\b).\] Let 
\[
\pull{\xi}_E:=\underset{j=1}{\overset{k}\boxtimes} {\left(evb_j^{\b}\right)}^*\pull\xi_\rort:{\xi^{\mM}}^*E^k_L\to E^k_{\xi^*L}.
\]
We think of $\pull\xi_E$ as an \orientor{ \Id_{\mM_{k+1,l}(\xi^*\target;\xi^*\b)}}.

\begin{theorem}\label{naturality of Q families}
The following diagram is commutative.
\[
\begin{tikzcd}
{\xi^{\mM}}^*E^k_L\ar[rr,"\pull\xi_E"]\ar[d,swap,"\expinv{\xi}\qor^{\left(\target;\b\right)}_{k+1,l}"]&&E^k_{\xi^*L}\ar[d,"\qor^{\left(\xi^*\target;\xi^*\b\right)}_{k+1,l}"]\\
\cort{evb_0^{\left(\xi^*\target\right)}}\otimes {\xi^{\mM}}^*\left(evb_0^{\target}\right)^*\rort_L\ar[rr,swap,"1\otimes evb_0^*\left(\pull{\xi}_\rort\right)"]&&\cort{evb_0^{\left(\xi^*\target\right)}}\otimes\left(evb_0^{\left(\xi^*\target\right)}\right)^*\rort_{\xi^*L}
\end{tikzcd}
\]
That is, the following equation of orientors holds.
\[
\pull\xi_\rort\bu\expinv{\xi}\qor^{\left(\target;\b\right)}_{k+1,l}=\qor^{\left(\xi^*\target;\xi^*\b\right)}_{k+1,l}\bu \pull\xi_E.
\]
Similarly, for $\b\in \Pi^{ad}(\target)$, the following diagram is commutative.
\[
\begin{tikzcd}
\underline\A\ar[drr,"Q_{-1}^{\left(\xi^*\target;\xi^*\b\right)}"]\ar[d,swap,"\expinv{\xi}\qor^{\left(\target;\b\right)}_{-1,l}"]\\\cort{\pi^{\mM'}}\otimes {\pi^{\mM'}}^*\xi^*\efield_L\ar[rr,swap,"1\otimes {\pi^{\mM'}}^*\pull\xi_{\efield}"]
&&\cort{\pi^{\mM'}}\otimes {\pi^{\mM'}}^*\efield_{\xi^*L}
\end{tikzcd}
\]That is, the following equation of orientors holds.
\[
\pull\xi_{\efield}\bu \expinv\xi Q^{\left(\target;\b\right)}_{-1,l}=Q^{\left(\xi^*\target;\xi^*\b\right)}_{-1,l}
\]
\end{theorem}
\begin{proof}
By Lemma~\ref{naturality of cij families lemma} it follows that the following diagram is commutative.
\[
\begin{tikzcd}
\left(\xi^\mM\right)^*{ev^\target}^*\boxtimes \rort_L\ar[rr,"{\xi^\mM}^*\mI^{\target}_{k+1,l}"]\ar[d,swap,"\left(ev^{\xi^*\target}\right)^*\boxtimes \pull\xi_{\rort}"]&&{\xi^\mM}^*evb_0^*\rort^{\otimes k}_L\ar[d,"evb_0^*\otimes \pull\xi_\rort"]\\
\left(ev^{\xi^*\target}\right)^*\boxtimes \rort_{\xi^*L}\ar[rr,swap,"\mI^{\xi^*\target}"]&&evb_0^*\rort_{\xi^*L}
\end{tikzcd}
\]
The theorem follows immediately by Lemma~\ref{naturality of J families} and remark~\ref{naturality of m families}, keeping in mind that $\mu(\b)=\mu(\xi^*\b)$.
\end{proof}


\bibliography{main.bbl}
\end{document}